\documentclass[sn-mathphys]{sn-jnl}

\usepackage{epstopdf} 

\jyear{2023}%

\usepackage{pifont}
\usepackage{amssymb}
\usepackage{colortbl}
\usepackage{subfigure}
\usepackage{hyperref}
\usepackage[algo2e]{algorithm2e} 
\usepackage{appendix}
 \usepackage{geometry}
 \usepackage{makecell}
\usepackage{tabularx}
\usepackage{adjustbox}
\theoremstyle{thmstyleone}%
\newtheorem{example}{Example}%
\newtheorem{proposition}{Proposition}
\newtheorem{theorem}{Theorem}
\newtheorem{definition}{Definition}

\newtheorem{assumption}{Assumption}

\newcommand{\be}{\begin{equation}}
\newcommand{\ee}{\end{equation}}
\renewcommand{\proofname}{\textbf{Proof}}

\definecolor{LightCyan}{rgb}{0.88,1,1}
\DeclareUnicodeCharacter{2113}{$l$}
\makeatletter
\renewcommand{\@thesubfigure}{\hskip\subfiglabelskip}
\newcolumntype{C}{p{1cm}}
\makeatother
\allowdisplaybreaks
\SetKwComment{Comment}{/* }{ */}

\usepackage{natbib}
\usepackage{url}
\title[Bilevel hyperparameter optimization for SVC]{{\huge\textmd{Global relaxation-based LP--Newton method for multiple hyperparameter selection in support vector classification with feature selection}}}
\begin{document}

\author[1]{\fnm{Yaru} \sur{Qian}}\email{y.qian@soton.ac.uk}

\author[1]{\fnm{Qingna} \sur{Li}}\email{qnl@bit.edu.cn}
\equalcont{These authors contributed equally to this work.}

\author*[2]{\fnm{Alain} \sur{Zemkoho}}\email{a.b.zemkoho@soton.ac.uk}
\equalcont{These authors contributed equally to this work.}



\affil[1]{\orgdiv{School of Mathematics and Statistics, Beijing Key Laboratory on MCAACI, and Key Laboratory of Mathematical Theory and Computation in Information Security}, \orgname{Beijing Institute of Technology}, \orgaddress{\city{Beijing}, \postcode{100081}, \country{PR China}}}
\affil*[2]{\orgdiv{School of Mathematical Sciences}, \orgname{University of Southampton}, \orgaddress{\city{Southampton}, \postcode{SO17 1BJ}, \country{UK}}}


\abstract{Support vector classification (SVC) is an effective tool for classification tasks in machine learning. Its performance relies on the selection of appropriate hyperparameters. This paper focuses on optimizing the regularization hyperparameter $C$ and determining feature bounds for feature selection within SVC, leading to a potentially large hyperparameter space. It is very well-known in machine learning that this could lead to the so-called {\em curse of dimensionality}. To address this challenge of multiple hyperparameter selection, the problem is formulated as a bilevel optimization problem, which is then transformed into a mathematical program with  equilibrium constraints (MPEC). Our primary contributions are two-fold. First, we establish the satisfaction of the MPEC-MFCQ for our problem reformulation. Furthermore, we introduce a novel global relaxation based linear programming (LP)-Mewton method (GRLPN) for solving this problem and provide corresponding convergence results. Typically, in global relaxation methods for MPECs, the algorithm for the corresponding subproblem is treated as a blackbox. Possibly for the first time in the literature, the subproblem is specifically studied in detail. Numerical experiments demonstrate GRLPN’s superiority in efficiency and accuracy over both grid search and traditional global relaxation methods solved using the well-known nonlinear programming solver, SNOPT.\\[-3ex]}



\maketitle
\section{Introduction}
  Support vector classification (SVC) is a widely used type of machine learning algorithms that is particularly effective for classification tasks. Classification tasks arise in a wide range of areas, including in economics, finance, and engineering. In engineering more specifically, applications of SVC have been discovered in chemical engineering \cite{de2014fault,susto2013predictive,ahmad2015application,liu2009obscure,heikamp2014support}, civil engineering \cite{deshpande2016performance,harirchian2020application,liu2016comprehensive,hadjidemetriou2018automated,jozdani2019comparing}, mechanics \cite{oskoei2008support,xiao2019leak,diao2021structural,chen2017remaining,laouti2011support}, and so on. In civil engineering, for example, SVC can be applied for the prediction of traffic flow and congestion in a road network \cite{deshpande2016performance}, seismic hazard assessment \cite{harirchian2020application}, soil quality classification \cite{liu2016comprehensive}, pavement condition evaluation \cite{hadjidemetriou2018automated}, as well as land use classification \cite{jozdani2019comparing}. Interested readers are referred to the survey \cite{wang2005support} for more applications of SVC in engineering and other fields. 
  
 In SVC, the process of selecting the best hyperparameters is a crucial topic, and has been extensively studied by numerous researchers from both theoretical and practical perspectives \cite{chapelle2002choosing,dong2007mpec,duan2003evaluation,keerthi2006efficient,kunapuli2008bilevel1,kunapuli2008bilevel,kunapuli2008classification}. However, focus has mostly been on grid search and gradient descent--based methods  \cite{chapelle2002choosing,duan2003evaluation,keerthi2006efficient}.

Recently, bilevel optimization--based techniques to compute optimal hyperparameters have also been gaining some attraction, given that they offer numerous advantages as is discussed in \cite{kunapuli2008classification}. However, solving bilevel programs numerically can be extremely challenging due to their inherent hierarchical structure. It is known that even when all the functions involved are linear, the computational complexity is already NP-hard \cite{ben1990computational}. Two standard approaches exist for solving bilevel programming problems, with the first one being the gradient--based methods, which are generally divided into two categories: iterative differentiation \cite{okuno2021lp,franceschi2017bridge,franceschi2018bilevel,shaban2019truncated} and implicit differentiation \cite{foo2007efficient,mackay2019self}, depending on how the gradient (w.r.t. hyperparameters) can be computed. 

In contrast, the second approach entails transforming the problem into a single-level problem, using various ways. One way is to use the lower-level problem to define an equivalent single--level problem \cite{ye2010new,ye1995optimality}. One can also replace the lower-level problem by its first order optimality condition to solve the resulting problem as a mathematical program with equilibrium constraints (MPEC) \cite{bennett2006model,kunapuli2008classification,kunapuli2008bilevel,conigliobilevel, wang2023fast}. Therefore, various algorithms for MPECs can be potentially applied to solve bilevel optimization problems. Next, we briefly summarize the algorithms for MPEC, which are closely related to the method used in our paper. Notice that the feasible set of MPECs exhibits a highly unique structure, which leads to the violation of the majority of conventional constraint qualifications, and as a result, specialized algorithms are typically employed to address MPECs effectively \cite{fletcher2006local,stein2012lifting,scholtes1999exact}.

A distinguished category of algorithms for MPECs is known as relaxation (or regularization) methods \cite{hoheisel2013theoretical}. The original global relaxation method for MPECs was developed by Scholtes \cite{scholtes2001convergence}, and subsequent advancements have introduced various relaxation schemes aiming to alleviate the challenges posed by complementarity constraints in MPECs through different approaches. Certain recent schemes have demonstrated superior theoretical properties compared to the original global relaxation method. Nonetheless, empirical evidence indicates that the global relaxation method continues to be one of the most efficient and dependable approaches \cite{hoheisel2013theoretical}. 

Some studies have utilized the global relaxation approach for selecting hyperparameters in SVC models \cite{li2022bilevel,li2022unified}. These papers specifically apply bilevel optimization to determine the optimal value of the regularization hyperparameter $C$ for special classes of $\ell_1$ and $\ell_2$--based SVC problems. Within these studies, the corresponding versions of the mathematical programming with equilibrium constraints--tailored Mangasarian-Fromovitz constraint qualification (MPEC-MFCQ) is established, and subsequently, the resulting relaxation subproblem is solved with the SNOPT solver.

In a similar context, a model in \cite{kunapuli2008classification} was proposed for simultaneously selecting the hyperparameter $C$ and determining feature bounds for feature selection. Their approach involved using the FILTER nonlinear programming solver to address the MPEC transformation of the problem. In this paper, we consider the same model as in \cite{kunapuli2008classification}, which is to choose multiple hyperparameters in an $\ell_1$--based SVC model with feature selection that we denote as (\ref{OP}). We subsequently transform it into an MPEC. Different from the approach in \cite{kunapuli2008classification}, we would like to ask the following questions: (a) Is there any theoretical property that the resulting MPEC enjoys? (b) Can we design an efficient numerical algorithm to solve the MPEC? If so, what is the convergence result?  These questions serve as the primary impetus driving the work presented in this paper. 

Based on these research questions, the main contributions of this work are as follows. Firstly, for the first time, we prove the fulfilment of MPEC--MFCQ for the resulting MPEC. Secondly,  in conjunction with the global relaxation approach, we introduce and analyze a LP Newton--based method for the associated subproblem. Thirdly, in terms of the numerical performance, our algorithm outperforms grid search — a pivotal technique for hyperparameter selection in machine learning — as well as the global relaxation method solved by the SNOPT solver.

For the remainder of the paper, note that Section \ref{sec1} introduces the problem (\ref{OP}), which is the main subject of our study. Subsequently, the problem is transformed into a tractable single--level optimization problem; i.e. an MPEC. We also establish the fulfillment of the MPEC-MFCQ, a condition crucial for ensuring the convergence of multiple classes of MPEC--based algorithm, including the relaxation scheme studied here.  In section \ref{sec3}, we present the linear programming (LP)-Newton-based global relaxation method (GRLPN) to solve our resulting MPEC. Numerical experiments demonstrating the efficiency of the proposed method are presented  in section \ref{sec4}. 
To focus our attention on the main ideas, most of the proofs are moved to the appendices.

{\bf Notations.} For $x \in  \mathbb{R}^{n}$, $\|x \|_{0}$ denotes the number of nonzero elements in $x$, while $\| x \|_{1}$ and $\| x \|_{2}$ correspond to the $l_{1}$-norm and $l_{2}$-norm of $x$, respectively. Also, we will use $x_{+}=((x_{1})_{+},\ \cdots,\ (x_{n})_{+}) \in \mathbb{R}^{n}, $ where $(x_{i})_{+}=\max(x_{i},\ 0).$ In matrices, the symbol ';' serves to indicate a line break. For any positive integer $n$, the notation $[n]$ refers to the set $\{1,..,n\}$. $ \mid \! M \! \mid $ stands for the number of elements in the set $M \subset \mathbb{R}^n$. Let \( i \mod K \) represent the remainder of \( i \) when divided by \( K \). We denote $\mathbf{1}_{k}$  with elements all ones in $\mathbb{R}^{k}$. $I_{k}$ is the identity matrix in $\mathbb{R}^{k \times k}$, while $e^{k}_{\gamma}$ is the $\gamma$-th row vector of an identity matrix in $\mathbb{R}^{k \times k}$. The notation $\mathbf{0}_{k \times q}$ represents a zero matrix in $\mathbb{R}^{k \times q}$ and $\mathbf{0}_{k}$ stands for a zero vector in $\mathbb{R}^{k}$. On the other hand, $\mathbf{0}_{(\tau,\ \kappa)}$ will be used for a submatrix of the zero matrix, where $\tau$ is the index set of the rows and $\kappa$ is the index set of the columns.  Similarly to the case of zero matrix, $I_{(\tau,\ \tau)}$ corresponds to a submatrix of an identity matrix indexed by both rows and columns in the set $\tau$. Finally, $\Theta_{(\tau,\ \cdot)}$ represents a submatrix of the matrix $\Theta$, where $\tau$ is the index set of the rows, and $x_{\tau}$ is a subvector of the vector $x$ corresponding to the index set $\tau$.

\section{Bilevel hyperparamter optimization for SVC}\label{sec1}
In this section, we provide a succinct introduction to the cross-validation technique, which serves as the foundation for introducing the bilevel optimization problem (\ref{OP}) for the purpose of selecting optimal hyperparameters for SVC. Subsequently, we proceed to transform (\ref{OP}) into a single-level problem through the utilization of Karush-Kuhn-Tucker (KKT) reformulation and analyze a hidden property of the resulting problem. Throughout this section, we adopt the notation and follow the formulation as presented in \cite{li2022bilevel}.

In $K$-fold cross-validation, the dataset is divided into two distinct subsets: a cross-validation subset $\Omega=\{(x_{i},y_{i})\}_{i=1}^{l_{1}} \in \mathbb{R}^{n+1}$ with $l_1$ data points, and a hold-out test set $\Theta$ with $l_2$ data points. Here, $x_{i} \in \mathbb{R}^{n}$ denotes a data point and $y_{i}\in \{\pm 1\}$ represents its corresponding label. $\Omega$ is partitioned into $K$ disjoint validation sets, denoted as $\Omega_t$. During the $t$-th iteration ($t = 1, \cdots, K$), the training set $\overline{\Omega}_{t}=\Omega \backslash \Omega_{t}$ is used for training purpose with the size of $m_2$; while the validation set $\Omega_t$ with the size of $m_1$ is employed for evaluating the performance. The main purpose of the cross validation technique is to optimize the average accuracy achieved on the validation sets. In the following context, we denote the corresponding index sets for the validation set $\Omega_t$ and training set $\overline{\Omega}_t$ as $\mathcal{N}_t$ and $\overline{\mathcal{N}}_t$, respectively.

Given these procedures, it is crucial to establish specific assumptions about the distribution and characteristics of the dataset, which guide our analysis and modeling efforts. As discussed in \cite{conigliobilevel}, we make the following foundational assumptions for each $K$-fold lower-level problem:
\begin{itemize}
    \item For each lower-level problem $t\in\left[K\right]$, there exists at least one training data with label 1, which is not a support vector, and there also exists at least one training data with label -1, which is not a support vector.
\item For each lower-level problem $t\in\left[K\right]$, there exists at least two training data with label 1, which are support vectors, and there also exists at least two training data with label -1, which are support vectors.
\end{itemize}

Building on the established methodology and dataset assumptions, we proceed to introduce the multiple hyperparameter selection for SVC with feature selection (denoted as MHS-SVCFS), which is studied in \cite{kunapuli2008classification} to choose hyperparameter $C\in\mathbb{R}$ and $\overline{w}\in\mathbb{R}^n$ (given $C_{lb},\ C_{ub}\in\mathbb{R}_+$ and $\overline{w}_{lb},\ \overline{w}_{ub}\in\mathbb{R}^n_+$):

\begin{equation*}\tag{MHS-SVCFS}\label{OP}
    \begin{array}{cl}
        \min\limits_{\substack{w^t, \overline{w} \in \mathbb{R}^n \\ C \in \mathbb{R}}} & \frac{1}{K}\sum\limits_{t=1}^{K}\frac{1}{m_1}\sum\limits_{i\in \mathcal{N}_t}\left\Vert \left(-y_i\cdot x_i^{\top} w^t\right)_+\right\Vert_0 \\
         \textrm{s.t.}&  C_{lb}\leq C \leq \ C_{ub},\ \overline{w}_{lb}\leq \overline{w}\leq \overline{w}_{ub},\\
         {}& \min\limits_{\substack{-\overline{w} \leq w^t \leq \overline{w} \\ b^t \in \mathbb{R}}}\left\{\frac{1}{2}\Vert w^t\Vert_2^2+C\sum\limits_{i\in \overline{\mathcal{N}}_t}\max\left(1-y_i\left(x_i^{\top} w^t\right),0\right)\right\} \mbox{ for } t=1,\cdots,K.
    \end{array}
\end{equation*}
To enhance notational clarity, we adopt an alternative formulation for the data points within the support vector machine framework by augmenting each original data point \(x_i\) with an additional component of 1. This modification facilitates the incorporation of the bias term into the vector representation.

The focus of (\ref{OP}) lies on selecting appropriate hyperparameters, which can be further divided into two components: the regularization parameter $C\in\mathbb{R}_+$ and the feature selection parameter $\overline{w}\in\mathbb{R}^n_+$.


In the lower-level problem, we solve a soft margin SVC problem, incorporating hyperparameters $C$ and $\overline{w}$ alongside bound constraints. The objective function in each case comprises two crucial elements: a squared norm term that maximizes the margin of the SVC model, and a penalized hinge loss term that accounts for misclassification errors, thereby increasing the model's robustness. Additionally, each lower-level problem is formulated as a quadratic programming problem with box constraints applied to each feature. This setup differs from traditional $\ell_1$-SVC in that it allows for the exclusion of non-contributory features; if some components in $\overline{w}$ are small or even zero, the corresponding features are considered redundant or irrelevant. The function $\Vert(\cdot)_{+}\Vert_{0}$ in the upper-level objective function, which is discontinuous and nonconvex, measures the quantification of misclassification errors across $K$ validation sets, each contribution to this measure corresponding to a misclassified data point under the current model parameters. This evaluation is commonly referred to as the cross-validation error (CV error) in classification contexts, which computes the average number of misclassified data points per model configuration.

 Following the techniques in \cite{kunapuli2008classification}, the lower-level problem in (\ref{OP}) can be represented as an MPEC problem by conducting the following three steps. 

\textbf{\em Step one.} Replace the lower-level problems by the KKT conditions. This involves introducing slack variables $\xi^{t} \in \mathbb{R}^{m_{2}}$, converting them into convex quadratic optimization problems. Specifically, we define the t-th lower-level problem $(t=1,\cdots,K)$ as:

\begin{equation}
    \begin{array}{cl}\label{3}
         \min\limits_{w^t\in\mathbb{R}^n,\xi^t\in\mathbb{R}^{m_2}}&\frac{1}{2}\Vert w^t\Vert_2^2+C\sum\limits_{i=1}^{m_2}\xi_i^t  \\
         \hbox{s.t.}&  y_i\left(x_i^{\top} w^t\right)\geq 1- \xi^t_i,\ \xi^t_i\geq 0,\ i=1,\cdots, m_2,\\
         {}&  -\overline{w}\leq w^t\leq\overline{w} .
    \end{array}
\end{equation}

Let $\alpha^t\in\mathbb{R}^{m_2},\ \mu^t\in\mathbb{R}^{m_2},\ \beta^t\in\mathbb{R}^{n}$ and $\gamma^t\in\mathbb{R}^{n}$ represent the Lagrangian multipliers associated with (\ref{3}). The corresponding  KKT condition of (\ref{3}) can be expressed as:
\begin{align}   
     &0\leq \alpha^t_i\perp y_ix_i^{\top} w^t-1+\xi^t_i \geq0,\ i=1,\cdots, m_2,\label{4a}\\
    &0\leq \xi^t\perp  \mu^t\geq0,\label{4b}\\
    & 0\leq \beta^t \perp w^t+\overline{w} \geq0,\label{4c}\\
     &0\leq \gamma^t\perp -w^t+\overline{w} \geq0,\label{4d}\\
      & w^t=\sum\limits_{k=1}^{m_2}y_k \alpha^t_k x_k+\beta^t-\gamma^t,\label{4e}\\
     & \mu^t=\mathbf{1}C-\alpha^t.\label{4f}
\end{align}

For $t=1, \cdots, K$, we could eliminate $w^t$ and $\mu^t$, which leads to the following:
\begin{align}   
     &0\leq \alpha^t_i\perp y_ix_i^{\top} \left(\sum\limits_{k=1}^{m_2}y_k \alpha^t_k x_k+\beta^t-\gamma^t\right)-1+\xi^t_i \geq0,\ i=1,\cdots, m_2, \label{4_1a}\\
    &0\leq \xi^t\perp  \mathbf{1}C-\alpha^t\geq0,\label{4_2b}\\
    & 0\leq \beta^t \perp \left(\sum\limits_{k=1}^{m_2}y_k \alpha^t_k x_k+\beta^t-\gamma^t\right)+\overline{w} \geq0,\label{4_3c}\\
     &0\leq \gamma^t\perp -\left(\sum\limits_{k=1}^{m_2}y_k \alpha^t_k x_k+\beta^t-\gamma^t\right)+\overline{w} \geq0.\label{4_4d}
\end{align}

\textbf{\em Step two.} Reformulate the upper-level objective function in (\ref{OP}) by an equivalent continuous reformulation. 
Notably, the upper-level objective function is formulated as the form of $\Vert r_+\Vert_0$ with $r\in\mathbb{R}^{m_1}$. Furthermore, 
in \cite{kunapuli2008classification}, it has been shown that for non-zero components of $r$, i.e. $r_i\neq 0,\ i=1,\cdots m_1$, an equivalent transformation can be applied as below:

\begin{equation}\label{UP}
    \Vert r_+\Vert_0\ =\ \left\{\sum_{i=1}^{m_1}\zeta_i\mid \zeta:=\arg\min\limits_u \left\{ -u^{\top}r\mid\mathbf{0}\leq u\leq \mathbf{1}\right\}\right\}.
\end{equation}

Analysis of each corresponding linear programming problem demonstrates:
\begin{equation}\label{linear_transformnation}
    \zeta_i=\begin{cases}
    1,\quad r_i>0,\\
    0,\quad r_i<0,\\
\end{cases}
\end{equation}
For the function $\Vert {\left(r_i\right)}_+\Vert_0$, the classification proceeds straightforwardly:
\begin{equation}
    \Vert {\left(r_i\right)}_+\Vert_0=\begin{cases}
    1,\quad r_i>0,\\
    0,\quad r_i< 0.
\end{cases}
\end{equation}
However, when $r_i = 0$, $\Vert {\left(r_i\right)}_+\Vert_0$ is evaluated as 0, while the solution to the linear programming problem spans the set $[0,1]$. To ensure the robustness and consistency of \eqref{UP}, we assume throughout the paper that  the SVC model either correctly or incorrectly classifies the data points, i.e. for any validation data point $x_i$, $i=1, \ldots, m_1$, it holds that $x_i^{\top} w^t \neq 0$, where $w^t$ is a lower-level variable, for  $t=1,\cdots, K$.
  \begin{figure}[htbp]
\centering
\includegraphics[width=0.4\textwidth]{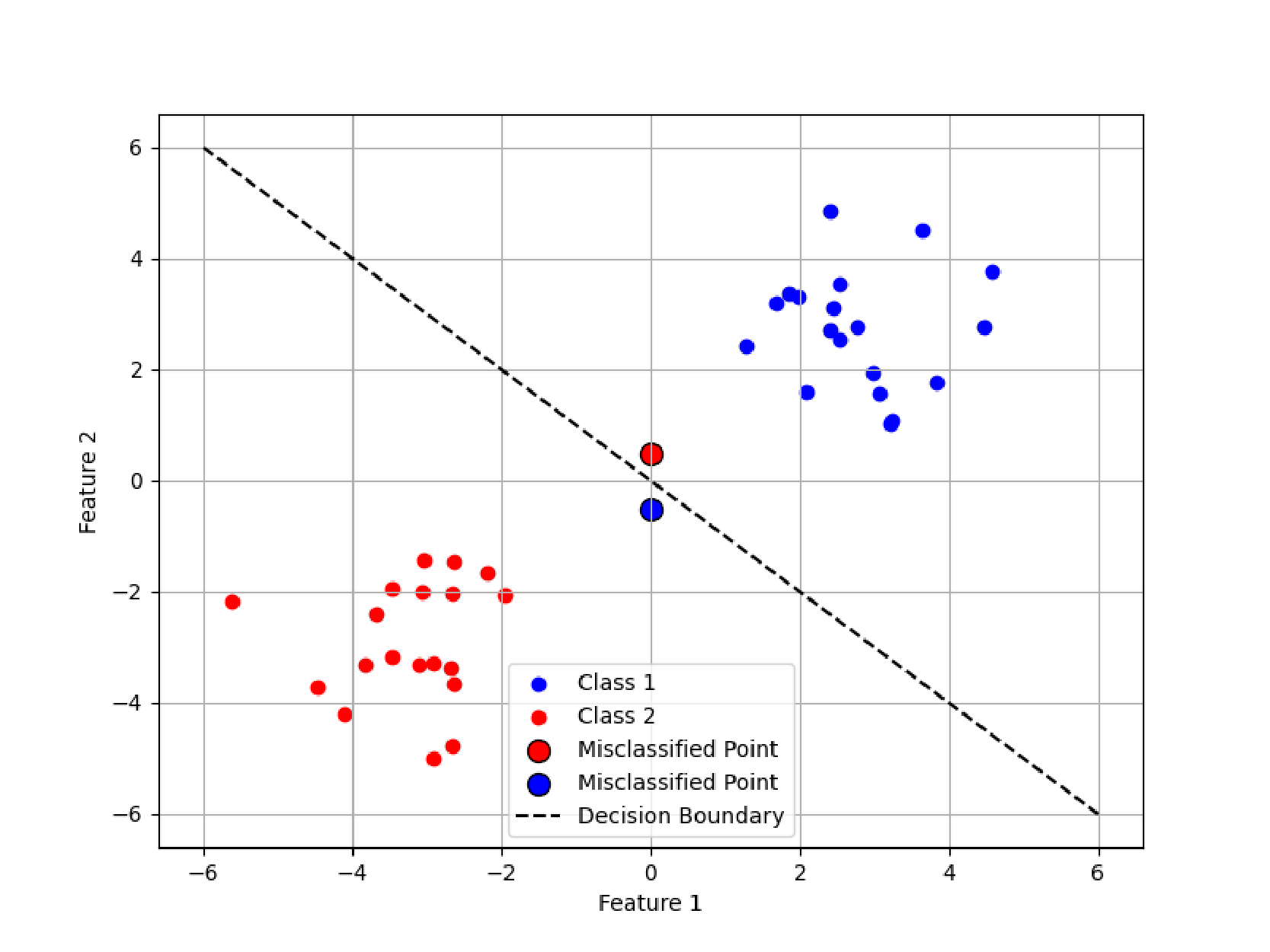}
\caption{Decision Boundary and Classification in a SVC}
\label{fig:svc_model_robustness}
\end{figure}

  In support of the assumption we made, Figure \ref{fig:svc_model_robustness} provides a visual representation of the classification model's behavior with a clear separation between the classes. This figure underscores the premise that all data points are either correctly or incorrectly classified, with none residing on the decision boundary. Under this assumption, the function $\sum\limits_{i\in\overline{\mathcal{N}}_t}\Vert \left(-y_i\cdot x_i^{\top} w^t\right)_+\Vert_0$ can be represented as the sum of all elements of the solution to a linear optimization problem as follows
\begin{equation}
    \begin{array}{cl}\label{LL}         \min\limits_{\zeta^t\in\mathbb{R}^{m_1}}&\sum\limits_{i=1}^{m_1}-\zeta^t_i\left(-y_ix_i^{\top}w^t\right)  \\
         \hbox{s.t.}& \zeta^t\geq\mathbf{0},\ \mathbf{1}-\zeta^t\geq\mathbf{0}.
    \end{array}
\end{equation} 
Subsequently, we replace (\ref{LL}) by its KKT condition:
\begin{align}
    & 0\leq \zeta^t\perp \lambda^t\geq 0,\label{9a}\\
    & 0\leq z^t\perp \mathbf{1}-\zeta^t\geq 0,\\
    & y_ix_i^{\top}w^t-\lambda^t_i+z^t_i=0,\ i=1,\cdots,m_1,\label{9c}
\end{align}
where $\lambda^t,z^t$ are the corresponding Lagrangian multipliers. Similarly to the elimination of $w^t,\ \mu^t$ in the lower-level problem, we proceed by removing $\lambda^t$, leading to the following complementarity condition:

\begin{align}
    & 0\leq \zeta^t_i\perp y_i x_i^{\top}w^t+z^t_i\geq 0 ,\ i=1,\cdots,m_1,\label{9_2a}\\
    & 0\leq z^t\perp \mathbf{1}-\zeta^t\geq 0.\label{9_2b}
\end{align}

\textbf{\textit{Step three}.} Combining the transformations of the upper-level objective function with the reformulated lower-level problems yield the following problem:
\begin{equation}\label{MPEC1}
    \begin{array}{cl}
         \min\limits_{\substack{C\in\mathbb{R},\ \overline{w}\in\mathbb{R}^{n},\\ \zeta^t\in\mathbb{R}^{m_1},\ z^t\in\mathbb{R}^{m_1},\\ \alpha^t\in\mathbb{R}^{m_2},\ \xi^t\in\mathbb{R}^{m_2},\\ \beta^t\in\mathbb{R}^{n},\ \gamma^t\in\mathbb{R}^{n},\\
         t=1,\cdots,K}}&\frac{1}{K m_1}\sum\limits_{t=1}^K\sum\limits_{i=1}^{m_1}\zeta_i^t  \\
         \hbox{s.t.}& C_{lb}\leq C \leq C_{ub},\ \overline{w}_{lb}\leq \overline{w}\leq \overline{w}_{ub},\\[1.5ex]
         {}& \text{and for}\ \  t=1,\cdots K:\\
         {}& 0\leq \zeta^t_i\perp y_i x_i^{\top}\left(\sum\limits_{k=1}^{m_2}y_k \alpha^t_k x_k+\beta^t-\gamma^t\right)+z^t_i\geq 0 ,\ i=1,\cdots,m_1,\\[1.5ex]
    {}& 0\leq z^t\perp \mathbf{1}-\zeta^t\geq 0,\\[2ex]
    {}&0\leq \alpha^t_i\perp y_i x_i^{\top} \left(\sum\limits_{k=1}^{m_2}y_k \alpha^t_k x_k+\beta^t-\gamma^t\right)-1+\xi^t_i \geq0,\ i=1,\cdots, m_2, \\[1.5ex]
    {}&0\leq \xi^t\perp  \mathbf{1}C-\alpha^t\geq0,\\[1.5ex]
    {}& 0\leq \beta^t \perp \left(\sum\limits_{k=1}^{m_2}y_k \alpha^t_k x_k+\beta^t-\gamma^t\right)+\overline{w} \geq0,\\[1.5ex]
     {}&0\leq \gamma^t\perp -\left(\sum\limits_{k=1}^{m_2}y_k \alpha^t_k x_k+\beta^t-\gamma^t\right)+\overline{w} \geq0.
    \end{array}
\end{equation}

Subsequently, the \(K\) upper-level and lower-level problems can be represented by their matrix form using the following notations. Define the vector $v = \left[C,\ \overline{w}^{\top},\ \zeta^{\top},\ z^{\top},\ \alpha^{\top},\ \xi^{\top}, \beta^{\top},\gamma^{\top}\right]^{\top}  \in  \mathbb{R}^{\overline{n}}+\overline{m}$, where 
\begin{equation}
    \overline{n}=1+n,\ \text{and}\quad \overline{m}= 2 K (m_{1}+m_{2}+n)+\overline{n}.
\end{equation}
We reach the subsequent problem
\begin{equation}\label{MPEC}
    \begin{array}{cl}
         \min\limits_{v\in\mathbb{R}^{\overline{n}+\overline{m}}}&\frac{1}{K m_1}\mathbf{1}^{\top}\zeta  \\
         \hbox{s.t.}& C_{lb}\leq C \leq C_{ub},\ \overline{w}_{lb}\leq \overline{w}\leq \overline{w}_{ub},\\
         {}&\mathbf{0}\leq \zeta\perp AB^{\top}\alpha+A\beta-A\gamma+z\geq \mathbf{0},\\
    {}& \mathbf{0}\leq z\perp \mathbf{1}-\zeta\geq \mathbf{0},\\
     {}&\mathbf{0}\leq \alpha\perp BB^{\top}\alpha+B\beta-B\gamma-\mathbf{1}+\xi \geq\mathbf{0},\\
     {}&\mathbf{0}\leq \xi\perp \mathbf{1}C-\alpha \geq\mathbf{0},\\
    {} &\mathbf{0}\leq \beta \perp B^{\top}\alpha+\beta-\gamma+E_K^n\overline{w} \geq\mathbf{0},\\
    {}& \mathbf{0}\leq \gamma\perp -B^{\top}\alpha-\beta+\gamma+E_K^n\overline{w} \geq\mathbf{0}
    \end{array}
\end{equation}
with $\zeta\in\mathbb{R}^{K m_1},\ z\in\mathbb{R}^{K m_1},\ \alpha\in\mathbb{R}^{K m_2},\ \xi\in\mathbb{R}^{K m_2},\ \beta\in\mathbb{R}^{K n},\ \gamma\in\mathbb{R}^{K n},\ $ $B\in\mathbb{R}^{K m_2\times K n},\ A\in\mathbb{R}^{K m_1\times K n}$, $E_K^n\in\mathbb{R}^{K n\times n}$ are respectively defined by
 \begin{equation}
     \zeta:=\left[\begin{array}{l}       \zeta^1\\
          \vdots \\
         \zeta^{K}    \end{array}\right],\ z:=\left[\begin{array}{c}
         z^1\\
          \vdots\\
         z^{K}
    \end{array}\right],\     \alpha:=\left[\begin{array}{l}       \alpha^1\\
          \vdots \\
         \alpha^{K}    \end{array}\right],     \ \xi:=\left[\begin{array}{c}
         \xi^1\\
          \vdots\\
         \xi^{K}
    \end{array}\right],\ \beta:=\left[\begin{array}{c}
         \beta^1\\
         \vdots\\
         \beta^{K}
    \end{array}\right],\    
    \gamma:=\left[\begin{array}{c}
         \gamma^1\\
         \vdots\\
         \gamma^{K}
    \end{array}\right],
 \end{equation}
 \begin{equation}\label{13}
    B:=\left[\begin{array}{ccc}
         B^1&{}&{}\\
         {}& \ddots&{}\\
         {}&{}&B^{K}
    \end{array}\right],\ 
     A:=\left[\begin{array}{ccc}
         A^1&{}&{}\\
         {}& \ddots&{}\\
         {}&{}&A^{K}
    \end{array}\right],\
    E_K^n:=\left[\begin{array}{c}
         I_n\\
         \vdots\\
         I_n
    \end{array}\right],
\end{equation}

\begin{equation}\label{B^t}
    B^{t}\!:=\! \left[\begin{array}{c}
y_{(t-1){m_{2}+1}} x_{(t-1){m_{2}+1}}^{\top} \\
\vdots \\
y_{t{m_{2}}} x_{t{m_{2}}}^{\top}
\end{array}\right] \! \in \! \mathbb{R}^{m_{2} \times n},\ 
 A^{t}:=\left[\begin{array}{c}
y_{(t-1){m_{1}+1}} x_{(t-1){m_{1}+1}}^{\top} \\
\vdots \\
y_{t{m_{1}}} x_{t{m_{1}}}^{\top}
\end{array}\right] \in \mathbb{R}^{m_{1} \times n} ,\ t=1, \cdots, K.
\setlength{\abovedisplayskip}{3pt}
\end{equation}
The inclusion of equilibrium constraints in problem (\ref{MPEC}) categorizes it as an MPEC, which can be written
in the compact form
\begin{equation}\label{compact}
    \begin{array}{cl}
         \min\limits_{v \in \mathbb{R}^{\overline{m}+\overline{n}}}& f(v)  \\
         \hbox{s.t.}&g(v)\geq  \mathbf{0},\\
         {}& \mathbf{0} \leq  H(v)  \perp  G(v)  \geq \mathbf{0},
    \end{array}
\end{equation}
where 
\[
f(v)\! := \!M^{\top}v, \quad G(v)\! := \!Pv+a, \quad  H(v) \!:= \!Qv, \quad \mbox{ and }\;\, g(v) :=R v+b
\]
with 

\begin{align*}
    &M\in\mathbb{R}^{\overline{m} +\overline{n}}, \ a \! \in \! \mathbb{R}^{\overline{m}},\
Q\!:= \!\left[\begin{array}{cc} \mathbf{0}_{\overline{n}\times\overline{m}}& I_{\overline{m}} 
\end{array}\right]\in \! \mathbb{R}^{(\overline{m}) \times (\overline{m}+\overline{n})},\
P\in \! \mathbb{R}^{(\overline{m}) \times (\overline{m}+\overline{n})}, \\
&R\! \in \! \mathbb{R}^{2\overline{n} \times (\overline{m}+\overline{n})}, \ 
b\!:=\!\left[\begin{array}{cccc}
     C_{ub}&
     -C_{lb}&
     \overline{w}_{ub}^{\top}&
     -\overline{w}_{lb}^{\top}\\
\end{array}\right]^{\top}  \! \in \! \mathbb{R}^{2\overline{n}}, 
\end{align*}
and
 $$
  M\! :=\!\frac{1}{T m_{1}} \left[\begin{array}{c} \mathbf{0}_{\overline{n}}\\ \mathbf{1}_{K m_1} \\ \mathbf{0}_{K m_1 } \\ \mathbf{0}_{K m_2 } \\ \mathbf{0}_{K m_2 } \\ \mathbf{0}_{K n }\\ \mathbf{0}_{K n }\end{array}\right] ,\ a \! := \! \left[\begin{array}{l} \mathbf{0}_{K m_1} \\ \mathbf{1}_{K m_1} \\-\mathbf{1}_{K m_2} \\ \mathbf{0}_{K m_2} \\ \mathbf{0}_{K n} \\ \mathbf{0}_{K n} \end{array}\right],\ R\! :=\!\left[\begin{array}{ccc}
     -1& \mathbf{0}_{1 \times  n}& \mathbf{0}_{1 \times  \overline{m}}\\
      1& \mathbf{0}_{1 \times  n}& \mathbf{0}_{1 \times  \overline{m}}\\ 
      0&-I_{n}& \mathbf{0}_{1 \times  \overline{m}}\\ 
      0&I_{n}& \mathbf{0}_{1 \times  \overline{m}}
\end{array}\right],
$$

\begin{align*}
    &\text{w.r.t.}\begin{array}{cccccccc}
\quad \ C &\quad \quad \  \overline{w} &\quad \quad \quad \  \zeta &\quad \quad \quad \quad z & \quad \quad \quad\quad \quad \alpha &\quad \quad \quad \quad \quad \xi &\quad  \quad \quad\quad \beta &\quad  \quad \quad \gamma \quad\quad\quad\\
\hline
\end{array}\\
          &P \! \!:=\left[\begin{array}{cccccccc}
          \mathbf{0}_{K m_1 }& \mathbf{0}_{K m_1 \times n}&\mathbf{0}_{K m_1 \times K m_1} & I_{K m_1} & A B^{\top} & \mathbf{0}_{K m_1 \times K m_2} & A & -A\\ \mathbf{0}_{K m_1} &\mathbf{0}_{K m_1 \times n}&-I_{K m_1 } & \mathbf{0}_{K m_1 \times K m_1} & \mathbf{0}_{K m_1 \times K m_2} & \mathbf{0}_{K m_1 \times K m_2}& \mathbf{0}_{K m_1 \times K n}& \mathbf{0}_{K m_1 \times K n}\\ \mathbf{0}_{K m_2} &\mathbf{0}_{K m_2 \times n}&\mathbf{0}_{K m_2 \times K m_1} & \mathbf{0}_{K m_2 \times K m_1} & B B^{\top} & I_{ K m_2} & B& -B\\
\mathbf{1}_{K m_2 } &\mathbf{0}_{K m_2 \times n}&\mathbf{0}_{K m_2 \times K m_1} & \mathbf{0}_{K m_2 \times K m_1} & -I_{K m_2 } & \mathbf{0}_{K m_2 \times K m_2 }& \mathbf{0}_{K m_2 \times K n}& \mathbf{0}_{K m_2 \times K n}\\
\mathbf{0}_{K n }& E_K^n & \mathbf{0}_{K n \times K m_1} & \mathbf{0}_{K n \times K m_1} & B^{\top} &\mathbf{0}_{K n \times K m_2}& I_{K n} & -I_{K n}\\
\mathbf{0}_{K n }& E_K^n & \mathbf{0}_{K n \times K m_1} & \mathbf{0}_{K n \times K m_1} & -B^{\top} &\mathbf{0}_{K n \times K m_2}& -I_{K n} & I_{K n}\\
\end{array}\right].
\end{align*}    

${}$\\



The commonly used nonlinear programming constraint qualifications, such as the Mangasarian–Fromovitz constraint qualification (MFCQ), are known to be violated in the context of MPEC problems (see \cite{ye1997exact}). However, specialized constraint qualifications have been developed specifically for MPECs, with the most commonly applied being MPEC-MFCQ \cite{jane2005necessary}. In \cite{samuel2025mathematical}, a relaxed version of the MPEC-MFCQ has been proposed, which has proven effective for MPEC reformulations in nonlinear SVC hyperparameter selection. This advancement prompts a critical inquiry: Could the MPEC-MFCQ also be applicable to equation (\ref{OP})? To explore this possibility, we first verify the MPEC-MFCQ condition for (\ref{compact}), beginning with an introduction to positive-linearly independence.
\begin{definition}\label{positive_linearly_independence}
 Consider two sets of vectors \( V_a := \left\{ \tilde{a}_i : i \in I_1 \right\} \) and \( V_b := \left\{ \tilde{b}_i : i \in I_2 \right\} \). We define \( (V_a, V_b) \) as \textit{positive-linearly dependent} if and only if there exist scalars \( \left\{\Tilde{\tau}_i\right\}_{i \in I_1} \) and \( \left\{\Tilde{\sigma}_i\right\}_{i \in I_2} \), not all of them zero, such that
\begin{equation*}
    \sum_{i \in I_1} \Tilde{\tau}_i \tilde{a}_i + \sum_{i \in I_2} \Tilde{\sigma}_i \tilde{b}_i = \mathbf{0},
\end{equation*}
where \( \Tilde{\tau}_i \geq 0 \) for each \( i \in I_1 \) and \( \Tilde{\sigma}_i \) is free for each \( i \in I_2 \). If no such scalars exist, then \( (V_a, V_b) \) is called \textit{positive-linearly independent}.
\end{definition}
Having established the concept of positive-linearly independence, we now proceed to introduce the definition of the MPEC-MFCQ.
\begin{definition}\label{def1}
\rm\cite{samuel2025mathematical} A feasible point $v$ for problem (\ref{compact}) satisfies the MPEC-MFCQ if and only if the set of gradient vectors
\begin{align}
    &V_a = \left\{\nabla g_{i}\left(v\right) \mid i \in  I_{g} \right\}\cup\left\{\nabla G_{i}\left(v\right) \mid i \in  I_{GH} \right\}
\cup \left\{\nabla H_{i}\left(v\right) \mid i \in   I_{GH}\right\}, \label{grad1} \\
    & V_b = \left\{\nabla G_{i}\left(v\right) \mid i \in  I_{G} \right\}
\cup \left\{\nabla H_{i}\left(v\right) \mid i \in   I_{H}\right\} \label{grad2}
\end{align}
are positive-linearly independent, where $I_G$, $I_{GH}$, $I_H$ and $I_g$ are defined by
\begin{align*}
& I_{G}:=   \{i\in[\overline{m}] \ \mid \ G_{i}(v)=0,\ H_{i}(v)>0 \}, 
& I_{GH}:&=   \{i\in[\overline{m}] \ \mid \ G_{i}(v)=0,\ H_{i}(v)=0 \},\\
& I_{H}:=  \{i\in[\overline{m}] \ \mid \ G_{i}(v)>0,\ H_{i}(v)=0 \},
& I_{g}:=&   \{i\in[2\overline{n}] \ \mid \ g_i(v)=0 \}.
\end{align*}
\end{definition}
The primary objective of this section is to demonstrate that our problem satisfies the MPEC-MFCQ, with the proof detailed in Appendix \ref{appendixa}. 
\begin{theorem}\label{thr1}
A feasible point $v = \left[C,\ \overline{w}^{\top},\ \zeta^{\top},\ z^{\top},\ \alpha^{\top},\ \xi^{\top}, \beta^{\top},\gamma^{\top}\right]^{\top}$ of  problem  \eqref{compact} automatically satisfies the MPEC-MFCQ, provided that $-\overline{w} < w^t < \overline{w}$ with $w^t=\left(B^t\right)^{\top}\alpha^t+\beta^t-\gamma^t$ for $t=1,\cdots, K$, and the matrix 
\[
\overline{B}:=\left(BB^{\top}\right)_{(\Lambda_1\cup\Lambda_3,\Lambda_1\cup\Lambda_3)}\; \mbox{ with }\; B  \;\mbox{ given in }\; \eqref{13} \; \mbox{ and }\; \eqref{B^t}
\]
is positive-definite. Note that here, the index sets $\Lambda_1$ and $\Lambda_3$ are respectively defined as follows:
    \begin{align}
        \Lambda_{1}&=\left\{i \in [K m_2]\ \mid \ \alpha_{i}=0,\ \left(B  B^{\top}\alpha+B\beta-B\gamma-\mathbf{1}+\xi\right)_{i}=0,\ \xi_{i}=0\right\}, \\
\Lambda_{3}&=\left\{i \in [K m_2]\ \mid \ 0< \alpha_{i} \leq C,\ \left(B  B^{\top}\alpha+B\beta-B\gamma-\mathbf{1}+\xi\right)_{i}=0,\ \xi_{i}=0\right\}.
    \end{align}
\end{theorem}




Note that here, the assumption $-\overline{w} < w^t < \overline{w}$, $t = 1, \cdots, K$ serves as a relatively mild constraint within the framework of the hyperparameter optimization problem (\ref{OP}), primarily functioning as a feature selection mechanism. Without this constraint, the model simplifies to a conventional bilevel hyperparameter selection problem, which has been extensively explored in the literature (\cite{li2022bilevel, couellan2015bi}). Although these constraints are not strictly necessary, they are retained in our analysis to facilitate a more structured and insightful discussion.

As for our requirement that the matrix $\overline{B}$ be positive-definite, it is important to note that this matrix is an inherently positive semi-definite matrix by its structural properties. Extending this to assume positive-definiteness is thus a moderate assumption. 
%
Also note that according to findings in \cite{li2022bilevel}, the index sets $\Lambda_1$ and $\Lambda_3$ identify data points that lie on the hyperplanes defined by $w^{\top} x = 1$ or $w^{\top} x = -1$ across each fold. 
Building upon these insights, let us provide a sufficient condition for the matrix $\overline{B}$ to be positive definite. 
\begin{proposition}\label{propsition_rank} 
If the matrix $B_{(\Lambda_1\cup\Lambda_3,\cdot)}$ exhibits full row rank, then the matrix $\overline{B}$ be positive-definite. 
\end{proposition}
\begin{proof}
    Let $d\in \mathbb{R}^{\mid\Lambda_1 \cup \Lambda_3\mid} \setminus \{0\}$.
Considering the quadratic form associated with the submatrix $\overline{B}$, 
    \begin{align}
        d^{\top} \overline{B} d &= d^{\top}B_{(\Lambda_1\cup\Lambda_3,\cdot)}\left(B^{\top}\right)_{(\cdot,\Lambda_1\cup\Lambda_3)} d=d^{\top}B_{(\Lambda_1\cup\Lambda_3,\cdot)}\left(B_{(\Lambda_1\cup\Lambda_3,\cdot)}\right)^{\top} d\\ &= \left(d^{\top}B_{(\Lambda_1\cup\Lambda_3,\cdot)}\right)\left(d^{\top}B_{(\Lambda_1\cup\Lambda_3,\cdot)}\right)^{\top}=  \left\Vert d^{\top}B_{(\Lambda_1\cup\Lambda_3,\cdot)}\right\Vert^2>0.
    \end{align}
    The inequality holds since $d^{\top}B_{(\Lambda_1\cup\Lambda_3,\cdot)}\neq \mathbf{0}$, a consequence of the row linear independence of $B_{(\Lambda_1\cup\Lambda_3,\cdot)}$. 
    Therefore, the matrix $\overline{B}$ 
    is established to be positive-definite.
\end{proof}

Subsequently, we apply Proposition \ref{propsition_rank} to a practical two-dimensional scenario within the context of bilevel optimization as defined in (\ref{OP}).
\begin{example}
    Consider an example involving a bilevel SVC model structured using 2-fold cross-validation. The dataset is organized as follows:

\begin{equation*}
   \begin{array}{cc}
     \text{Fold 1}: & \text{Negative class:} \ \{[0,1,1], [-1,1,1]\}, \ \text{Positive class:} \ \{[1,0,1], [2,-1,1]\} \\
     \text{Fold 2}: & \text{Negative class:} \ \{[-1,0,1], [1,3,1]\}, \ \text{Positive class:} \ \{[0,-1,1], [1,-2,1]\}
   \end{array} 
\end{equation*}

The last component in each vector facilitates the inclusion of a bias term in the model. Parameter bounds are set with \(C_{lb}=0.1,\ C_{ub}=1\) and \(\overline{w}_{lb}=[0.1,0.1,0.1],\ \overline{w}_{ub}=[2,2,2]\). Each data point is indexed within the set \(Q_l = \{1, \dots, 8\}\). For this dataset, the classifiers for each fold are both determined by the parameter \(w=[1,-1,0]\). These parameters result in specific data points lying precisely on the decision boundaries specified by \(w^{\top} x = 1\) and \(w^{\top} x = -1\). The following figures illustrate the distribution of data points along with the decision boundaries for each fold.

\begin{figure}[ht]
\centering
\begin{minipage}{0.49\textwidth}
    \centering
    \includegraphics[width=\textwidth]{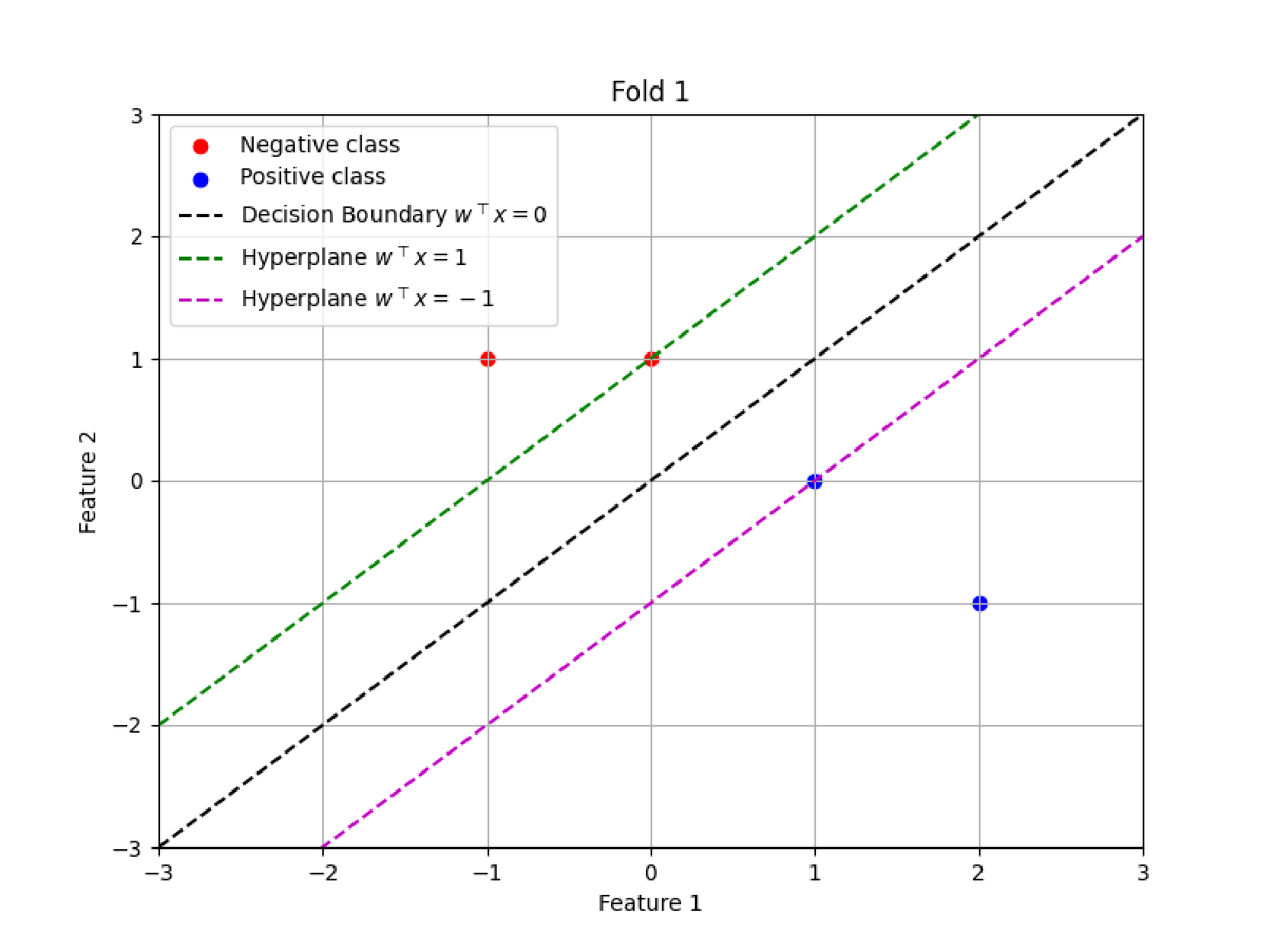} 
\end{minipage}\hfill
\begin{minipage}{0.49\textwidth}
    \centering
    \includegraphics[width=\textwidth]{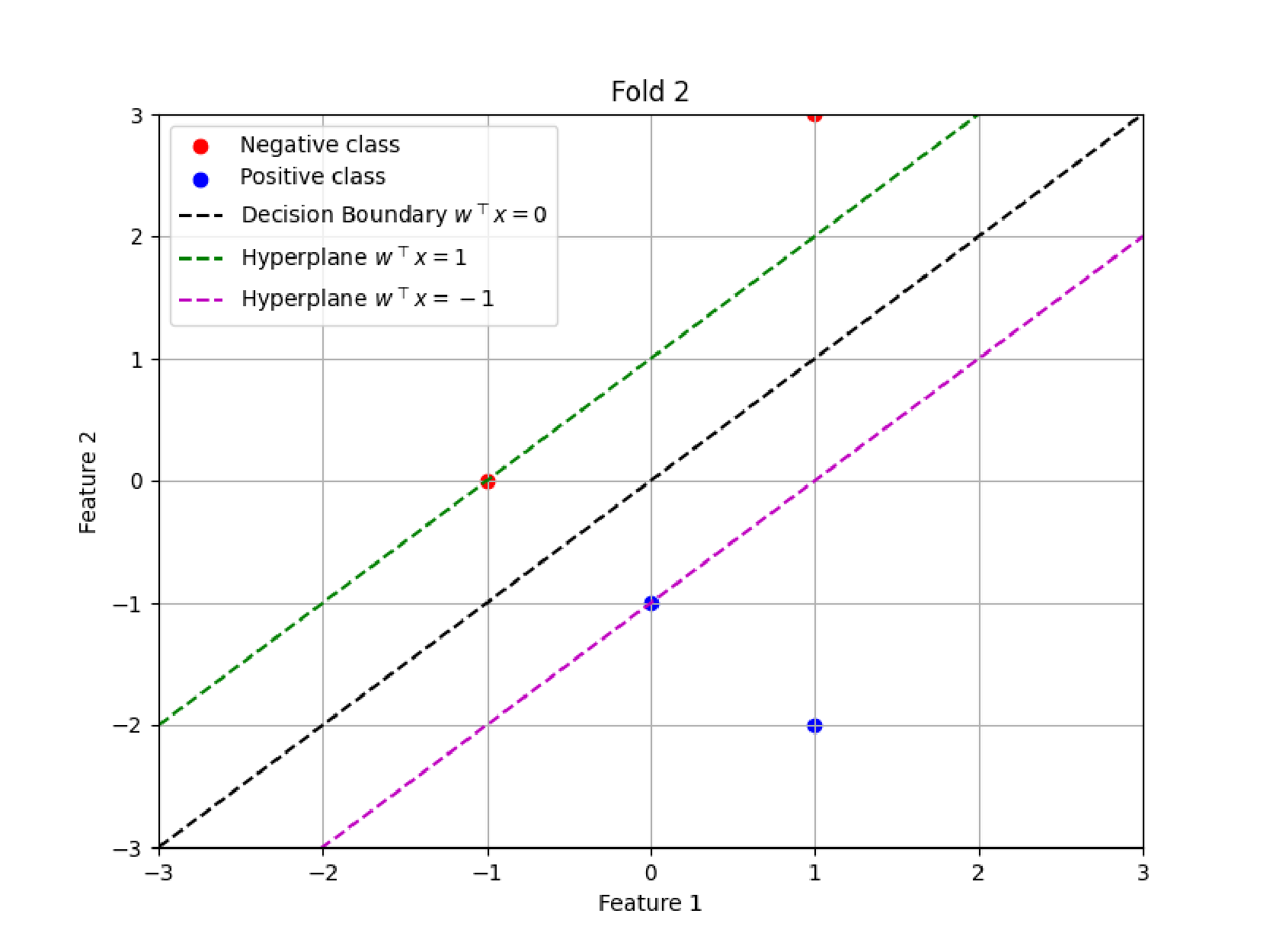} 
\end{minipage}
\end{figure}

As discussed above, we identify indices \(\Lambda_1 \cup \Lambda_3\) containing data points located on the defined hyperplanes. These indices are detailed as follows:

\begin{align*}
    \Lambda_1 \cup \Lambda_3 &= \left\{i \in Q_l \mid w^\top x_i = \pm 1\right\} 
    =  \left\{1, 3, 5, 7\right\},
\end{align*}
given that based on the description of the example, \(w=[1,-1,0]\). Here, \(x_i^j\) represents the \(j\)-th component of vector \(x_i\).
We can easily check that matrix 

\begin{equation*}
B_{(\Lambda_1 \cup \Lambda_3, \cdot)} = \begin{bmatrix}
     0 & -1 & -1 & 0 & 0 & 0 \\
     1 & 0 & 1 & 0 & 0 & 0 \\
     0 & 0 & 0 & 1 & 0 & -1 \\
     0 & 0 & 0 & 0 & -1 & 1
\end{bmatrix}
\end{equation*}
has full row rank. Hence, $\overline{B}$ positive-definite by Proposition \ref{propsition_rank}.
\end{example}

To close this section, note that Theorem \ref{thr1} stands as a fundamental property that is essential for the convergence of our proposed algorithm, which is presented in the next section.

\section{Global relaxation-based LP--Newton method}\label{sec3}

In this section, we present an efficient numerical algorithm to address the MPEC reformulation as described in (\ref{compact}). Among the various methodologies available for solving MPECs, the Global Relaxation Method (GRM), introduced in \cite{scholtes2001convergence}, is particularly noteworthy for its widespread adoption and robust performance. A comparative analysis of five relaxation techniques in \cite{hoheisel2013theoretical} highlighted the GRM’s superior convergence properties and overall numerical effectiveness. Concurrently, the LP--Newton method \cite{facchinei2014lp,fischer2016globally,fischer2016convergence} has been extensively applied in solving systems of equations and has shown significant merits in addressing nonsmooth equations. This raises an intriguing question: can the LP--Newton method be effectively combined with the GRM to solve (\ref{compact})? To explore this possibility, we propose the GRLPN method, which we detail in the subsequent sections.

The main difficulty associated with solving MPECs due to the presence of complementarity constraints. Consequently, the basic idea behind GRM involves the relaxation of these complementarity constraints. At each iteration, the GRM aims to solve the following subproblem (parameterized in $\tau\downarrow 0$):

\begin{equation*}\label{NLP}\tag{NLP-$\tau$}
    \begin{array}{cl}
         \min\limits_{v\in\mathbb{R}^{\overline{n}+\overline{m}}}&  f(v)  \\
         \hbox{s.t.}& G_{i}(v) \geq 0, \quad   \ i=1, \cdots, \overline{m},\\
         {}&H_{i}(v) \geq 0, \quad   \ i=1, \cdots, \overline{m},\\
         {}&G_{i}(v)H_{i}(v) \leq \tau, \quad   \ i=1, \cdots, \overline{m},\\
         {}& g_{i}(v) \geq 0, \quad   \ i=1, \cdots, 2\overline{n}.
    \end{array}
\end{equation*}
The literature (see, e.g., \cite{li2022unified,li2022bilevel,kunapuli2008classification}) encompasses the GRM--based methodology for addressing the hyperparameter selection problem. Nonetheless, we observe a scarcity of publications that specifically address the improvement of solving the subproblem (\ref{NLP}), as most of them directly rely on the NLP solvers. Notice that in (\ref{NLP}), the objective function is a linear function, and in terms of constraints, except the nonlinear constraints $G_{i}(v)H_{i}(v) \leq \tau, \ i=1,\cdots, \overline{m}$, the rest are all linear constraints. Therefore, it is possible to design efficient algorithms to explore the above special structure in (\ref{NLP}). This is the motivation of our GRLPN algorithm.


As pointed out in \cite{hoheisel2013theoretical}, the MPEC-MFCQ condition ensures that the MFCQ holds for the relaxed problem when $\tau$ is sufficiently small. This justification supports the application of the KKT conditions to problem (\ref{NLP}). The Lagrangian function for this problem is critical to formulating the KKT conditions and is given by: \begin{equation*}\label{Lagfun} L(v,\lambda):=f(v) - \sum_{i=1}^{2\overline{n}} \lambda_i^g g_i(v) - \sum_{i=1}^{\overline{m}} \lambda_i^G G_i(v) - \sum_{i=1}^{\overline{m}} \lambda_i^H H_i(v) + \sum_{i=1}^{\overline{m}} \lambda_i^{GH} \left[G_i(v)H_i(v)-\tau\right], \end{equation*} 
where $\lambda:=\left(\left(\lambda^g\right)^{\top},\left(\lambda^G\right)^{\top},\left(\lambda^H\right)^{\top},\left(\lambda^{GH}\right)^{\top}\right)^{\top}\in\mathbb{R}^{2\overline{n}}\times\mathbb{R}^{\overline{m}}\times\mathbb{R}^{\overline{m}}\times\mathbb{R}^{\overline{m}}$.

The KKT conditions associated with this formulation are detailed below: 
\begin{align}
     & \nabla_v L(v,\lambda)=\mathbf{0},\label{26a}\\ 
    & g(v)\geq \mathbf{0},\ \lambda^g\geq \mathbf{0},\ \sum_{i=1}^{2\overline{n}}\lambda_i^g g_i(v)=0,\\
    & G(v)\geq \mathbf{0},\ \lambda^G\geq \mathbf{0},\ \sum_{i=1}^{\overline{m}}\lambda_i^G G_i(v)=0,\\
     & H(v)\geq \mathbf{0},\ \lambda^H\geq \mathbf{0},\ \sum_{i=1}^{\overline{m}}\lambda_i^H H_i(v)=0,\\
      & \tau-G_i(v)H_i(v)\geq 0,\ i=1,\cdots,\overline{m},\\
      &\lambda^{GH}\geq \mathbf{0},\ \sum_{i=1}^{\overline{m}}\lambda_i^{GH} \left[ \tau-G_i(v)H_i(v)\right]=0.\label{26e}
\end{align}

To effectively address the KKT conditions, we reformulate these into a system of equations. To this end, inspired by the techniques in \cite{galantai2012properties}, we define the minimization function for (\ref{26a})--(\ref{26e}) as follows: 
\begin{equation}\label{eqsys1}
     K_{\tau}(v,\lambda): =\left[\begin{array}{c} \nabla_v L(v,\lambda)\\ \min\{\lambda^g,g(v)\}\\
     \min\{\lambda^G,G(v)\}\\
     \min\{\lambda^H,H(v)\}\\
     \min\{\lambda^{GH},\mathbf{\tau}- G(v)H(v)\}
    \end{array}\right]=\mathbf{0},\\
\end{equation} 
with $\mathbf{\tau}-G(v)H(v):=\lbrack \tau-G_1(v)H_1(v),\ \cdots\ , \tau-G_{\overline{m}}(v)H_{\overline{m}}(v)\rbrack^{\top}$ for simplicity.

Given the nonlinearity of the term $\mathbf{\tau}- G(v)H(v)$ in the minimization function, the direct application of the LP--Newton--type framework is impractical. Therefore, we introduce a slack variable $ u:=\mathbf{\tau}- G(v)H(v)$ with $u\in\mathbb{R}^{\overline{m}}$ for substitution, which results in the following system:

\begin{equation}\label{eqsys}
     F_{\tau}(\Tilde{z}): =\left[\begin{array}{c} \nabla_v L(v,\lambda)\\ 
    u+G(v)H(v)-\mathbf{\tau}\\
    \min\{\lambda^g,g(v)\}\\
     \min\{\lambda^G,G(v)\}\\
     \min\{\lambda^H,H(v)\}\\
    \min\{\lambda^{GH},u\}\\
    \end{array}\right]=\mathbf{0}\in \mathbb{R}^{\overline{q}},\ \Tilde{z}\in\Tilde{\Omega}, 
\end{equation}
where $\Tilde{z}:=(v,\lambda,u)\in \mathbb{R}^{\overline{q}} $, with $\overline{q}=3\overline{n}+5\overline{m}$ and  
\[
\Tilde{\Omega}:=\left\{\Tilde{z}\mid \lambda\geq \mathbf{0}, \;\; u\geq \mathbf{0},\;\; g(v)\geq \mathbf{0},\;\; G(v)\geq \mathbf{0},\;\; H(v)\geq \mathbf{0}\right\}.
\]

Note that with this substitution, both sides of the minimization function become linear. Furthermore,  
 the function $F_{\tau}(\Tilde{z})$ is piecewise continuously differentiable and comprises at most $2^{2\overline{n}+3\overline{m}}$ distinct pieces. Each piece, denoted by $F^i_{\tau}(\Tilde{z})$ for $i = 1, \cdots, 2^{2\overline{n}+3\overline{m}}$, corresponds to a continuously differentiable quadratic function. Therefore, $F_{\tau}(\Tilde{z})$ can be expressed as a selection among these quadratic functions; specifically, it holds that
\[
F_{\tau}(\Tilde{z}) \in \{\,F^i_{\tau}(\Tilde{z}) \mid i = 1,\cdots, 2^{2\overline{n}+3\overline{m}}\,\},
\]
where each selection function identifies the active (or minimal) quadratic piece associated with the minimization problem at a given point $\Tilde{z}$.


We now present the algorithm for solving \eqref{eqsys}, which is the LP-Newton method introduced in \cite{fischer2016globally}. This algorithm starts with a current iterate $\Tilde{z}^j \in\Tilde{\Omega}$, and the subsequent iterate is determined as $$\Tilde{z}^{j+1} := \Tilde{z}^j + \theta^m d^j,$$ where the pair $(\eta^j,d^j)\in \mathbb{R}\times\mathbb{R}^{\overline{q}}$ is a solution to the linear programming problem defined by
\begin{equation}\label{LPframe}
    \begin{array}{cl}
         \min\limits_{\eta\in\mathbb{R},d\in\mathbb{R}^{\overline{q}}}&  \eta\\
         \hbox{s.t.}& \Vert F_{\tau}\left(\Tilde{z}^j\right)+\mathcal{J} \left(\Tilde{z}^j\right)d\Vert_{\infty}\leq \eta\Vert F_{\tau}\left(\Tilde{z}^j\right)\Vert_{\infty}^2,\\
         {}&\Vert d\Vert_{\infty}\leq\eta\Vert F_{\tau}\left(\Tilde{z}^j\right)\Vert_{\infty},\\
          {}&\Tilde{z}^j+d\in\Tilde{\Omega},
    \end{array}
\end{equation}
with $\mathcal{J} \left(\Tilde{z}^j\right)\in\partial F_{\tau}\left(\Tilde{z}^j\right)$, and  $\partial F_{\tau}(\Tilde{z}^j)$ is the Clarke's generalized Jacobian of $F_{\tau}(\Tilde{z}^j)$ throughout the paper \cite{clarke1983nonsmooth}.  
At each iteration, the active selection function $F_{\tau}(\Tilde{z}^j)$ is determined by comparing the values of the minimization function on both sides. Consequently, $\mathcal{J}(\Tilde{z}^j)$ could be selected as the Jacobian of the selection function at the point $\Tilde{z}^j$. Furthermore, $m\geq 0$ is the minimum integer satisfying
    \begin{equation}\label{check}
	    \Vert F_{\tau}\left(\Tilde{z}^j+\theta^m d^j\right)\Vert_{\infty}\leq  \Vert F_{\tau}\left(\Tilde{z}^j\right)\Vert_{\infty}+\sigma_1\theta^m\Delta_j,
     \end{equation}
with $\sigma_1\in(0,1)$ being a constant and 
\begin{equation}\label{Delta}
    \Delta_j :=-\Vert F_{\tau}\left(\Tilde{z}^j\right)\Vert \left(1-\eta\Vert F_{\tau}\left(\Tilde{z}^j\right)\Vert\right)
\end{equation}
serving as an appropriate approximation of the directional derivative of $\Vert F_{\tau}(\Tilde{z}^j)\Vert$ along direction $d^j$ in \cite{fischer2016globally}. 

    
    

\begin{algorithm}
\caption{GRLPN}\label{alo1}
\KwData{Initial relaxation parameter ${\tau}_0$, parameters $\sigma_1 \in  \left(0, 1\right),\ \sigma_2 \in  \left(0, 1\right),\ \theta \in  \left(0, 1\right),\ {\tau}_{\min} > 0$, initial point $\left(v^0,\ \lambda^0,\ u^0\right)$.}
$k:=0$.\;

\While{outer stopping criteria is not satisfied}{
  $j:=0,\ \Tilde{z}^j:=\left(\Tilde{z}^j_v,\Tilde{z}^j_{\lambda},\Tilde{z}^j_u\right)=\left(v^k,\lambda^k,u^k\right)$.\;
  
  \While{inner stopping criteria is not satisfied}{
  Compute $\left(\eta^j, d^j\right)$ with $\eta_j\in\mathbb{R},\ d^j\in\mathbb{R}^{\bar{q}}$ as a solution of (\ref{LPframe}) with $\tau={\tau}_k$ at $\Tilde{z}^j$.\; 

  Let $m\geq 0$ be the minimum integer satisfying (\ref{check}).\;

  Let $\Tilde{z}^{j+1}:=\Tilde{z}^{j}+\theta^m d^j$.\;

  $j\gets j+1$.\;
  }
  Let $v^{k+1}=\Tilde{z}_v^j,\ {\lambda}^{k+1}=\Tilde{z}_{\lambda}^j,\ {u}^{k+1}=\Tilde{z}_{u}^j,\ {\tau}_{k+1}:=\sigma_2 {\tau}_{k}$.\;
  
  $k\gets k+1$.\;
}
\KwResult{The final iterate $v_{opt} := v^{k}$.}
\end{algorithm}

    
    


The framework of our algorithm, the global relaxation-based LP-Newton method (GRLPN), is presented in Algorithm \ref{alo1}, where the inner (while) loop tackles the equivalent KKT system \eqref{eqsys} associated with (\ref{NLP}) using the LP-Newton method. This step is carried out by iteratively solving the (\ref{LPframe}) and doing the line search (\ref{check}), all with a fixed relaxation parameter ${\tau}_k$. Conversely, the outer loop employs the global relaxation technique to systematically reduce the relaxation parameter ${\tau}_k$. In terms of the inner stopping criteria, the inner loop terminates if either of the conditions holds:
\begin{equation*}\tag{I}\label{inner}
     \Vert F_{\tau}\left(\Tilde{z}^{j}\right)\Vert_{\infty}=0 \,\mbox{ or }\,
    \Delta_j=0.
\end{equation*}
For the outer loop, we utilize the standard stopping criteria (\ref{outer}) for global relaxation as follows:
\begin{equation*}\tag{O}\label{outer}
    \ {\tau}_k\leq {\tau}_{min}.
\end{equation*}
 
In this context, the tag (\ref{inner}) refers to the stopping criteria for the inner loop, and (\ref{outer}) denotes the criteria for the outer loop.


    We next present the convergence results of the GRLPN. Theorem \ref{thr2} demonstrates the convergence of the inner loop, building on Corollary 4.1 from \cite{fischer2016globally}. Additionally, Theorems \ref{thr4} and \ref{thrm5} illustrate the convergence properties of the outer loop in Algorithm \ref{alo1}.
 
Before delving into the theorems, it is essential to define some key mathematical notations. Let $\tau:={\tau}_k$ be a fixed value, and define $Z:=\left\{\Tilde{z}\in\Tilde{\Omega} \mid F_{\tau}(\Tilde{z})=0\right\}$ as the solution set for equation (\ref{eqsys}), with $\Tilde{z}^* \in Z$ as an arbitrarily chosen but fixed point. For $p\in\left\{1,\cdots, 2^{2\overline{n}+3\overline{m}}\right\}$, let $Z^p:=\left\{\Tilde{z}\in\Tilde{\Omega}\mid F_{\tau}^p(\Tilde{z})=0\right\}$ denote the solution set of the selection function $F^p_{\tau}(\Tilde{z})$. Additionally, let $h^p_{\tau}(\Tilde{z})=\Vert F^p_{\tau}(\Tilde{z})\Vert$ represent the merit function associated with the selection function $F^p_{\tau}(\Tilde{z})$. We also define $\mathcal{I}(\Tilde{z}^*)=\left\{i=1,\cdots, 2^{2\overline{n}+3\overline{m}} \mid F_{\tau}(\Tilde{z}^*)=F^i_{\tau}(\Tilde{z}^*)\right\}$ as the index set of all selection functions at the point $\Tilde{z}^*$. For every active selection $p \in \mathcal{I}(\Tilde{z}^*)$, equation (\ref{eqsys}) can be reformulated as follows:
 \begin{equation}\label{sel}
\small
     F_{\tau}^p(\Tilde{z}^*): =\left[\begin{array}{c} \nabla_v L(v,\lambda)\\ 
    u+G(v)H(v)-\mathbf{\tau}\\
    \lambda^g_{\mathcal{A}_g^{1}\cup\mathcal{A}_g^{3,1}} \\
    g\left(v\right)_{\mathcal{A}_g^{2}\cup\mathcal{A}_g^{3,2}} \\
    \lambda^G_{\mathcal{A}_G^{1}\cup\mathcal{A}_G^{3,1}}\\
    G(v)_{\mathcal{A}_G^{2}\cup\mathcal{A}_G^{3,2}}\\
     \lambda^H_{\mathcal{A}_H^{1}\cup\mathcal{A}_{H}^{3,1}}\\
     H(v)_{\mathcal{A}_H^{2}\cup\mathcal{A}_H^{3,2}}\\
    \lambda^{GH}_{\mathcal{A}_{GH}^{1}\cup\mathcal{A}_{GH}^{3,1}}\\
      u_{\mathcal{A}_{GH}^{2}\cup\mathcal{A}_{GH}^{3,2}}\\
    \end{array}\right]=\mathbf{0}, \ 
    \text{where} \ \left\{     
    \begin{array}{c} 
    \mathcal{A}_g^{1}:=\mathcal{A}_g^{1}(\Tilde{z}^*):=\{ i\in [2\overline{n}]\mid \left(\lambda^g\right)^*_i<g_i(v^*)\},\\
    \mathcal{A}_g^{2}:=\mathcal{A}_g^{2}(\Tilde{z}^*):=\{ i\in [2\overline{n}]\mid \left(\lambda^g\right)^*_i> g_i(v^*)\},\\
    \mathcal{A}_g^{3}:=\mathcal{A}_g^{3}(\Tilde{z}^*):=\{ i\in [2\overline{n}]\mid \left(\lambda^g\right)^*_i= g_i(v^*)\},\\
     \mathcal{A}_G^{1}:=\mathcal{A}_G^{1}(\Tilde{z}^*):=\{ i\in [\overline{m}]\mid (\lambda^G)^*_i<G_i(v^*)\},\\
     \mathcal{A}_G^{2}:=\mathcal{A}_G^{2}(\Tilde{z}^*):=\{ i\in [\overline{m}]\mid (\lambda^G)^*_i> G_i(v^*)\},\\
     \mathcal{A}_G^{3}:=\mathcal{A}_G^{3}(\Tilde{z}^*):=\{ i\in [\overline{m}]\mid (\lambda^G)^*_i= G_i(v^*)\},\\
      \mathcal{A}_H^{1}:=\mathcal{A}_H^{1}(\Tilde{z}^*):=\{ i\in [\overline{m}]\mid (\lambda^H)^*_i<H_i(v^*)\},\\
      \mathcal{A}_H^{2}:=\mathcal{A}_H^{2}(\Tilde{z}^*):=\{ i\in [\overline{m}]\mid (\lambda^H)^*_i> H_i(v^*)\},\\
      \mathcal{A}_H^{3}:=\mathcal{A}_H^{3}(\Tilde{z}^*):=\{ i\in [\overline{m}]\mid (\lambda^H)^*_i= H_i(v^*)\},\\
       \mathcal{A}_{GH}^{1}:=\mathcal{A}_{GH}^{1}(\Tilde{z}^*):=\{ i\in [\overline{m}]\mid (\lambda^{GH})^*_i<u^*_i\},\\
       \mathcal{A}_{GH}^{2}:=\mathcal{A}_{GH}^{2}(\Tilde{z}^*):=\{ i\in [\overline{m}]\mid (\lambda^{GH})^*_i> u^*_i\},\\
       \mathcal{A}_{GH}^{3}:=\mathcal{A}_{GH}^{3}(\Tilde{z}^*):=\{ i\in [\overline{m}]\mid (\lambda^{GH})^*_i= u^*_i\}\\
    \end{array}
    \right.
    \end{equation}
For  $\mathcal{A}_g^{3}$, we use  $\mathcal{A}_g^{3,1}$ and  $\mathcal{A}_g^{3,2}$ to denote the partition of  $\mathcal{A}_g^{3}$, such that \[\mathcal{A}_g^{3,1}\bigcup  \mathcal{A}_g^{3,2}= \mathcal{A}_g^{3},\ \mathcal{A}_g^{3,1}\bigcap  \mathcal{A}_g^{3,2}=\emptyset, \ \min\left\{\lambda^g_{\mathcal{A}_g^{3}},\ g(v^*)_{\mathcal{A}_g^{3}}\right\}=\left[\begin{array}{c}\lambda^g_{\mathcal{A}_g^{3,1}}\\
g(v^*)_{\mathcal{A}_g^{3,2}}
\end{array}\right].\]
That is, $\mathcal{A}_g^{3}=(\mathcal{A}_g^{3,1},\mathcal{A}_g^{3,2})$. Similar definitions are used for  $\mathcal{A}_G^{3}=(\mathcal{A}_G^{3,1},\mathcal{A}_G^{3,2})$, $\mathcal{A}_H^{3}=(\mathcal{A}_H^{3,1},\mathcal{A}_H^{3,2})$, and $\mathcal{A}_{GH}^{3}=(\mathcal{A}_{GH}^{3,1},\mathcal{A}_{GH}^{3,2})$. Different combinations of these sets result in $2^{\mid \mathcal{A}_g^3 \cup \mathcal{A}_G^3 \cup \mathcal{A}_H^3 \cup \mathcal{A}_{GH}^3 \mid}$ possible configurations for selecting subsets among all partitions.

With these definitions established, we now turn to the analysis of the local convergence results for the inner loop, based on the subsequent assumption.
\begin{assumption}\label{assum1}
    Assume that the matrix below has full row rank:
    \begin{equation*}
    \left[\begin{array}{cccccc}
         \nabla^2_{vv} L(v,\lambda)&-\nabla g(v)_{(\cdot,\mathcal{A}_g^2\cup\mathcal{A}_g^{3,2})}& -\nabla G(v)_{(\cdot,\mathcal{A}_G^2\cup\mathcal{A}_G^{3,2})}&-\nabla H(v)_{(\cdot,\mathcal{A}_H^2\cup\mathcal{A}_H^{3,2})}&S_{(\cdot,\mathcal{A}_{GH}^2\cup\mathcal{A}_{GH}^{3,2})}&\mathbf{0} \\          S^{\top}&\mathbf{0}&\mathbf{0}&\mathbf{0}&\mathbf{0}&I_{(\cdot,\mathcal{A}_{GH}^1\cup\mathcal{A}_{GH}^{3,1})}\\
         \nabla g(v)^{\top}_{(\mathcal{A}_g^2\cup\mathcal{A}_g^{3,2},\cdot)}&\mathbf{0}&\mathbf{0}&\mathbf{0}&\mathbf{0}&\mathbf{0}\\
         \nabla G(v)^{\top}_{(\mathcal{A}_G^2\cup\mathcal{A}_G^{3,2},\cdot)}&\mathbf{0}&\mathbf{0}&\mathbf{0}&\mathbf{0}&\mathbf{0}\\
         \nabla H(v)^{\top}_{(\mathcal{A}_H^2\cup\mathcal{A}_H^{3,2},\cdot)}&\mathbf{0}&\mathbf{0}&\mathbf{0}&\mathbf{0}&\mathbf{0}\\
    \end{array}
    \right],
    \end{equation*}
    where $S$ is defined as $\left[G_1(v)\nabla H_1(v)+H_1(v)\nabla G_1(v),\cdots,G_{\overline{m}}(v)\nabla H_{\overline{m}}(v)+H_{\overline{m}}(v)\nabla G_{\overline{m}}(v)\right]$.
\end{assumption}
Building on Assumption \ref{assum1}, we demonstrate that the inner loop achieves quadratic convergence rate.
\begin{theorem}\label{thr2}
     Let Assumption \ref{assum1}
     hold at $\Tilde{z}^*$ for each partition $p\in\mathcal{I}(\Tilde{z}^*)$ in (\ref{sel}). Then there exists a radius $r>0$ such that any infinite sequence $\left\{\Tilde{z}^k\right\}$ generated by the inner loop of Algorithm \ref{alo1} with starting point $\Tilde{z}^0\in\mathcal{B}_r(\Tilde{z}^*)\cap\Omega$ converges to some $\hat{z}\in Z$, and the rate of convergence is Q-quadratic. Here, $\mathcal{B}_r(\Tilde{z}^*):=\left\{ \Tilde{z}\in\mathbb{R}^{\overline{q}}\mid \Vert \Tilde{z}-\Tilde{z}^*\Vert\leq r\right\}$ denote the ball of radius $r$ around $\Tilde{z}^*$.
     \end{theorem}
     \begin{proof}
     Observing that the selection functions $ F^i_{\tau}(\Tilde{z}),\ i=1,\cdots, 2^{2\overline{n}+3\overline{m}}$ are quadratic functions, and therefore have Lipschitz continuous derivatives. According to \cite{fischer2016globally}, our objective is to establish that Condition 2 from \cite{fischer2016convergence} is satisfied. This condition is met if the MFCQ applies to (\ref{sel}). Under Assumption \ref{assum1}, it is straightforward to verify that MFCQ indeed holds.
     \end{proof}
     Note that Assumption \ref{assum1} is specifically required for the inner loop to achieve quadratic convergence. However, even if this assumption does not hold, the inner loop still converges. To demonstrate the practical applicability of Assumption \ref{assum1}, we provide an example showing how it can be satisfied in the context of our problem.
      \begin{example}
          Consider the following MPEC:
          \begin{equation}\label{eg1}
    \begin{array}{cl}
         \min\limits_{v}&v_1+v_2-v_3  \\
         \hbox{s.t.}& g(v)=v_3\leq 0,\\
         {}&G(v)=v_1\geq 0,\\
         {}&H(v)=v_2\geq 0,\\
         {}&G(v)H(v)=v_1v_2= 0,
    \end{array}
\end{equation}
with $v=\left(v_1,v_2,v_3\right)\in\mathbb{R}^3$.
According to \cite{flegel2005constraint}, the origin is the unique solution to this problem and satisfies the MPEC-MFCQ condition. We proceed to consider a global relaxation problem of (\ref{eg1}), replacing the complementarity constraint with $G(v)H(v)=v_1v_2\leq 0.02$. The KKT system for the relaxed problem with a slack variable $u$ is detailed as follows:
\begin{equation}\label{eg2}
F_{\tau}(v,\lambda):=\left[\begin{array}{c}
     \nabla_v L(v,\lambda)\\ 
     u+ v_1 v_2-0.02\\
    \min\{\lambda_1, v_3\}\\
     \min\{\lambda_2,v_1\}\\
     \min\{\lambda_3,v_2\}\\
    \min\{\lambda_4,u\}\\
    \end{array}\right]=\mathbf{0}, 
\end{equation}
where $$L(v,\lambda)=v_1+v_2-v_3+\lambda_1v_3-\lambda_2v_1-\lambda_3v_2+\lambda_4(v_1v_2-0.02)$$ is the corresponding Lagrangian function. Verification shows that the point $(v_1^*,v_2^*,v_3^*,\lambda_1^*,\lambda_2^*,\lambda_3^*,\lambda_4^*,u^*)=(0,0,0,1,1,1,0,0.02)$ is a solution to (\ref{eg2}). As there is only one selection function and the associated matrix in Assumption \ref{assum1} maintains full row rank, Assumption \ref{assum1} is satisfied.
      \end{example}

Finally, we are ready to discuss the convergence properties of the outer loop of Algorithm \ref{alo1}. The following theorem demonstrates that the sequence generated by Algorithm \ref{alo1} converges to a C-stationary point of problem (\ref{compact}). We begin by defining a key concept: the C-stationary point within the context of an MPEC.
  
 \begin{definition}
Let $v$ be a feasible point for the MPEC in (\ref{compact}). 
Then $v$ is said to be a C-stationary point if there are multipliers $\gamma\in \mathbb{R}^{\overline{m}},\ \nu \in \mathbb{R}^{\overline{m}}$ and $\lambda \in \mathbb{R}^{2\overline{n}}$, such that
$$
\nabla f\left(v\right)-\sum_{i=1}^{\overline{m}} \gamma_{i} \nabla G_{i}\left(v\right) -\sum_{i=1}^{\overline{m}} \nu_{i} \nabla H_{i}\left(v\right)-\sum_{i=1}^{2\overline{n}} \lambda_{i} \nabla g_{i}\left(v\right)=0,
$$
with 
$\lambda_i\geq0,\ \lambda_ig_i(v)=0$ for $i=1,\cdots,2\overline{n}$,   $\gamma_{i}=0$ for $ i \in I_{H},\ \nu_{i}=0$ for $i \in I_{G}$, and $\gamma_{i} \nu_{i} \geq 0$ for $i \in I_{GH}$.
\end{definition}
It is worth noting that various other definitions of stationarity exist for problem (\ref{compact}). For a more in-depth exploration of these concepts, see key references like \cite{dempe2019two,dempe2012karush,flegel2005constraint,zemkoho2021theoretical}. Building on the definition of C-stationary point, we present the main convergence result.
\begin{theorem} \label{thr4}{\rm}
Let $\{{\tau}_{k}\} \downarrow 0$ and let $v^k$ be a stationary point of \eqref{NLP} with $v^{k} \rightarrow v$, where $v$ is the feasible point of \eqref{compact}. Then $v$ is a C-stationary point of the problem (\ref{compact}).
\end{theorem}
\begin{proof}
With the fulfilment of the MPEC-MFCQ in Theorem \ref{thr1} and Theorem 5 in \cite{samuel2025mathematical}, we directly obtain the result.
\end{proof}

While Theorem \ref{thr4} indicates convergence to a C-stationary point if the inner loop of Algorithm \ref{alo1} is exactly solved, achieving such precision is often impractical numerically. Thus, we implement a practical stopping criterion for the inner loop as follows: 
\begin{equation*}\tag{$\rm I^*$}\label{innerap}
    \Vert F_{\tau}(\Tilde{z}^{j})\Vert_{\infty}\leq\epsilon_k,
\end{equation*}
 where $\epsilon_k>0$ is the tolerance level.

Moreover, we provide the following result (see Appendix \ref{appendixb} for the proof) that using (\ref{innerap}) as the stopping criteria does not compromise the convergence properties toward a C-stationary point.
\begin{theorem}\label{thrm5}
Let $\left\{{\tau}_k\right\}\downarrow0$, $\left\{\epsilon_k\right\}\downarrow0$ and assume that $\epsilon_k=O({\tau}_k^2)$ holds. Assume further that MPEC-MFCQ holds for (\ref{MPEC}). Let $\left\{v^k\right\}$ be a sequence generated by stopping the inner loop of Algorithm \ref{alo1} with stopping criteria \eqref{innerap}. Then any accumulation point $v^*$ of the sequence $\left\{v^k\right\}$ is a C-stationary point of problem \eqref{compact}.
\end{theorem}
 To summarize, this section introduces the GRLPN for solving the \eqref{compact} along with associated convergence results. In the next section, we will provide the numerical experiments.

\section{Numerical results}\label{sec4}

In this section, we report the numerical performance of the \texttt{GRLPN}. All numerical tests were conducted in MATLAB R2021b on a Windows 11 Dell Laptop equipped with an 11th Gen Intel(R) Core(TM) i7-11800H CPU at 2.3 GHz and 16 GB of RAM. All datasets, collected from the LIBSVM library\footnote{\url{https://www.csie.ntu.edu.tw/~cjlin/libsvmtools/datasets/}}, are used in various applications in computer science, astroparticle studies, biomedical engineering, space physics, and ionosphere science. We use a three-fold cross-validation set, i.e., $T=3$. Accordingly, the regularization hyperparameter $C$ was adjusted by a scaling factor of $\frac{3}{2}$ to accommodate the utilization of the entire dataset in the post-processing phase of the model construction, as discussed in \cite{kunapuli2008classification}.

\begin{table}[htbp]
\centering
\caption{Computational results for $T=3$.} \label{table2}
\begin{tabular}{cccccccc}
\rowcolor{LightCyan}
 Data set& Method  & $E_{t}$ ($\%$) & $E_{C}$ ($\%$) & \textit{Vio}& \textit{time}&\textit{size}& $C$ \\
\rowcolor{LightCyan}
$(l_1,l_2,n)$&{}&{}&{}&{}&{}&{}&{}\\
\midrule
fourclass & \texttt{GRLPN} & \textbf{21.53}&22.67 & 4.80e$-$3 & 3:24&9082&1.6e$-$1 \\
$(300,562,2)$ &\texttt{GRST} &36.83&33.33 & \textbf{2.11e$-$4} & 00:48 &1816&1.20e$-$4\\
{}&\texttt{GS}&\textbf{21.53}&\textbf{22}&-&\textbf{00:03}&3&1.5e$-$1\\
\rowcolor{LightCyan}
nonskin &\texttt{GRLPN}&\textbf{0.00}&\textbf{0.00}&\textbf{4.5e$-$3}&22:18&10919&7.37\\
\rowcolor{LightCyan}
$(360,640,3)$&\texttt{GRST}&29.53&8.33&2.0e$-$2&08:54&2183&1.5e$-$4\\
\rowcolor{LightCyan}
{}&\texttt{GS}&\textbf{0.00}&\textbf{0.00}&-&\textbf{00:04}&4&1.5e$-$4\\
svmguide1 &\texttt{GRLPN}&\textbf{8.28}&10.56&4.5e$-$3&19:01&10956&1.5e$-$2\\
$(360,640,4)$&\texttt{GRST}&18.75&5.56&\textbf{1.3e$-$3}&02:03&2190&1.5e$-$4\\
{}&\texttt{GS}&14.84&\textbf{5.00}&-&\textbf{0:15}&5&1.5e$-$2\\
\rowcolor{LightCyan}
liver &\texttt{GRLPN}&34.78&\textbf{14.14}&\textbf{1.61e$-$5}&00:30&3163&1.47\\
\rowcolor{LightCyan}
$(99,46,5)$&\texttt{GRST}&48.00&25.25&4.75e$-$4&00:08&631&1.5e$-$4\\
\rowcolor{LightCyan}
{}&\texttt{GS}&\textbf{30.43}&16.16&-&\textbf{00:04}&6&1.5e1\\
diabetes & \texttt{GRLPN}&\textbf{22.69}& 25.19  &2.90e$-$3 &10:24&8404&5.8e$-$1 \\
$(270,498,8)$& \texttt{GRST} & 33.13& 38.15  & \textbf{1.00e$-$4}& \textbf{01:18}&1678&1.50e$-$4\\
{}&\texttt{GS}& 24.50&\textbf{22.59}&-&03:41&9&1.50e4\\
\rowcolor{LightCyan}
breast  & \texttt{GRLPN} & \textbf{2.26}&\textbf{2.92}  & \textbf{1.07e$-$4}& 04:35 &7578&9.3e$-$1 \\
\rowcolor{LightCyan}
$(240,172,10)$& \texttt{GRST} & 29.57&45.00 & 6.17e$-$4& \textbf{02:58}&1512&1.50e$-$4\\
\rowcolor{LightCyan}
{}&\texttt{GS}&2.71&3.33&-&28:24&11&1.5e$-$1\\
heart  & \texttt{GRLPN} &\textbf{12.35}& \textbf{12.17} & \textbf{6.57e$-$5}&06:58&6159&1.14e$-$2\\
$(189,81,13)$& \texttt{GRST} &43.21&44.97 &2.33e$-$4 &\textbf{01:46}&1227 &2.33e$-$4\\
{}&\texttt{GS}&\textbf{12.35}& 12.70&-&08:31:48&14&1.5e$-$1\\
\rowcolor{LightCyan}
australian & \texttt{GRLPN} & \textbf{14.52}&\textbf{11.89}  & \textbf{2.31e$-$4} & 08:30 &8626 &1.51e$-$2\\
\rowcolor{LightCyan}
$(270,420,14)$ &\texttt{GRST} &44.29&44.81  & 4.60e$-$3 & \textbf{07:57} &1720&1.50e$-$4\\
\rowcolor{LightCyan}
{}&\texttt{GS}& \textbf{14.52}&12.59  &-& {32:43:38}&15&1.50e$-$2\\
german.number & \texttt{GRLPN} & \textbf{25.16}&\textbf{16.11}   & \textbf{1.44e$-$8} & \textbf{25:57} &11696 &1.42e$-$2\\
$(360,640,24)$ &\texttt{GRST} &31.09&28.06  & 1.99e$-$4 & 26:31 &2330&1.99e$-$4\\
{}&\texttt{GS}&-&-&-&-&-&-\\
\rowcolor{LightCyan}
ionosphere & \texttt{GRLPN}&6.67 &\textbf{6.10}&\textbf{1.33e$-$6}&\textbf{08:07}&8646&7.5e$-$1 \\
\rowcolor{LightCyan}
$(246,105,34)$& \texttt{GRST} &\textbf{2.90}&33.33& 2.10e$-$3 &29:48&1716 &1.50e$-$4\\
\rowcolor{LightCyan}
{}&\texttt{GS}&-&-&-&-&-&-\\


splice & \texttt{GRLPN}&\textbf{21.25}&\textbf{8.06}  &\textbf{1.88e$-$4} & \textbf{16:07}  &13028&1.50e$-$2\\
$(360,640,60)$& \texttt{GRST} & 24.21 &17.50&9.15e$-$2 &37:05&2582&1.65e$-$2\\
{}&\texttt{GS}&-&-&-&-&-&-\\
\rowcolor{LightCyan}
a1a & \texttt{GRLPN} & \textbf{20.44}&\textbf{17.08}  & \textbf{2.04e$-$4} & \textbf{11:15} &13411&7.55e$-$2\\
\rowcolor{LightCyan}
$(300,200,119)$ &\texttt{GRST} &25.44&21.00  &1.82e$-$2 & 55:57 &2635&1.50e$-$2 \\
\rowcolor{LightCyan}
{}&\texttt{GS}&-&-&-&-&-&-\\
w1a & \texttt{GRLPN} & \textbf{0.41}&\textbf{8.33}  & \textbf{3.40e$-$3}& \textbf{42:31 }&20108&7.67e$-$2 \\
$(300,200,300)$ &\texttt{GRST} &\textbf{0.41}&14.44  & 5.30e$-$3 & 2:42:22 &3902&2.92e$-$2 \\
{}&\texttt{GS}&-&-&-&-&-&-\\
\bottomrule
\end{tabular}
\end{table}


Here, we evaluate the performance of \texttt{GRLPN} against two benchmark methods: the grid search method (denoted as \texttt{GS}) and the global relaxation method solved using the MATLAB-based package SNOPT (denoted as \texttt{GRST}). The parameter settings for these methods are as follows: lower and upper bounds for the regularization hyperparameter $C$ are set as $C_{lb} = 1.0\times 10^{-4}$ and $C_{ub} = 1.0\times 10^{4}$, respectively. The weight vector bounds are $\left(\overline{w}_{lb}\right)_i = 1.0\times 10^{-6}$ and $\left(\overline{w}_{ub}\right)_i = 1.5$ for $i\in[n]$. The relaxation parameters used are ${\tau}_{0} = 1.0\times 10^{-1}$ and ${\tau}_{\min} = 1.0\times 10^{-8}$, with a reduction factor of $\sigma = 1.0\times 10^{-1}$.

For the \texttt{GRLPN}, the termination criterion for inner loops is defined by (\ref{innerap}), and the convergence tolerance for the inner loop is set at $\epsilon_k = 1.0\times 10^{-2}$. The outer-loop iterations are terminated once $\operatorname{Vio}\left(v\right)<1.0\times 10^{-3}$ is satisfied, where $\operatorname{Vio(v)}$ is defined in \eqref{Vio}. The linear programming framework (\ref{LPframe}) associated with \texttt{GRLPN} is handled using the Gurobi optimizer\footnote{\url{https://www.gurobi.com/documentation/}}.

In the \texttt{GRST} configuration, the default settings of SNOPT are used, with the 'iteration limit' specifically adjusted to 200,000. For the \texttt{GS} method, we vary $C$ over the set $\left\{10^{-4},10^{-3},10^{-2},10^{-1},1,10,100,1000,10^4\right\}$ and $\overline{w}_j$ over $\left\{0, 0.75, 1.5\right\}$ for $j \in [n]$. To establish an initial point for \texttt{GRLPN} and \texttt{GRST}, we adopt a two-step strategy: First, we select $C$ from $\frac{1}{(T-1)\times m_2}\left\{10^{-2},10^{-1},1,10,100\right\}$, and then solve a classical SVC without bound constraints on $w$ over the training set to obtain the feature vector $w$. The $\overline{w}$ is subsequently set to $\min\left\{\lvert w\rvert, \overline{w}_{ub}\right\}$. In datasets with missing entries (e.g., a1a, w1a), we set the corresponding $\overline{w}$ components to 0.5. Finally, the Lagrangian multipliers are determined by solving the associated optimization problems.

 
We proceed to evaluate the performance of the methods described above according to several metrics, including:
\begin{itemize}
    \item[(1)] Test error ($E_{t}$): formulated as
$E_{t}=\frac{1}{l_{2}} \sum_{(x, y) \in \Theta} \frac{1}{2} \lvert \operatorname{sign}\left(w^{\top} x\right)-y \rvert;$ 
\item[(2)] Cross-Validation Error (\(E_{C}\)): evaluated through the objective function in problem (\ref{OP});

\item[(3)] CPU Time (\textit{time}): computational time required by each method;

\item[(4)] Maximum Violation (\textit{Vio}): quantified as
\begin{equation}\label{Vio}
    \operatorname{Vio}(v)=\|\min\{G(v), H(v)\}\|_{\infty},
\end{equation}

used to measure the accuracy of the complementarity constraints;

\item[(5)] Problem Size (\textit{size}): number of variables involved in each method;

\item[(6)] Final Value of Parameter \(C\).

\end{itemize}
 The numerical results corresponding to these metrics are summarized in Table \ref{table2}.

\begin{table}[htbp]
\centering
\caption{The impact of ${\tau}_k$.} \label{table3}
\begin{tabular}{ccccccc}

\rowcolor{LightCyan}
Data set & Method  & $E_{t}$ ($\%$) & $E_{C}$ ($\%$) & \textit{Vio}& \textit{time}&$C$ \\
\midrule
fourclass & \texttt{GRLPN} & \textbf{21.53}&22.67 & \textbf{4.80e$-$3} & 03:24&1.6e$-$1 \\
{} &\texttt{InLP} & 22.42&\textbf{22.33}  & 1.48e$-$2 & \textbf{00:31} &1.3e$-$1 \\
\rowcolor{LightCyan}
heart  & \texttt{GRLPN} &\textbf{12.35}& \textbf{12.17} & \textbf{6.57e$-$5} &06:58&11.41\\
\rowcolor{LightCyan}
{}& \texttt{InLP}  &\textbf{12.35}&12.70 & 9.10e$-$3 &\textbf{02:52}&11.84\\
breast  & \texttt{GRLPN} & \textbf{2.26}&\textbf{2.92}  & \textbf{1.07e$-$4}& \textbf{04:35} &9.3e$-$1 \\
{}& \texttt{InLP} & \textbf{2.26}&\textbf{2.92}  & 8.70e$-$3& 07:20 &9.3e$-$1 \\
\rowcolor{LightCyan}
ionosphere & \texttt{GRLPN}&6.67 &\textbf{6.10}&\textbf{1.33e$-$6}&\textbf{08:07}&7.5e-1 \\
\rowcolor{LightCyan}
{}& \texttt{InLP} &\textbf{4.76} &\textbf{6.10}&9.50e$-$3&19:15&7.5e-1 \\
diabetes & \texttt{GRLPN}&\textbf{22.69}& 25.19  &\textbf{2.90e$-$3 }&\textbf{10:24}&5.8e$-$1 \\
{}& \texttt{InLP} &\textbf{22.69}& \textbf{24.81}  &2.30e$-$2 &12:06&8.6e$-$1 \\

\rowcolor{LightCyan}
splice & \texttt{GRLPN}&21.25&8.06  &\textbf{1.88e$-$4} & \textbf{16:07}  &1.50e$-$2\\
\rowcolor{LightCyan}
{}& \texttt{InLP} &\textbf{21.09}&\textbf{7.22}  &1.20e$-$2 & 57:27  &1.50e$-$2\\
german.number & \texttt{GRLPN} & \textbf{25.16}&16.11  & \textbf{1.44e$-$8} & \textbf{25:57}&14.17 \\
{} &\texttt{InLP} & \textbf{25.16}&\textbf{13.61}  & 1.97e$-$2  &39:19 &14.98\\
\bottomrule
\end{tabular}
\end{table}
\begin{figure}[htbp]
\centering
\subfigure[\texttt{GRLPN: fourclass}]{
\includegraphics[scale=0.23]{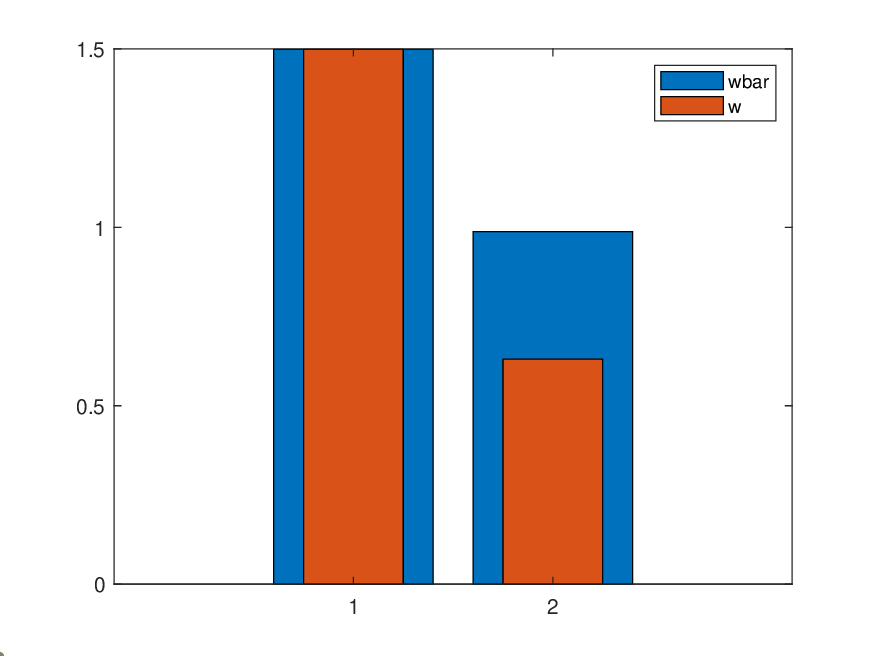} 
}
\quad
\subfigure[\texttt{GRLPN: nonskin}]{
\includegraphics[scale=0.23]{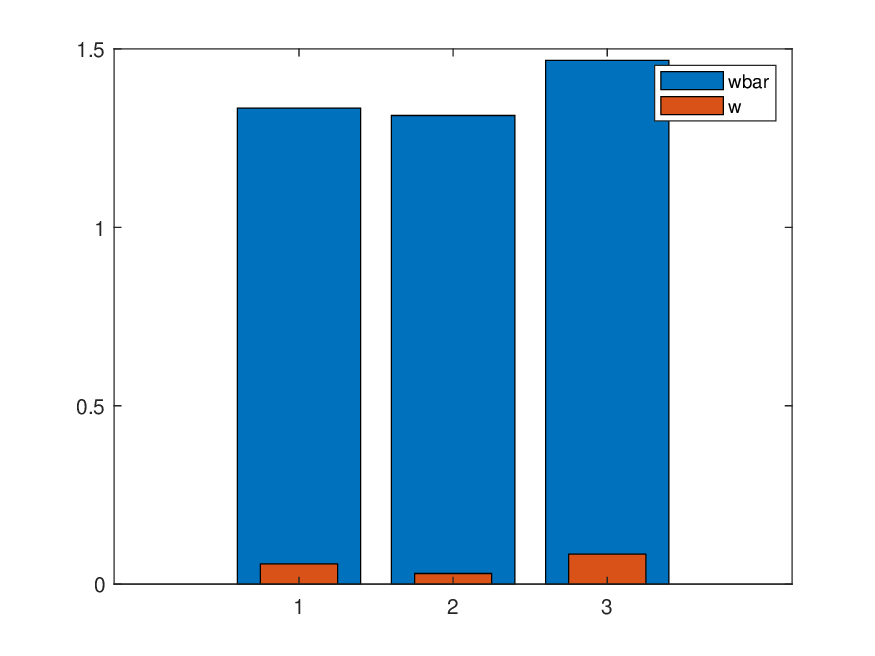}  
}
\quad
\subfigure[\texttt{GRLPN: svmguide1}]{
\includegraphics[scale=0.23]{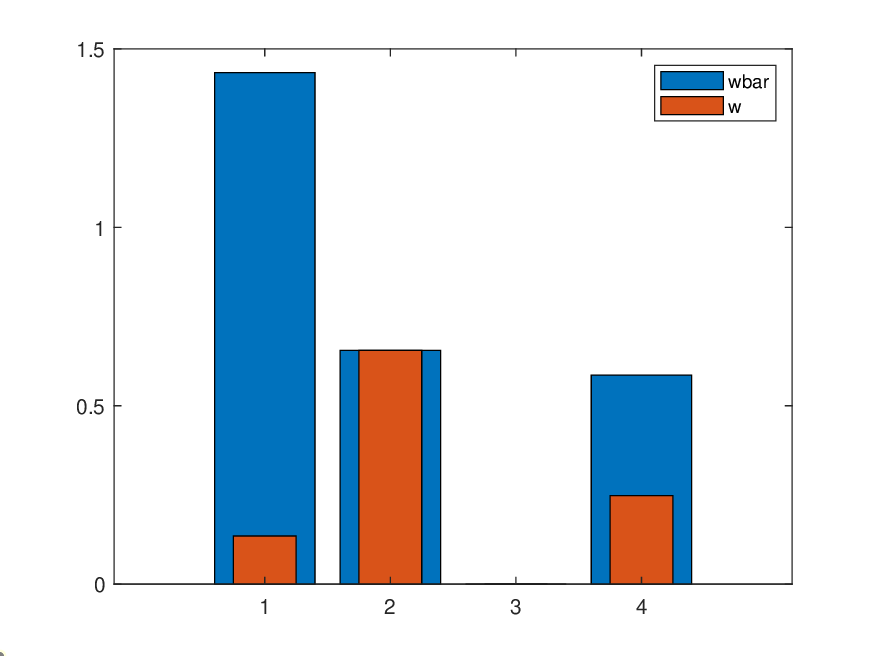}  
}
\quad
\subfigure[\texttt{GRLPN: liver}]{
\includegraphics[scale=0.23]{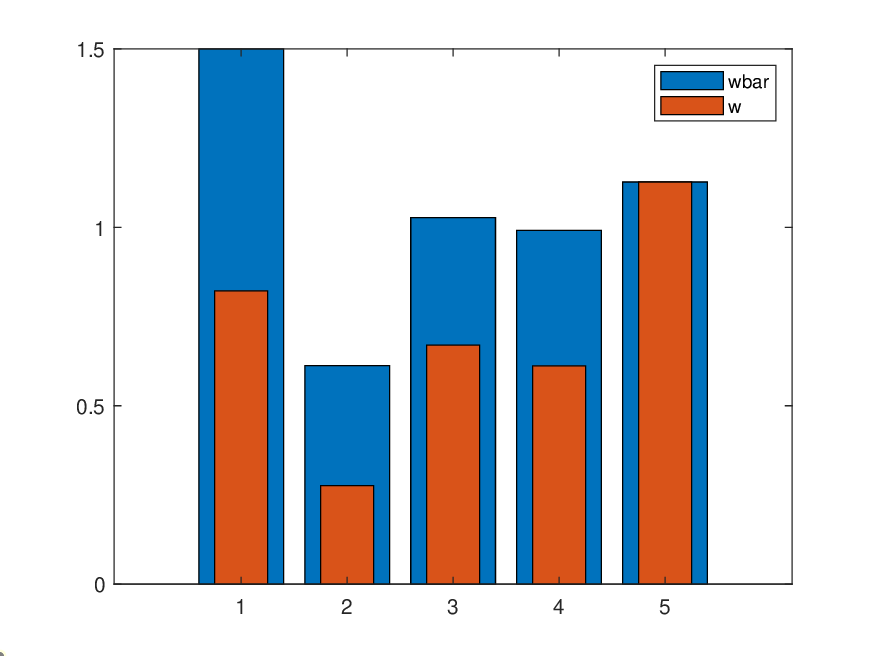} 
}
\subfigure[\texttt{GS: fourclass}]{
\includegraphics[scale=0.23]{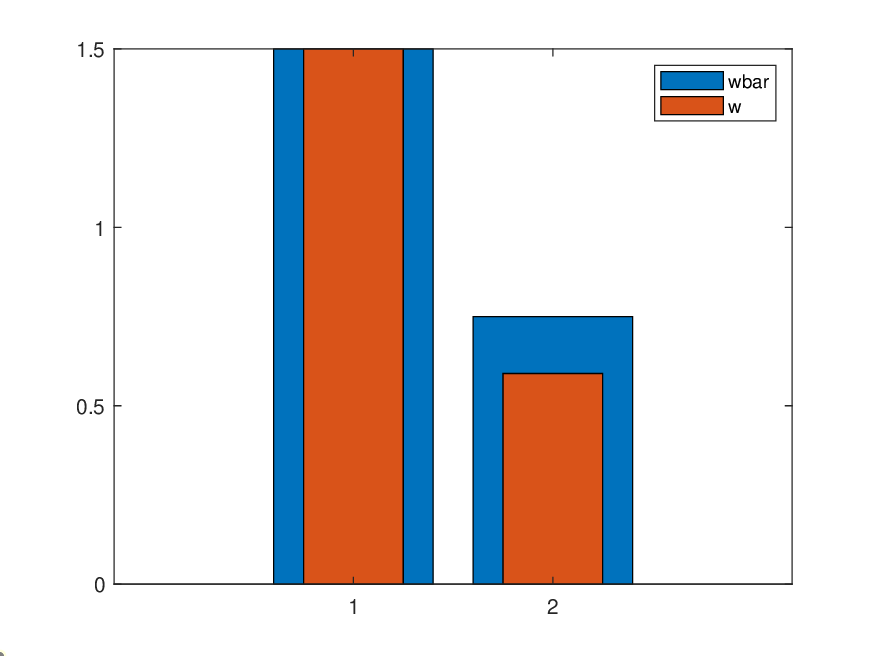} 
}
\quad
\subfigure[\texttt{GS: nonskin}]{
\includegraphics[scale=0.23]{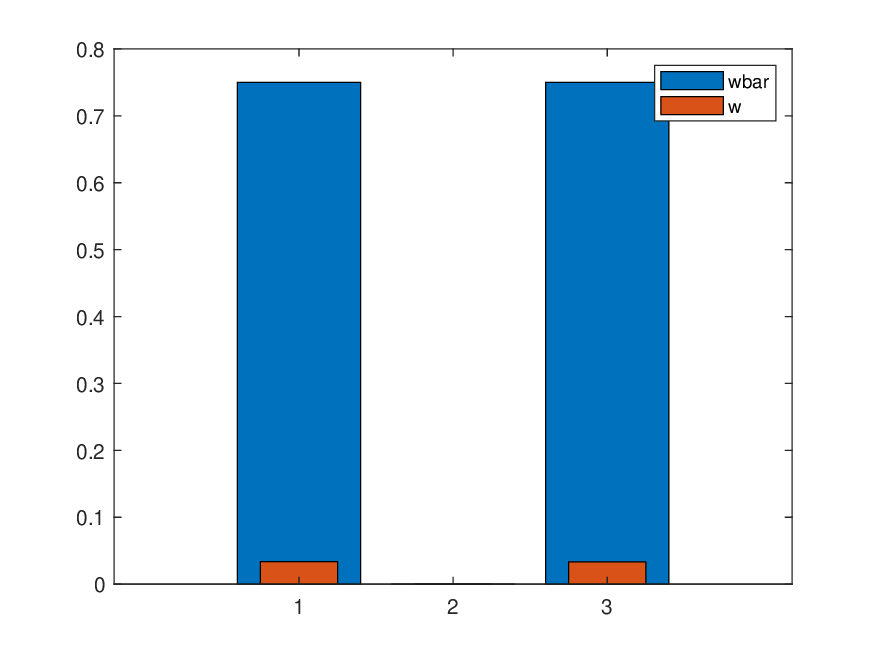}  
}
\quad
\subfigure[\texttt{GS: svmguide1}]{
\includegraphics[scale=0.23]{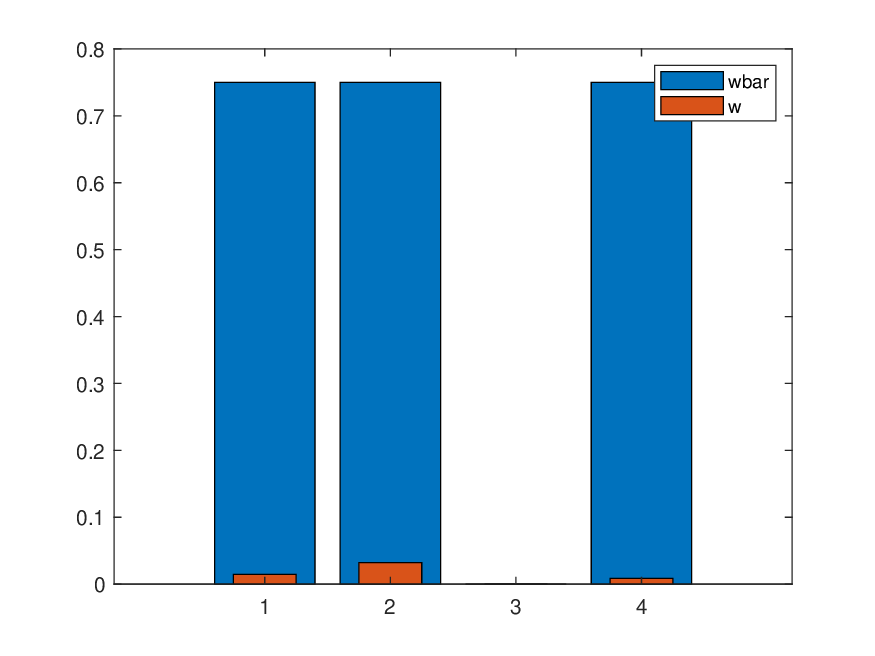}  
}
\quad
\subfigure[\texttt{GS: liver}]{
\includegraphics[scale=0.23]{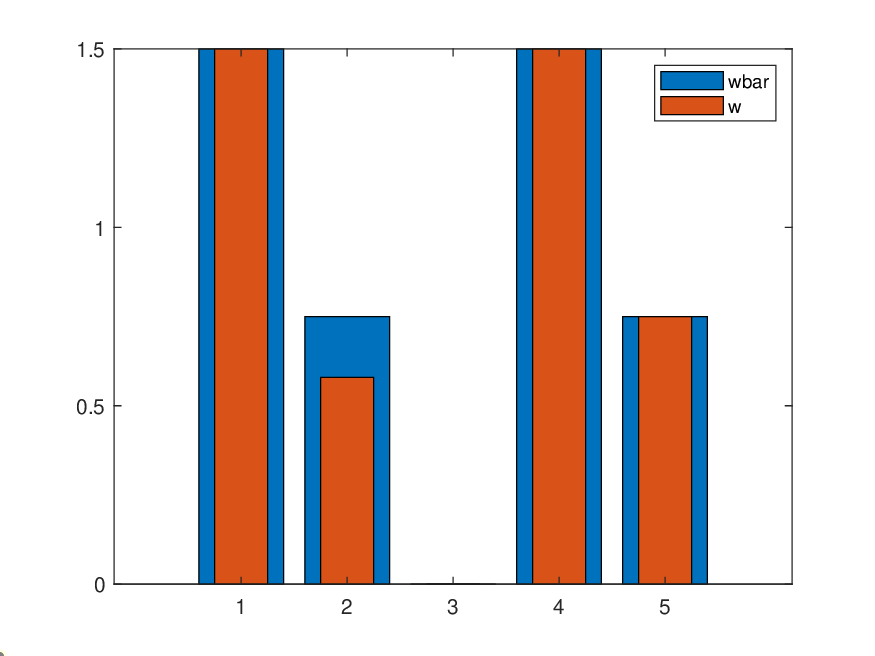} 
}

\subfigure[\texttt{GRST: fourclass}]{
\includegraphics[scale=0.23]{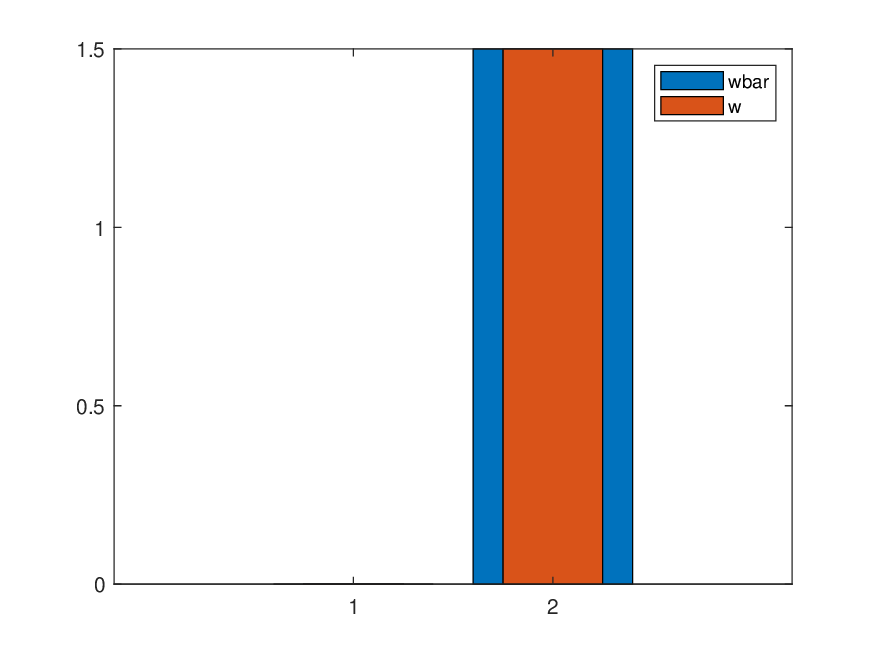} 
}
\quad
\subfigure[\texttt{GRST: nonskin}]{
\includegraphics[scale=0.23]{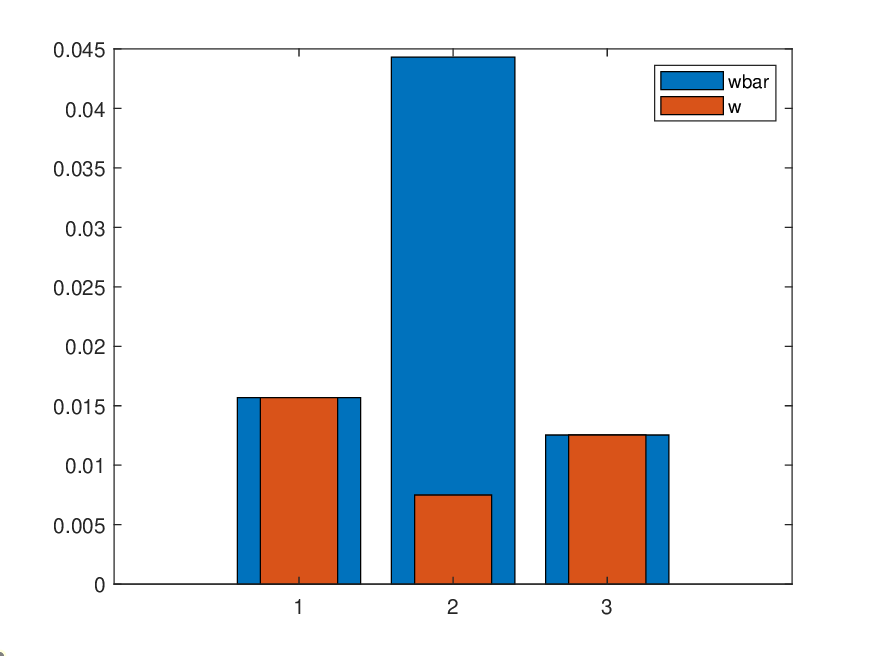}  
}
\quad
\subfigure[\texttt{GRST: svmguide1}]{
\includegraphics[scale=0.23]{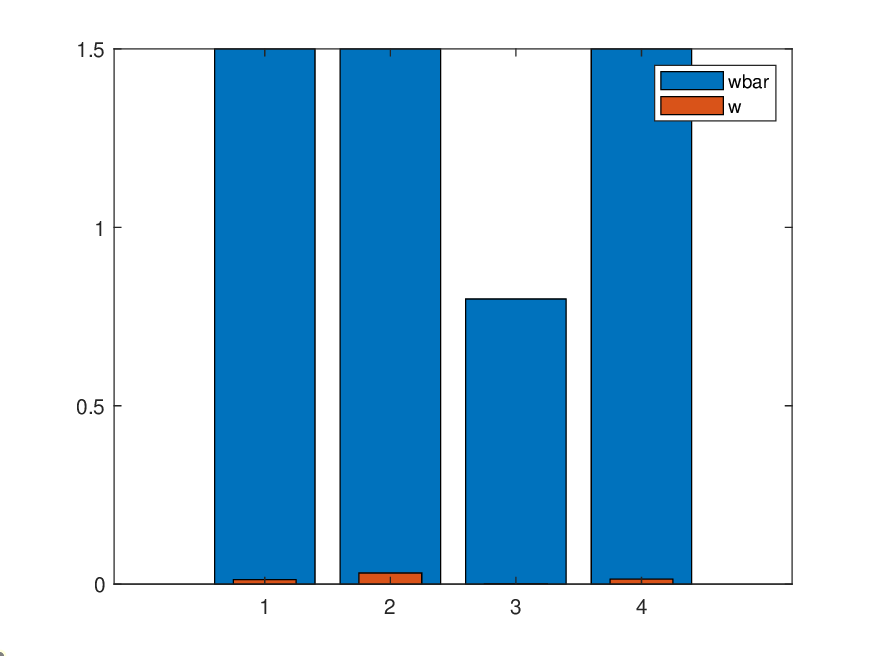}  
}
\quad
\subfigure[\texttt{GRST: liver}]{
\includegraphics[scale=0.23]{livergs.eps} 
}
\caption{Feature selection results illustrated componentwise for the datasets \texttt{fourclass}, \texttt{nonskin}, \texttt{svmguide1}, and \texttt{liver}}\label{Fig1}
\end{figure}

\begin{figure}[htbp]
\centering

\subfigure[\texttt{GRLPN: diabetes}]{
\includegraphics[scale=0.23]{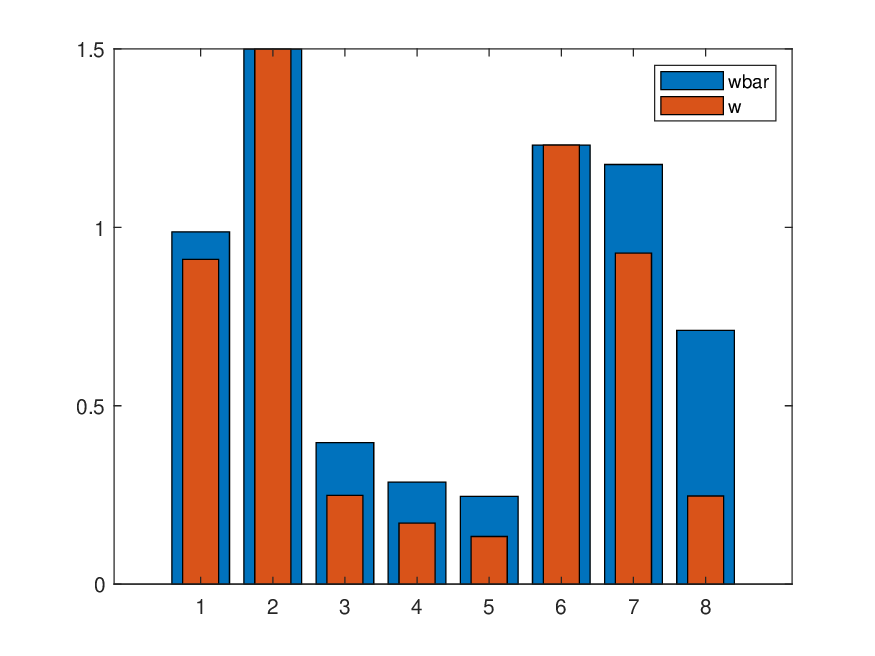}  
}
\quad
\subfigure[\texttt{GRLPN: breast}]{
\includegraphics[scale=0.23]{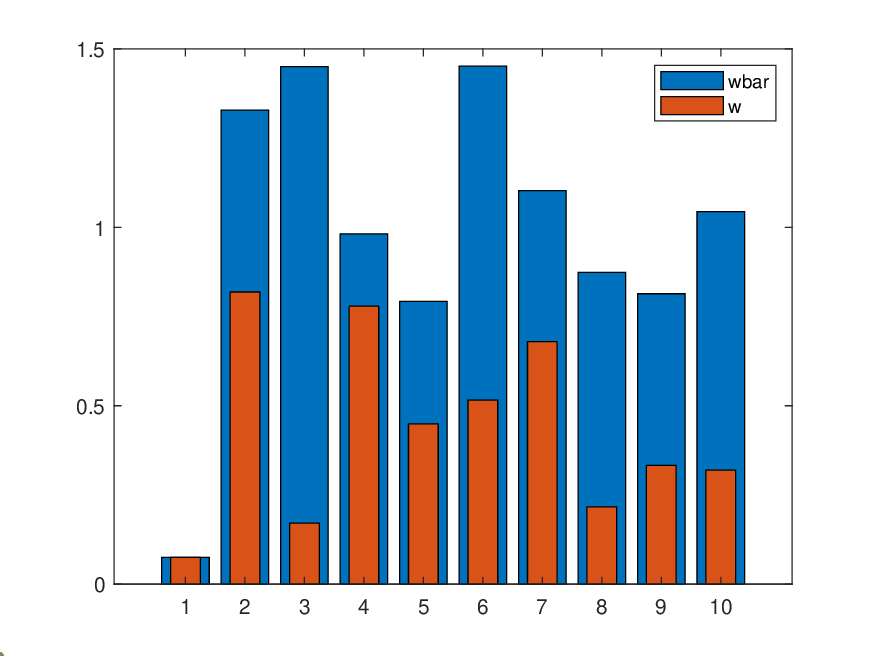}  
}
\quad
\subfigure[\texttt{GRLPN: heart}]{
\includegraphics[scale=0.23]{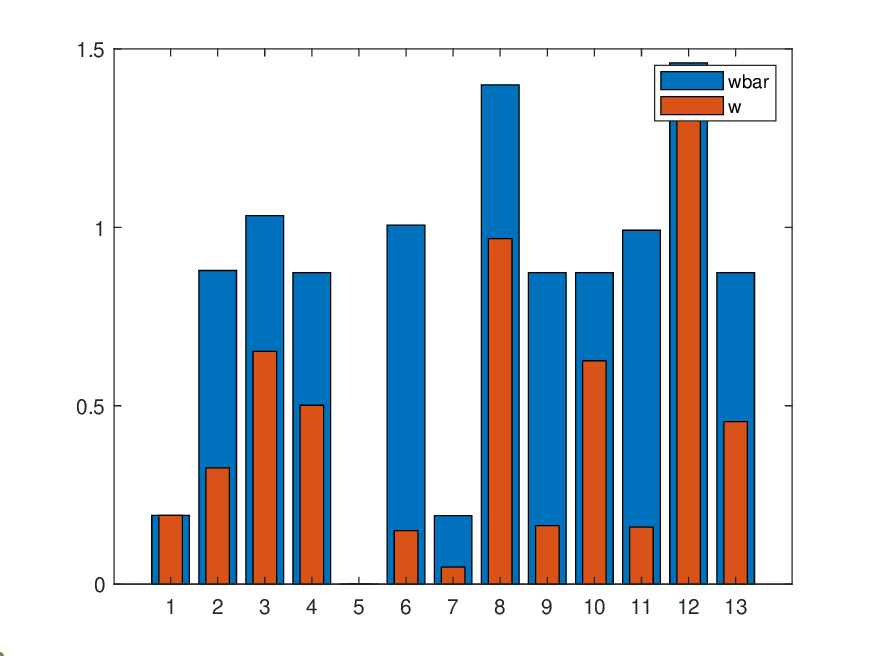} 
}
\quad
\subfigure[\texttt{GRLPN: australian}]{
\includegraphics[scale=0.23]{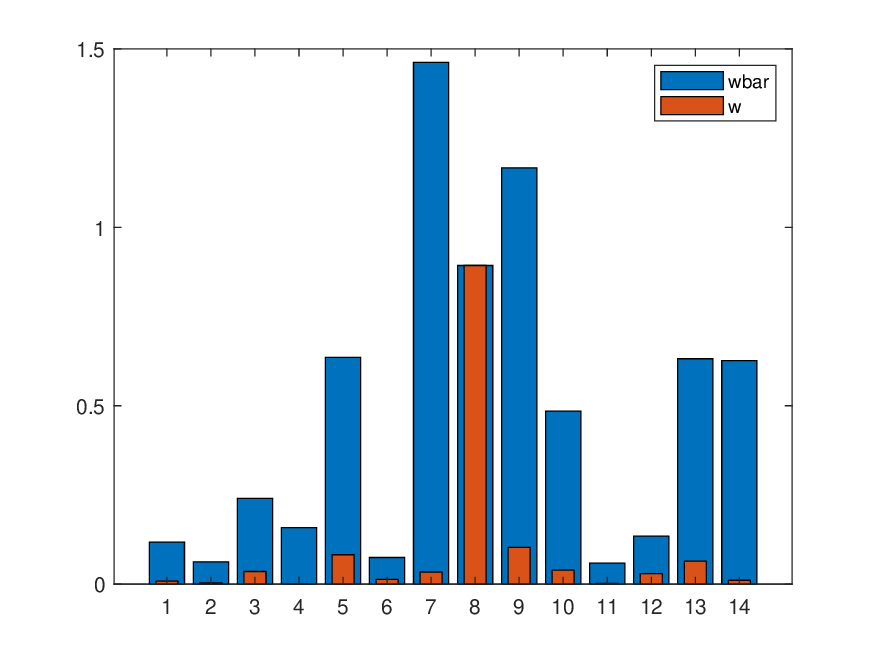}
}
\quad
\subfigure[\texttt{GS: diabetes}]{
\includegraphics[scale=0.23]{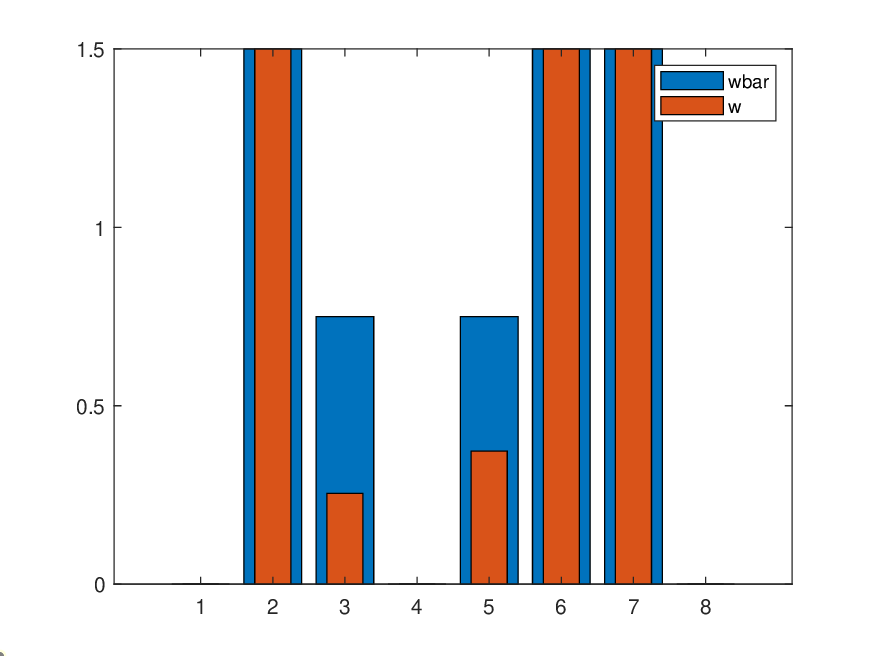}  
}
\quad
\subfigure[\texttt{GS: breast}]{
\includegraphics[scale=0.23]{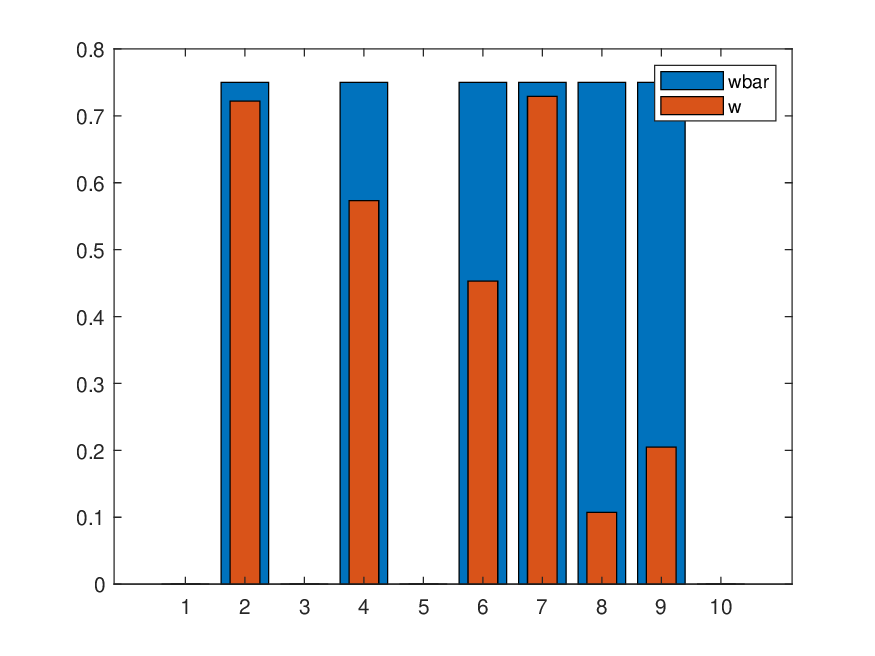}  
}
\quad
\subfigure[\texttt{GS: heart}]{
\includegraphics[scale=0.23]{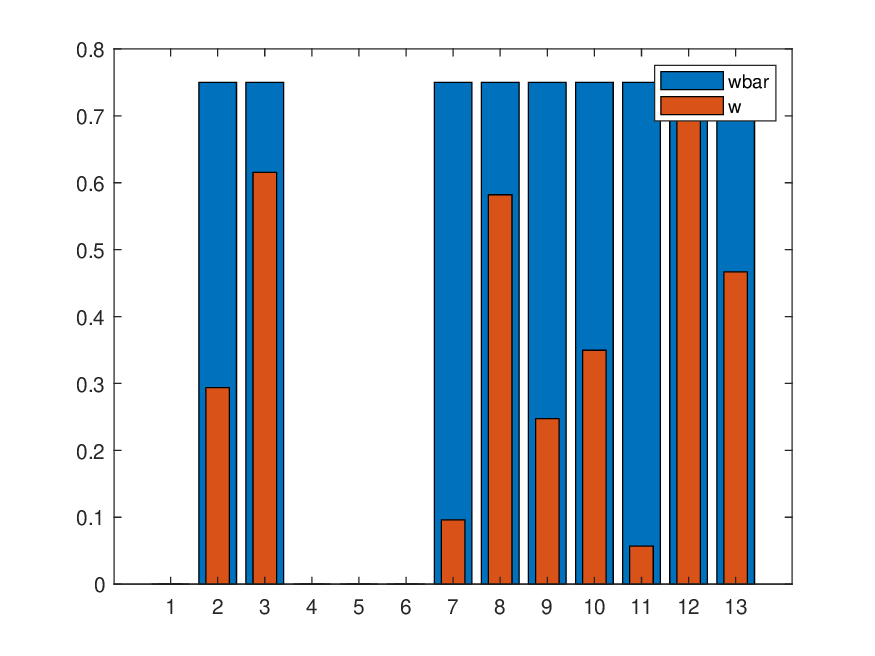} 
}
\quad
\subfigure[\texttt{GS: australian}]{
\includegraphics[scale=0.23]{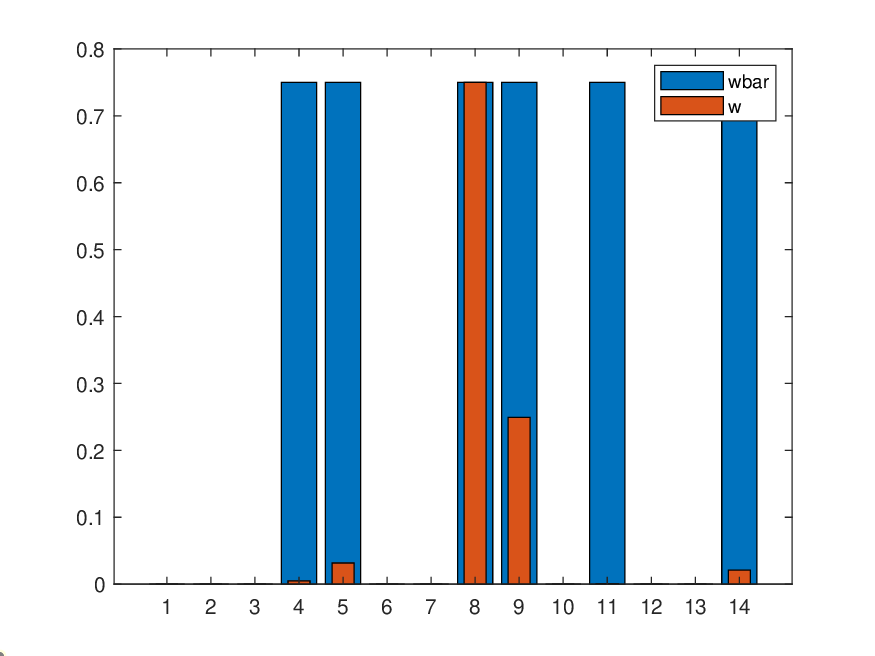} 
}
\quad
\subfigure[\texttt{GRST: diabetes}]{
\includegraphics[scale=0.23]{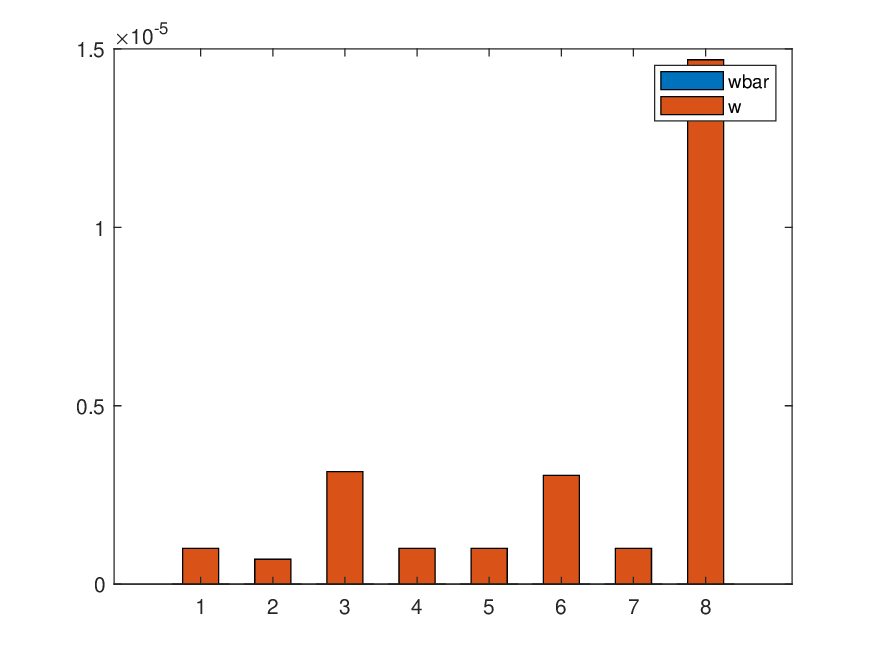}  
}
\quad
\subfigure[\texttt{GRST: breast}]{
\includegraphics[scale=0.23]{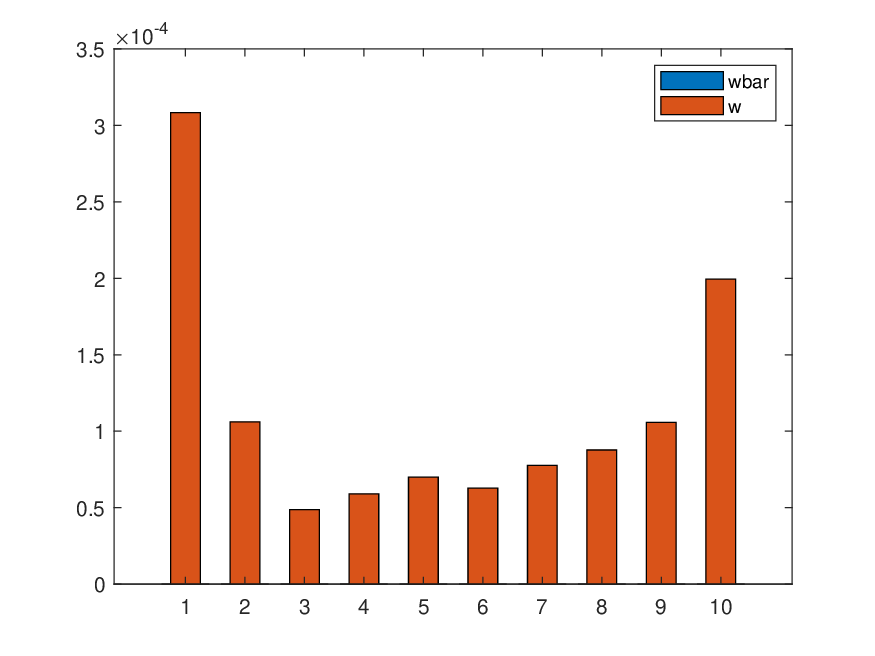}  
}
\quad
\subfigure[\texttt{GRST: heart}]{
\includegraphics[scale=0.23]{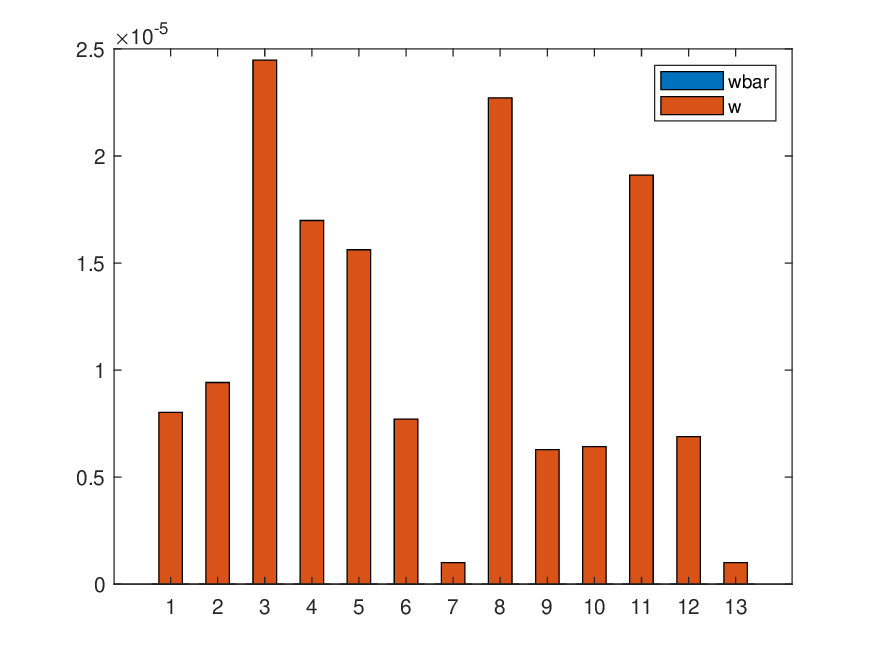} 
}
\quad
\subfigure[\texttt{GRST: australian}]{
\includegraphics[scale=0.23]{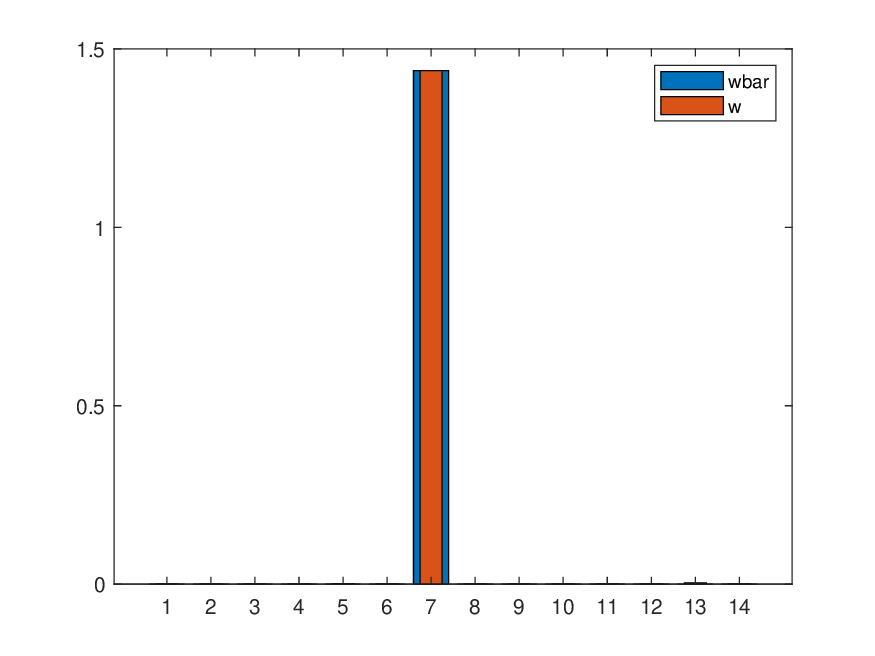} 
}
\caption{Feature selection results illustrated componentwise for the datasets \texttt{diabetes}, \texttt{breast}, \texttt{heart}, and \texttt{Australian}}\label{Fig2}
\end{figure}


In terms of computational efficiency, the \texttt{GS} method demonstrates superior performance compared to \texttt{GRLPN} for small-scale problems ($n \leq 8$). However, when addressing large-scale problems ($n > 8$), \texttt{GRLPN} significantly outperforms \texttt{GS}. Specifically, the computational complexity of \texttt{GS} grows rapidly, as it requires solving approximately $\mathcal{O}(T 3^{m_1})$ subproblems for selecting $\overline{w}$. Consequently, \texttt{GS} becomes computationally impractical as the number of features increases, and thus experiments involving large-scale datasets have been omitted for \texttt{GS}. In contrast, \texttt{GRLPN} does not suffer from this limitation. A similar pattern emerges in comparing computational time between \texttt{GRLPN} and \texttt{GRST}: for larger datasets (e.g., \texttt{splice}, \texttt{a1a}, and \texttt{w1a}), \texttt{GRLPN} consistently achieves significantly faster performance. With respect to the test error, \texttt{GRLPN} attains results superior or comparable to \texttt{GS}, and clearly outperforms \texttt{GRST}. Additionally, in terms of cross-validation error, \texttt{GRLPN} remains highly competitive against \texttt{GS} and considerably surpasses the performance obtained by \texttt{GRST}.


Furthermore, Figures~\ref{Fig1} and~\ref{Fig2} illustrate the results of the feature selection procedure, where the blue bars denote the value of $\overline{w}_i,\ i\in[n]$ and the red bars represent the value of $w_i,\ i\in[n]$ on the test datasets. Each bar's magnitude corresponds to the importance assigned to the respective feature. A red bar truncated by a blue bar indicates that the associated feature has been effectively removed by the selection mechanism. For \texttt{GRLPN}, features corresponding to values of $\overline{w}$ below the threshold $\sqrt{{\tau}_k}$ are deemed irrelevant and subsequently set to zero, after which the test error is evaluated. From the figures, it is evident that \texttt{GS} tends to yield sparser solutions, primarily due to its exhaustive search strategy, which aggressively explores feature combinations in pursuit of an optimal subset. Nevertheless, \texttt{GRLPN} achieves a more favorable balance between sparsity and generalization performance. Additionally, it is worth noting that the final values of $C$ and $\overline{w}$ generated by SNOPT are frequently obtained near the lower bounds and are highly sensitive to the variations in the lower bounds.


Next, to demonstrate the influence of the parameter ${\tau}_k$, we fix it directly at a small constant value (${\tau}_k = 1.0\times 10^{-4}$) and solve problem (\ref{NLP}) only once. The resulting approach is denoted by \texttt{InLP}. The experimental settings for \texttt{GRLPN} remain unchanged from previous experiments. The numerical results for this comparison are summarized in Table~\ref{table3}. As can be observed from the table, \texttt{GRLPN} generally requires less computational time and achieves significantly better (smaller) maximum violation values compared to the \texttt{InLP} method.

\section{Conclusions}\label{sec5}
In this paper, we have developed a global relaxation-based LP-Newton (GRLPN) method for simultaneously addressing multiple hyperparameter selection in SVC. We have established that the MPEC-MFCQ property is preserved within our formulation and provided theoretical convergence guarantees for the proposed GRLPN method. Numerical experiments illustrate the efficiency and practical advantages of GRLPN, particularly for datasets involving a larger number of features.

\section*{Declarations}
\subsection*{Ethics approval}
Not applicable.
\subsection*{Consent to participate}
All the authors have consented to participate in this research.

\subsection*{Consent for publication}
All the authors have given their consent for the paper to be submitted for publication.

\subsection*{Availability of data and material}
The codes and data use for the experiments can be made available if requested.  

\subsection*{Funding and acknowledgements}
The work of QL's research is supported by the National Science Foundation of China (NSFC) 12071032, 12271526. AZ's work is partly funded by the EPSRC grants with references EP/V049038/1 and EP/X040909/1.

\subsection*{Author Contribution declaration}
Yaru Qian conceptualized the study, developed the methodology, conducted theoretical analysis, performed numerical experiments, and drafted the manuscript. Qingna Li contributed to the conceptualization, supervised the theoretical analysis, and revised the manuscript. Alain Zemkoho supervised numerical experiments, provided critical feedback, and reviewed the manuscript. All authors discussed results and approved the final manuscript.

\subsection*{Competing Interests}
The authors declare that they have no competing interests.

\bibliography{paperla}


\begin{thebibliography}{55}
\ifx \bisbn   \undefined \def \bisbn  #1{ISBN #1}\fi
\ifx \binits  \undefined \def \binits#1{#1}\fi
\ifx \bauthor  \undefined \def \bauthor#1{#1}\fi
\ifx \batitle  \undefined \def \batitle#1{#1}\fi
\ifx \bjtitle  \undefined \def \bjtitle#1{#1}\fi
\ifx \bvolume  \undefined \def \bvolume#1{\textbf{#1}}\fi
\ifx \byear  \undefined \def \byear#1{#1}\fi
\ifx \bissue  \undefined \def \bissue#1{#1}\fi
\ifx \bfpage  \undefined \def \bfpage#1{#1}\fi
\ifx \blpage  \undefined \def \blpage #1{#1}\fi
\ifx \burl  \undefined \def \burl#1{\textsf{#1}}\fi
\ifx \doiurl  \undefined \def \doiurl#1{\url{https://doi.org/#1}}\fi
\ifx \betal  \undefined \def \betal{\textit{et al.}}\fi
\ifx \binstitute  \undefined \def \binstitute#1{#1}\fi
\ifx \binstitutionaled  \undefined \def \binstitutionaled#1{#1}\fi
\ifx \bctitle  \undefined \def \bctitle#1{#1}\fi
\ifx \beditor  \undefined \def \beditor#1{#1}\fi
\ifx \bpublisher  \undefined \def \bpublisher#1{#1}\fi
\ifx \bbtitle  \undefined \def \bbtitle#1{#1}\fi
\ifx \bedition  \undefined \def \bedition#1{#1}\fi
\ifx \bseriesno  \undefined \def \bseriesno#1{#1}\fi
\ifx \blocation  \undefined \def \blocation#1{#1}\fi
\ifx \bsertitle  \undefined \def \bsertitle#1{#1}\fi
\ifx \bsnm \undefined \def \bsnm#1{#1}\fi
\ifx \bsuffix \undefined \def \bsuffix#1{#1}\fi
\ifx \bparticle \undefined \def \bparticle#1{#1}\fi
\ifx \barticle \undefined \def \barticle#1{#1}\fi
\bibcommenthead
\ifx \bconfdate \undefined \def \bconfdate #1{#1}\fi
\ifx \botherref \undefined \def \botherref #1{#1}\fi
\ifx \url \undefined \def \url#1{\textsf{#1}}\fi
\ifx \bchapter \undefined \def \bchapter#1{#1}\fi
\ifx \bbook \undefined \def \bbook#1{#1}\fi
\ifx \bcomment \undefined \def \bcomment#1{#1}\fi
\ifx \oauthor \undefined \def \oauthor#1{#1}\fi
\ifx \citeauthoryear \undefined \def \citeauthoryear#1{#1}\fi
\ifx \endbibitem  \undefined \def \endbibitem {}\fi
\ifx \bconflocation  \undefined \def \bconflocation#1{#1}\fi
\ifx \arxivurl  \undefined \def \arxivurl#1{\textsf{#1}}\fi
\csname PreBibitemsHook\endcsname

\bibitem{de2014fault}
\begin{barticle}
\bauthor{\bparticle{de} \bsnm{Souza}, \binits{D.L.}},
\bauthor{\bsnm{Granzotto}, \binits{M.H.}},
\bauthor{\bparticle{de} \bsnm{Almeida}, \binits{G.M.}},
\bauthor{\bsnm{Oliveira-Lopes}, \binits{L.C.}}:
\batitle{Fault detection and diagnosis using support vector machines-a svc and svr comparison}.
\bjtitle{Journal of Safety Engineering}
\bvolume{3}(\bissue{1}),
\bfpage{18}--\blpage{29}
(\byear{2014})
\end{barticle}
\endbibitem

\bibitem{susto2013predictive}
\begin{bchapter}
\bauthor{\bsnm{Susto}, \binits{G.A.}},
\bauthor{\bsnm{Schirru}, \binits{A.}},
\bauthor{\bsnm{Pampuri}, \binits{S.}},
\bauthor{\bsnm{Pagano}, \binits{D.}},
\bauthor{\bsnm{McLoone}, \binits{S.}},
\bauthor{\bsnm{Beghi}, \binits{A.}}:
\bctitle{A predictive maintenance system for integral type faults based on support vector machines: An application to ion implantation}.
In: \bbtitle{2013 IEEE International Conference on Automation Science and Engineering (CASE)},
pp. \bfpage{195}--\blpage{200}
(\byear{2013}).
\bcomment{IEEE}
\end{bchapter}
\endbibitem

\bibitem{ahmad2015application}
\begin{barticle}
\bauthor{\bsnm{Ahmad}, \binits{I.}},
\bauthor{\bsnm{Jeenanunta}, \binits{C.}}:
\batitle{Application of support vector classification algorithms for the prediction of quality level of frozen shrimps (litopenaeus vannamei) suitable for sensor-based time-temperature monitoring}.
\bjtitle{Food and Bioprocess Technology}
\bvolume{8},
\bfpage{134}--\blpage{147}
(\byear{2015})
\end{barticle}
\endbibitem

\bibitem{liu2009obscure}
\begin{barticle}
\bauthor{\bsnm{Liu}, \binits{J.}},
\bauthor{\bsnm{Yuan}, \binits{X.}}:
\batitle{Obscure bleeding detection in endoscopy images using support vector machines}.
\bjtitle{Optimization and Engineering}
\bvolume{10}(\bissue{2}),
\bfpage{289}--\blpage{299}
(\byear{2009})
\end{barticle}
\endbibitem

\bibitem{heikamp2014support}
\begin{barticle}
\bauthor{\bsnm{Heikamp}, \binits{K.}},
\bauthor{\bsnm{Bajorath}, \binits{J.}}:
\batitle{Support vector machines for drug discovery}.
\bjtitle{Expert Opinion on Drug Discovery}
\bvolume{9}(\bissue{1}),
\bfpage{93}--\blpage{104}
(\byear{2014})
\end{barticle}
\endbibitem

\bibitem{deshpande2016performance}
\begin{bchapter}
\bauthor{\bsnm{Deshpande}, \binits{M.}},
\bauthor{\bsnm{Bajaj}, \binits{P.R.}}:
\bctitle{Performance analysis of support vector machine for traffic flow prediction}.
In: \bbtitle{2016 International Conference on Global Trends in Signal Processing, Information Computing and Communication (ICGTSPICC)},
pp. \bfpage{126}--\blpage{129}
(\byear{2016}).
\bcomment{IEEE}
\end{bchapter}
\endbibitem

\bibitem{harirchian2020application}
\begin{barticle}
\bauthor{\bsnm{Harirchian}, \binits{E.}},
\bauthor{\bsnm{Lahmer}, \binits{T.}},
\bauthor{\bsnm{Kumari}, \binits{V.}},
\bauthor{\bsnm{Jadhav}, \binits{K.}}:
\batitle{Application of support vector machine modeling for the rapid seismic hazard safety evaluation of existing buildings}.
\bjtitle{Energies}
\bvolume{13}(\bissue{13}),
\bfpage{3340}
(\byear{2020})
\end{barticle}
\endbibitem

\bibitem{liu2016comprehensive}
\begin{barticle}
\bauthor{\bsnm{Liu}, \binits{Y.}},
\bauthor{\bsnm{Wang}, \binits{H.}},
\bauthor{\bsnm{Zhang}, \binits{H.}},
\bauthor{\bsnm{Liber}, \binits{K.}}:
\batitle{A comprehensive support vector machine-based classification model for soil quality assessment}.
\bjtitle{Soil and Tillage Research}
\bvolume{155},
\bfpage{19}--\blpage{26}
(\byear{2016})
\end{barticle}
\endbibitem

\bibitem{hadjidemetriou2018automated}
\begin{barticle}
\bauthor{\bsnm{Hadjidemetriou}, \binits{G.M.}},
\bauthor{\bsnm{Vela}, \binits{P.A.}},
\bauthor{\bsnm{Christodoulou}, \binits{S.E.}}:
\batitle{{\small{Automated pavement patch detection and quantification using support vector machines}}}.
\bjtitle{{\small{Journal of Computing in Civil Engineering}}}
\bvolume{32}(\bissue{1}),
\bfpage{04017073}
(\byear{{\small{2018}}})
\end{barticle}
\endbibitem

\bibitem{jozdani2019comparing}
\begin{barticle}
\bauthor{\bsnm{Jozdani}, \binits{S.E.}},
\bauthor{\bsnm{Johnson}, \binits{B.A.}},
\bauthor{\bsnm{Chen}, \binits{D.}}:
\batitle{{\small{Comparing deep neural networks, ensemble classifiers, and support vector machine algorithms for object-based urban land use/land cover classification}}}.
\bjtitle{Remote Sensing}
\bvolume{11}(\bissue{14}),
\bfpage{1713}
(\byear{2019})
\end{barticle}
\endbibitem

\bibitem{oskoei2008support}
\begin{barticle}
\bauthor{\bsnm{Oskoei}, \binits{M.A.}},
\bauthor{\bsnm{Hu}, \binits{H.}}:
\batitle{Support vector machine-based classification scheme for myoelectric control applied to upper limb}.
\bjtitle{IEEE Transactions on Biomedical Engineering}
\bvolume{55}(\bissue{8}),
\bfpage{1956}--\blpage{1965}
(\byear{2008})
\end{barticle}
\endbibitem

\bibitem{xiao2019leak}
\begin{barticle}
\bauthor{\bsnm{Xiao}, \binits{R.}},
\bauthor{\bsnm{Hu}, \binits{Q.}},
\bauthor{\bsnm{Li}, \binits{J.}}:
\batitle{Leak detection of gas pipelines using acoustic signals based on wavelet transform and support vector machine}.
\bjtitle{Measurement}
\bvolume{146},
\bfpage{479}--\blpage{489}
(\byear{2019})
\end{barticle}
\endbibitem

\bibitem{diao2021structural}
\begin{barticle}
\bauthor{\bsnm{Diao}, \binits{Y.}},
\bauthor{\bsnm{Jia}, \binits{D.}},
\bauthor{\bsnm{Liu}, \binits{G.}},
\bauthor{\bsnm{Sun}, \binits{Z.}},
\bauthor{\bsnm{Xu}, \binits{J.}}:
\batitle{Structural damage identification using modified hilbert--huang transform and support vector machine}.
\bjtitle{Journal of Civil Structural Health Monitoring}
\bvolume{11},
\bfpage{1155}--\blpage{1174}
(\byear{2021})
\end{barticle}
\endbibitem

\bibitem{chen2017remaining}
\begin{barticle}
\bauthor{\bsnm{Chen}, \binits{Z.}},
\bauthor{\bsnm{Cao}, \binits{S.}},
\bauthor{\bsnm{Mao}, \binits{Z.}}:
\batitle{Remaining useful life estimation of aircraft engines using a modified similarity and supporting vector machine ({SVM}) approach}.
\bjtitle{Energies}
\bvolume{11}(\bissue{1}),
\bfpage{28}
(\byear{2017})
\end{barticle}
\endbibitem

\bibitem{laouti2011support}
\begin{barticle}
\bauthor{\bsnm{Laouti}, \binits{N.}},
\bauthor{\bsnm{Sheibat-Othman}, \binits{N.}},
\bauthor{\bsnm{Othman}, \binits{S.}}:
\batitle{Support vector machines for fault detection in wind turbines}.
\bjtitle{IFAC Proceedings Volumes}
\bvolume{44}(\bissue{1}),
\bfpage{7067}--\blpage{7072}
(\byear{2011})
\end{barticle}
\endbibitem

\bibitem{wang2005support}
\begin{botherref}
\oauthor{\bsnm{Wang}, \binits{L.}}, et al.:
Support vector machines: Theory and applications.
Springer Engineering eBooks 2005 English/International
(2005)
\end{botherref}
\endbibitem

\bibitem{chapelle2002choosing}
\begin{barticle}
\bauthor{\bsnm{Chapelle}, \binits{O.}},
\bauthor{\bsnm{Vapnik}, \binits{V.}},
\bauthor{\bsnm{Bousquet}, \binits{O.}},
\bauthor{\bsnm{Mukherjee}, \binits{S.}}:
\batitle{Choosing multiple parameters for support vector machines}.
\bjtitle{Machine Learning}
\bvolume{46},
\bfpage{131}--\blpage{159}
(\byear{2002})
\end{barticle}
\endbibitem

\bibitem{dong2007mpec}
\begin{bchapter}
\bauthor{\bsnm{Dong}, \binits{Y.-L.}},
\bauthor{\bsnm{Xia}, \binits{Z.-Q.}},
\bauthor{\bsnm{Wang}, \binits{M.-Z.}}:
\bctitle{An mpec model for selecting optimal parameter in support vector machines}.
In: \bbtitle{The First International Symposium on Optimization and Systems Biology},
pp. \bfpage{351}--\blpage{357}
(\byear{2007})
\end{bchapter}
\endbibitem

\bibitem{duan2003evaluation}
\begin{barticle}
\bauthor{\bsnm{Duan}, \binits{K.}},
\bauthor{\bsnm{Keerthi}, \binits{S.S.}},
\bauthor{\bsnm{Poo}, \binits{A.N.}}:
\batitle{Evaluation of simple performance measures for tuning {SVM} hyperparameters}.
\bjtitle{Neurocomputing}
\bvolume{51},
\bfpage{41}--\blpage{59}
(\byear{2003})
\end{barticle}
\endbibitem

\bibitem{keerthi2006efficient}
\begin{botherref}
\oauthor{\bsnm{Keerthi}, \binits{S.}},
\oauthor{\bsnm{Sindhwani}, \binits{V.}},
\oauthor{\bsnm{Chapelle}, \binits{O.}}:
An efficient method for gradient-based adaptation of hyperparameters in {SVM} models.
Advances in Neural Information Processing Systems
\textbf{19}
(2006)
\end{botherref}
\endbibitem

\bibitem{kunapuli2008bilevel1}
\begin{bbook}
\bauthor{\bsnm{Kunapuli}, \binits{G.}}:
\bbtitle{A Bilevel Optimization Approach to Machine Learning}.
\bpublisher{Rensselaer Polytechnic Institute},
\blocation{New York}
(\byear{2008})
\end{bbook}
\endbibitem

\bibitem{kunapuli2008bilevel}
\begin{barticle}
\bauthor{\bsnm{Kunapuli}, \binits{G.}},
\bauthor{\bsnm{Bennett}, \binits{K.P.}},
\bauthor{\bsnm{Hu}, \binits{J.}},
\bauthor{\bsnm{Pang}, \binits{J.-S.}}:
\batitle{Bilevel model selection for support vector machines}.
\bjtitle{Data Mining and Mathematical Programming}
\bvolume{45},
\bfpage{129}--\blpage{158}
(\byear{2008})
\end{barticle}
\endbibitem

\bibitem{kunapuli2008classification}
\begin{barticle}
\bauthor{\bsnm{Kunapuli}, \binits{G.}},
\bauthor{\bsnm{Bennett}, \binits{K.P.}},
\bauthor{\bsnm{Hu}, \binits{J.}},
\bauthor{\bsnm{Pang}, \binits{J.-S.}}:
\batitle{Classification model selection via bilevel programming}.
\bjtitle{Optimization Methods \& Software}
\bvolume{23}(\bissue{4}),
\bfpage{475}--\blpage{489}
(\byear{2008})
\end{barticle}
\endbibitem

\bibitem{ben1990computational}
\begin{barticle}
\bauthor{\bsnm{Ben-Ayed}, \binits{O.}},
\bauthor{\bsnm{Blair}, \binits{C.E.}}:
\batitle{Computational difficulties of bilevel linear programming}.
\bjtitle{Operations Research}
\bvolume{38}(\bissue{3}),
\bfpage{556}--\blpage{560}
(\byear{1990})
\end{barticle}
\endbibitem

\bibitem{okuno2021lp}
\begin{barticle}
\bauthor{\bsnm{Okuno}, \binits{T.}},
\bauthor{\bsnm{Takeda}, \binits{A.}},
\bauthor{\bsnm{Kawana}, \binits{A.}},
\bauthor{\bsnm{Watanabe}, \binits{M.}}:
\batitle{On ℓp-hyperparameter learning via bilevel nonsmooth optimization}.
\bjtitle{The Journal of Machine Learning Research}
\bvolume{22}(\bissue{1}),
\bfpage{11093}--\blpage{11139}
(\byear{2021})
\end{barticle}
\endbibitem

\bibitem{franceschi2017bridge}
\begin{botherref}
\oauthor{\bsnm{Franceschi}, \binits{L.}},
\oauthor{\bsnm{Donini}, \binits{M.}},
\oauthor{\bsnm{Frasconi}, \binits{P.}},
\oauthor{\bsnm{Pontil}, \binits{M.}}:
A bridge between hyperparameter optimization and learning-to-learn.
arXiv preprint arXiv:1712.06283
(2017)
\end{botherref}
\endbibitem

\bibitem{franceschi2018bilevel}
\begin{bchapter}
\bauthor{\bsnm{Franceschi}, \binits{L.}},
\bauthor{\bsnm{Frasconi}, \binits{P.}},
\bauthor{\bsnm{Salzo}, \binits{S.}},
\bauthor{\bsnm{Grazzi}, \binits{R.}},
\bauthor{\bsnm{Pontil}, \binits{M.}}:
\bctitle{Bilevel programming for hyperparameter optimization and meta-learning}.
In: \bbtitle{International Conference on Machine Learning},
pp. \bfpage{1568}--\blpage{1577}
(\byear{2018})
\end{bchapter}
\endbibitem

\bibitem{shaban2019truncated}
\begin{bchapter}
\bauthor{\bsnm{Shaban}, \binits{A.}},
\bauthor{\bsnm{Cheng}, \binits{C.-A.}},
\bauthor{\bsnm{Hatch}, \binits{N.}},
\bauthor{\bsnm{Boots}, \binits{B.}}:
\bctitle{Truncated back-propagation for bilevel optimization}.
In: \bbtitle{The 22nd International Conference on Artificial Intelligence and Statistics},
pp. \bfpage{1723}--\blpage{1732}
(\byear{2019}).
\bcomment{PMLR}
\end{bchapter}
\endbibitem

\bibitem{foo2007efficient}
\begin{botherref}
\oauthor{\bsnm{Foo}, \binits{C.-s.}},
\oauthor{\bsnm{Ng}, \binits{A.}}, et al.:
Efficient multiple hyperparameter learning for log-linear models.
Advances in Neural Information Processing Systems
\textbf{20}
(2007)
\end{botherref}
\endbibitem

\bibitem{mackay2019self}
\begin{botherref}
\oauthor{\bsnm{MacKay}, \binits{M.}},
\oauthor{\bsnm{Vicol}, \binits{P.}},
\oauthor{\bsnm{Lorraine}, \binits{J.}},
\oauthor{\bsnm{Duvenaud}, \binits{D.}},
\oauthor{\bsnm{Grosse}, \binits{R.}}:
Self-tuning networks: Bilevel optimization of hyperparameters using structured best-response functions.
arXiv preprint arXiv:1903.03088
(2019)
\end{botherref}
\endbibitem

\bibitem{ye2010new}
\begin{barticle}
\bauthor{\bsnm{Ye}, \binits{J.J.}},
\bauthor{\bsnm{Zhu}, \binits{D.L.}}:
\batitle{New necessary optimality conditions for bilevel programs by combining the {MPEC} and value function approaches}.
\bjtitle{SIAM Journal on Optimization}
\bvolume{20}(\bissue{4}),
\bfpage{1885}--\blpage{1905}
(\byear{2010})
\end{barticle}
\endbibitem

\bibitem{ye1995optimality}
\begin{barticle}
\bauthor{\bsnm{Ye}, \binits{J.J.}},
\bauthor{\bsnm{Zhu}, \binits{D.}}:
\batitle{Optimality conditions for bilevel programming problems}.
\bjtitle{Optimization}
\bvolume{33}(\bissue{1}),
\bfpage{9}--\blpage{27}
(\byear{1995})
\end{barticle}
\endbibitem

\bibitem{bennett2006model}
\begin{bchapter}
\bauthor{\bsnm{Bennett}, \binits{K.P.}},
\bauthor{\bsnm{Hu}, \binits{J.}},
\bauthor{\bsnm{Ji}, \binits{X.}},
\bauthor{\bsnm{Kunapuli}, \binits{G.}},
\bauthor{\bsnm{Pang}, \binits{J.-S.}}:
\bctitle{Model selection via bilevel optimization}.
In: \bbtitle{The 2006 IEEE International Joint Conference on Neural Network Proceedings},
pp. \bfpage{1922}--\blpage{1929}
(\byear{2006}).
\bcomment{IEEE}
\end{bchapter}
\endbibitem

\bibitem{conigliobilevel}
\begin{botherref}
\oauthor{\bsnm{Coniglio}, \binits{S.}},
\oauthor{\bsnm{Dunn}, \binits{A.}},
\oauthor{\bsnm{Li}, \binits{Q.}},
\oauthor{\bsnm{Zemkoho}, \binits{A.}}:
Bilevel hyperparameter optimization for nonlinear support vector machines.
Optimization Online: {https://optimization-online.org/?p=24056},
1--78
(2023)
\end{botherref}
\endbibitem

\bibitem{wang2023fast}
\begin{botherref}
\oauthor{\bsnm{Wang}, \binits{Y.}},
\oauthor{\bsnm{Li}, \binits{Q.}}:
A fast smoothing newton method for bilevel hyperparameter optimization for svc with logistic loss.
arXiv preprint arXiv:2308.07734
(2023)
\end{botherref}
\endbibitem

\bibitem{fletcher2006local}
\begin{barticle}
\bauthor{\bsnm{Fletcher}, \binits{R.}},
\bauthor{\bsnm{Leyffer}, \binits{S.}},
\bauthor{\bsnm{Ralph}, \binits{D.}},
\bauthor{\bsnm{Scholtes}, \binits{S.}}:
\batitle{Local convergence of sqp methods for mathematical programs with equilibrium constraints}.
\bjtitle{SIAM Journal on Optimization}
\bvolume{17}(\bissue{1}),
\bfpage{259}--\blpage{286}
(\byear{2006})
\end{barticle}
\endbibitem

\bibitem{stein2012lifting}
\begin{barticle}
\bauthor{\bsnm{Stein}, \binits{O.}}:
\batitle{Lifting mathematical programs with complementarity constraints}.
\bjtitle{Mathematical Programming}
\bvolume{131},
\bfpage{71}--\blpage{94}
(\byear{2012})
\end{barticle}
\endbibitem

\bibitem{scholtes1999exact}
\begin{barticle}
\bauthor{\bsnm{Scholtes}, \binits{S.}},
\bauthor{\bsnm{St{\"o}hr}, \binits{M.}}:
\batitle{Exact penalization of mathematical programs with equilibrium constraints}.
\bjtitle{SIAM Journal on Control and Optimization}
\bvolume{37}(\bissue{2}),
\bfpage{617}--\blpage{652}
(\byear{1999})
\end{barticle}
\endbibitem

\bibitem{hoheisel2013theoretical}
\begin{barticle}
\bauthor{\bsnm{Hoheisel}, \binits{T.}},
\bauthor{\bsnm{Kanzow}, \binits{C.}},
\bauthor{\bsnm{Schwartz}, \binits{A.}}:
\batitle{{\small{Theoretical and numerical comparison of relaxation methods for mathematical programs with complementarity constraints}}}.
\bjtitle{Mathematical Programming}
\bvolume{137}(\bissue{1}),
\bfpage{257}--\blpage{288}
(\byear{2013})
\end{barticle}
\endbibitem

\bibitem{scholtes2001convergence}
\begin{barticle}
\bauthor{\bsnm{Scholtes}, \binits{S.}}:
\batitle{Convergence properties of a regularization scheme for mathematical programs with complementarity constraints}.
\bjtitle{SIAM Journal on Optimization}
\bvolume{11}(\bissue{4}),
\bfpage{918}--\blpage{936}
(\byear{2001})
\end{barticle}
\endbibitem

\bibitem{li2022bilevel}
\begin{barticle}
\bauthor{\bsnm{Li}, \binits{Q.}},
\bauthor{\bsnm{Li}, \binits{Z.}},
\bauthor{\bsnm{Zemkoho}, \binits{A.}}:
\batitle{Bilevel hyperparameter optimization for support vector classification: theoretical analysis and a solution method}.
\bjtitle{Mathematical Methods of Operations Research}
\bvolume{96},
\bfpage{315}--\blpage{350}
(\byear{2022})
\end{barticle}
\endbibitem

\bibitem{li2022unified}
\begin{bchapter}
\bauthor{\bsnm{Li}, \binits{Z.}},
\bauthor{\bsnm{Qian}, \binits{Y.}},
\bauthor{\bsnm{Li}, \binits{Q.}}:
\bctitle{A unified framework and a case study for hyperparameter selection in machine learning via bilevel optimization}.
In: \bbtitle{2022 5th International Conference on Data Science and Information Technology (DSIT)},
pp. \bfpage{1}--\blpage{8}
(\byear{2022}).
\bcomment{IEEE}
\end{bchapter}
\endbibitem

\bibitem{ye1997exact}
\begin{barticle}
\bauthor{\bsnm{Ye}, \binits{J.}},
\bauthor{\bsnm{Zhu}, \binits{D.}},
\bauthor{\bsnm{Zhu}, \binits{Q.J.}}:
\batitle{Exact penalization and necessary optimality conditions for generalized bilevel programming problems}.
\bjtitle{SIAM Journal on Optimization}
\bvolume{7}(\bissue{2}),
\bfpage{481}--\blpage{507}
(\byear{1997})
\end{barticle}
\endbibitem

\bibitem{jane2005necessary}
\begin{barticle}
\bauthor{\bsnm{Ye}, \binits{J.J.}}:
\batitle{Necessary and sufficient optimality conditions for mathematical programs with equilibrium constraints}.
\bjtitle{Journal of Mathematical Analysis and Applications}
\bvolume{307}(\bissue{1}),
\bfpage{350}--\blpage{369}
(\byear{2005})
\end{barticle}
\endbibitem

\bibitem{samuel2025mathematical}
\begin{botherref}
\oauthor{\bsnm{Ward}, \binits{S.}},
\oauthor{\bsnm{Zemkoho}, \binits{A.}},
\oauthor{\bsnm{Ahipasaoglu}, \binits{S.}}:
Mathematical programs with complementarity constraints and applications to hyperparameter tuning for nonlinear support vector machines.
Under Review
(2025)
\end{botherref}
\endbibitem

\bibitem{couellan2015bi}
\begin{barticle}
\bauthor{\bsnm{Couellan}, \binits{N.}},
\bauthor{\bsnm{Wang}, \binits{W.}}:
\batitle{Bi-level stochastic gradient for large scale support vector machine}.
\bjtitle{Neurocomputing}
\bvolume{153},
\bfpage{300}--\blpage{308}
(\byear{2015})
\end{barticle}
\endbibitem

\bibitem{facchinei2014lp}
\begin{barticle}
\bauthor{\bsnm{Facchinei}, \binits{F.}},
\bauthor{\bsnm{Fischer}, \binits{A.}},
\bauthor{\bsnm{Herrich}, \binits{M.}}:
\batitle{An {LP}-newton method: nonsmooth equations, kkt systems, and nonisolated solutions}.
\bjtitle{Mathematical Programming}
\bvolume{146},
\bfpage{1}--\blpage{36}
(\byear{2014})
\end{barticle}
\endbibitem

\bibitem{fischer2016globally}
\begin{barticle}
\bauthor{\bsnm{Fischer}, \binits{A.}},
\bauthor{\bsnm{Herrich}, \binits{M.}},
\bauthor{\bsnm{Izmailov}, \binits{A.F.}},
\bauthor{\bsnm{Solodov}, \binits{M.V.}}:
\batitle{A globally convergent {LP}-newton method}.
\bjtitle{SIAM Journal on Optimization}
\bvolume{26}(\bissue{4}),
\bfpage{2012}--\blpage{2033}
(\byear{2016})
\end{barticle}
\endbibitem

\bibitem{fischer2016convergence}
\begin{barticle}
\bauthor{\bsnm{Fischer}, \binits{A.}},
\bauthor{\bsnm{Herrich}, \binits{M.}},
\bauthor{\bsnm{Izmailov}, \binits{A.F.}},
\bauthor{\bsnm{Solodov}, \binits{M.V.}}:
\batitle{Convergence conditions for newton-type methods applied to complementarity systems with nonisolated solutions}.
\bjtitle{Computational Optimization and Applications}
\bvolume{63},
\bfpage{425}--\blpage{459}
(\byear{2016})
\end{barticle}
\endbibitem

\bibitem{galantai2012properties}
\begin{barticle}
\bauthor{\bsnm{Gal{\'a}ntai}, \binits{A.}}:
\batitle{Properties and construction of {NCP} functions}.
\bjtitle{Computational Optimization and Applications}
\bvolume{52},
\bfpage{805}--\blpage{824}
(\byear{2012})
\end{barticle}
\endbibitem

\bibitem{clarke1983nonsmooth}
\begin{bchapter}
\bauthor{\bsnm{Clarke}, \binits{F.H.}}:
\bctitle{Nonsmooth analysis and optimization}.
In: \bbtitle{Proceedings of the International Congress of Mathematicians},
vol. \bseriesno{5},
pp. \bfpage{847}--\blpage{853}
(\byear{1983})
\end{bchapter}
\endbibitem

\bibitem{flegel2005constraint}
\begin{botherref}
\oauthor{\bsnm{Flegel}, \binits{M.L.}}:
Constraint qualifications and stationarity concepts for mathematical programs with equilibrium constraints.
PhD thesis,
Universit{\"a}t W{\"u}rzburg
(2005)
\end{botherref}
\endbibitem

\bibitem{dempe2019two}
\begin{barticle}
\bauthor{\bsnm{Dempe}, \binits{S.}},
\bauthor{\bsnm{Mordukhovich}, \binits{B.S.}},
\bauthor{\bsnm{Zemkoho}, \binits{A.B.}}:
\batitle{Two-level value function approach to non-smooth optimistic and pessimistic bilevel programs}.
\bjtitle{Optimization}
\bvolume{68}(\bissue{2-3}),
\bfpage{433}--\blpage{455}
(\byear{2019})
\end{barticle}
\endbibitem

\bibitem{dempe2012karush}
\begin{barticle}
\bauthor{\bsnm{Dempe}, \binits{S.}},
\bauthor{\bsnm{Zemkoho}, \binits{A.B.}}:
\batitle{On the {K}arush--{K}uhn--{T}ucker reformulation of the bilevel optimization problem}.
\bjtitle{Nonlinear Analysis: Theory, Methods \& Applications}
\bvolume{75}(\bissue{3}),
\bfpage{1202}--\blpage{1218}
(\byear{2012})
\end{barticle}
\endbibitem

\bibitem{zemkoho2021theoretical}
\begin{barticle}
\bauthor{\bsnm{Zemkoho}, \binits{A.B.}},
\bauthor{\bsnm{Zhou}, \binits{S.}}:
\batitle{{\small{Theoretical and numerical comparison of the Karush--Kuhn--Tucker and value function reformulations in bilevel optimization}}}.
\bjtitle{Computational Optimization and Applications}
\bvolume{78}(\bissue{2}),
\bfpage{625}--\blpage{674}
(\byear{2021})
\end{barticle}
\endbibitem

\end{thebibliography}
\begin{appendices}
\section{Proof of Theorem \ref{thr1}}\label{appendixa}
Before presenting the proof, we first introduce some essential notation.

We use the following index sets to represent the validation and training points in K-fold cross validation:
\begin{equation*}
    Q_u:=\left\{i\mid i=1, \cdots, K m_1\right\},\quad Q_l:=\left\{i\mid i=1, \cdots, K m_2\right\}.
\end{equation*}
For the bound constrains in the lower level problems, we define the subsequent index set:
\begin{equation*}
    Q_b:=\left\{i\mid i=1, \cdots, K n\right\}.
\end{equation*}
To prove Theorem \ref{thr1}, we need the further definitions for index sets as follows:
\begin{subequations} \label{eq23}
\begin{align*}
& I_{H_{1}}\ \; := \; \left\{i \in Q_{u} \ \mid \ \zeta_{i}=0,\ \left(A B^{\top}\alpha+A\beta-A\gamma +z\right)_{i}>0\right\},  \\
& I_{H_{2}} \ \; :=\; \left\{i \in Q_{u} \ \mid \ z_{i}=0,\ 1-\zeta_{i}>0\right\},  \\
& I_{H_{3}} \ \; := \; \left\{i \in Q_{l} \ \mid \ \alpha_{i}=0,\ \left(B B^{\top}\alpha+B\beta-B\gamma -\mathbf{1}+\xi\right)_{i}>0\right\}, \\
& I_{H_{4}} \ \; := \; \left\{i \in Q_{l} \ \mid \ \xi_{i}=0,\ C-\alpha_{i}>0\right\}, \\
& I_{H_{5}} \ \; := \; \left\{i \in Q_b \ \mid \ \beta_{i}=0,\  \left(B^{\top}\alpha+\beta-\gamma\right)_i>-\overline{w}_{(i\ mod \ K)}\right\}, \\
& I_{H_{6}} \ \; := \; \left\{i \in Q_b \ \mid \ \gamma_{i}=0,\  \left(B^{\top}\alpha+\beta-\gamma\right)_i<\overline{w}_{(i\ mod \ K)}\right\}, \\
& I_{G_{1}} \ \; :=\; \left\{i \in Q_{u} \ \mid \ \zeta_{i}>0,\ \left(A B^{\top}\alpha+A\beta-A\gamma +z\right)_{i}=0\right\},\\
& I_{G_{2}} \ \; :=\; \left\{i \in Q_{u} \ \mid \ z_{i}>0,\ 1-\zeta_{i}=0\right\},\\
& I_{G_{3}} \ \; :=\; \left\{i \in Q_{l} \ \mid \ \alpha_{i}>0,\ \left(B B^{\top}\alpha+B\beta-B\gamma -\mathbf{1}+\xi\right)_{i}=0\right\},\\
& I_{G_{4}}\ \; :=\; \left\{i \in Q_{l} \ \mid \ \xi_{i}>0,\ C-\alpha_{i}=0\right\}, \\
& I_{G_{5}}\ \; :=\; \left\{i \in Q_b \ \mid \ \beta_{i}>0,\  \left(B^{\top}\alpha+\beta-\gamma\right)_i=-\overline{w}_{(i\ mod \ K)}\right\}, \\
& I_{G_{6}}\ \; :=\; \left\{i \in Q_b \ \mid \ \gamma_{i}>0,\  \left(B^{\top}\alpha+\beta-\gamma\right)_i=\overline{w}_{(i\ mod \ K)}\right\}, \\
& I_{GH_{1}} :=\; \left\{i \in Q_{u} \ \mid \ \zeta_{i}=0,\ \left(A B^{\top}\alpha+A\beta-A\gamma +z\right)_{i}=0\right\},\\
& I_{GH_{2}} :=\; \left\{i \in Q_{u} \ \mid \ z_{i}=0,\ 1-\zeta_{i}=0\right\}, \\
& I_{GH_{3}} :=\; \left\{i \in Q_{l} \ \mid \ \alpha_{i}=0,\ \left(B B^{\top}\alpha+B\beta-B\gamma -\mathbf{1}+\xi\right)_{i}=0\right\}, \\
& I_{GH_{4}} :=\; \left\{i \in Q_{l} \ \mid \ \xi_{i}=0,\ C-\alpha_{i}=0\right\},\\
& I_{GH_{5}}\ \; :=\; \left\{i \in Q_b \ \mid \ \beta_{i}=0,\  \left(B^{\top}\alpha+\beta-\gamma\right)_i=-\overline{w}_{(i\ mod \ K)}\right\}, \\
& I_{GH_{6}}\ \; :=\; \left\{i \in Q_b \ \mid \ \gamma_{i}=0,\  \left(B^{\top}\alpha+\beta-\gamma\right)_i=\overline{w}_{(i\ mod \ K)}\right\},\\
&I_{g_{1}}\ \; :=\; \left\{ 1\mid \text{if}\ C=C_{lb} \right\},\\
&I_{g_{2}}\ \; :=\; \left\{ 1\mid \text{if}\
 C=C_{ub}\right\},\\
&I_{g_{3}}\ \; :=\; \left\{ i\in[n]\mid \overline{w}_i=\left(\overline{w}_{lb}\right)_i \right\},\\
&I_{g_{4}}\ \; :=\; \left\{ i\in[n]\mid 
 \overline{w}_i=\left(\overline{w}_{ub}\right)_i\right\}.
\end{align*}
\end{subequations}
 It holds that $I_{H}=\underset{k=1} {\overset{6}{\cup}}I_{H_{k}},\ I_{G}=\underset{k=1} {\overset{6}{\cup}}I_{G_{k}}$, $I_{GH}=\underset{k=1} {\overset{6}{\cup}}I_{GH_{k}}$, and $I_{g}=\underset{k=1} {\overset{4}{\cup}}I_{g_{k}}$.

     We decompose the complementarity constraints in (\ref{MPEC}) into the following index sets:
\begin{eqnarray*}
\Psi_{1}&:=&\left\{i \in Q_{u}\ \mid \ 0\leq \zeta_{i}<1,\ \left( A  B^{\top}\alpha+A\beta-A\gamma+z\right)_{i}=0,\ z_{i}=0\right\}, \\
\Psi_{2}&:=&\left\{i \in Q_{u}\ \mid \ \zeta_{i}=0,\ \left(A  B^{\top}\alpha+A\beta-A\gamma+z\right)_{i}>0,\ z_{i}=0\right\}, \\
\Psi_{3}&:=&\left\{i \in Q_{u}\ \mid \ \zeta_{i}=1,\ \left(A  B^{\top}\alpha+A\beta-A\gamma+z\right)_{i}=0,\ z_{i}>0\right\},\\
\Lambda_{1}&:=&\left\{i \in Q_{l}\ \mid \ \alpha_{i}=0,\ \left(B  B^{\top}\alpha+B\beta-B\gamma-\mathbf{1}+\xi\right)_{i}=0,\ \xi_{i}=0\right\}, \\
\Lambda_{2}&:=&\left\{i \in Q_{l}\ \mid \ \alpha_{i}=0,\ \left(B  B^{\top}\alpha+B\beta-B\gamma -\mathbf{1}+\xi\right)_{i}>0,\ \xi_{i}=0\right\}, \\
\Lambda_{3}&:=&\left\{i \in Q_{l}\ \mid \ 0< \alpha_{i} \leq C,\ \left(B  B^{\top}\alpha+B\beta-B\gamma-\mathbf{1}+\xi\right)_{i}=0,\ \xi_{i}=0\right\},\\
\Lambda_{u}&:=&\left\{i \in Q_{l}\ \mid \ \alpha_{i}=C,\ \left(B B^{\top}\alpha+B\beta-B\gamma-\mathbf{1}+\xi\right)_{i}=0,\ \xi_{i}>0\right\},\\
\Pi_{1}&:=&\left\{ i\in Q_b\mid-\overline{w}_{(i\ mod \ K)}= \left(B^{\top}\alpha+\beta-\gamma\right)_i<\overline{w}_{(i\ mod \ K)},\beta_i\geq 0,\gamma_i=0  \right\},\\
\Pi_{2}&:=&\left\{ i\in Q_b\mid-\overline{w}_{(i\ mod \ K)}< \left(B^{\top}\alpha+\beta-\gamma\right)_i<\overline{w}_{(i\ mod \ K)},\beta_i= 0,\gamma_i=0 \right\},\\
\Pi_{3}&:=&\left\{ i\in Q_b\mid-\overline{w}_{(i\ mod \ K)}< \left(B^{\top}\alpha+\beta-\gamma\right)_i=\overline{w}_{(i\ mod \ K)},\beta_i=0,\gamma_i\geq0 \right\}.
\end{eqnarray*}
Here $\overline{w}_{(i\ mod \ K)}$ represents the component of $\overline{w}$ corresponding to the index $i\ mod \ K$.\\
Subsequently, the index sets $\Psi_{1},\ \Pi_{1},\ \Pi_{3},\ \Lambda_{3}$ are further partitioned into the following subsets: 
\begin{eqnarray*}
\Psi^{0}_{1} &:=& \left\{i \in Q_{u}\ \mid \ \zeta_{i}=0,\ \left(A B^{\top}\alpha+A\beta-A\gamma+z\right)_{i}=0,\ z_{i}=0\right\}, \\
\Psi^{+}_{1}&:=& \left\{i \in Q_{u}\ \mid \ 0<\zeta_{i}<1,\ \left(A B^{\top}\alpha+A\beta-A\gamma+z\right)_{i}=0,\ z_{i}=0\right\},\\
 \Pi_{1}^0&:=&\left\{ i\in Q_b\mid-\overline{w}_{(i\ mod \ K)}= \left(B^{\top}\alpha+\beta-\gamma\right)_i<\overline{w}_{(i\ mod \ K)},\beta_i=0,\gamma_i=0  \right\},\\
\Pi_{1}^+&:=&\left\{ i\in Q_b\mid-\overline{w}_{(i\ mod \ K)}= \left(B^{\top}\alpha+\beta-\gamma\right)_i<\overline{w}_{(i\ mod \ K)},\beta_i> 0,\gamma_i=0  \right\},\\
\Pi_{3}^0&:=&\left\{ i\in Q_b\mid-\overline{w}_{(i\ mod \ K)}< \left(B^{\top}\alpha+\beta-\gamma\right)_i=\overline{w}_{(i\ mod \ K)},\beta_i=0,\gamma_i=0 \right\},\\
\Pi_{3}^+&:=&\left\{ i\in Q_b\mid-\overline{w}_{\left(i\ mod \ K\right)}< \left(B^{\top}\alpha+\beta-\gamma\right)_i=\overline{w}_{(i\ mod \ K)},\beta_i=0,\gamma_i>0 \right\},\\
  \Lambda^{+}_{3} &:=&\left\{i \in Q_{l}\ \mid \ 0<\alpha_{i}<C,\ \left(B  B^{\top}\alpha+B\beta-B\gamma-\mathbf{1}+\xi\right)_{i}=0,\ \xi_{i}=0\right\},  \\
 \Lambda^{c}_{3} &:=& \left\{i \in Q_{l}\ \mid \ \alpha_{i}=C,\ \left(B  B^{\top}\alpha+B\beta-B\gamma-\mathbf{1}+\xi\right)_{i}=0,\ \xi_{i}=0\right\}.
\end{eqnarray*}

Let
$ \setlength{\abovedisplayskip}{1.5pt} I^{k}:=I_{H_{k}}\cup I_{G_{k}} \cup I_{GH_{k}},\ k=1,\cdots, 6.
\setlength{\belowdisplayskip}{1.5pt} $
Each index set $I^{k}$, for $k=1,\cdots, 6$, corresponds to the union of the three partition components associated with the respective complementarity conditions in system~(\ref{MPEC}). Specifically, we have:

\begin{itemize}
\item [{}]  Part 1: $I^{1}$ for the system $\mathbf{0}\leq \zeta\perp AB^{\top}\alpha+A\beta-A\gamma+z\geq \mathbf{0}$;
\item [{}]  Part 2: $I^{2}$ for the system  $ \mathbf{0}\leq z\perp \mathbf{1}-\zeta\geq \mathbf{0}$;
\item [{}]  Part 3: $I^{3}$ for the system $\mathbf{0}\leq \alpha\perp BB^{\top}\alpha+B\beta-B\gamma-\mathbf{1}+\xi \geq\mathbf{0}$;
\item [{}]  Part 4: $I^{4}$ for the system $ \mathbf{0}\leq \xi\perp \mathbf{1}C-\alpha \geq\mathbf{0}$;
\item [{}]  Part 5: $I^{5}$ for the system $\mathbf{0}\leq \beta \perp B^{\top}\alpha+\beta-\gamma+E_K^n\overline{w} \geq\mathbf{0}$;
\item [{}]  Part 6: $I^{6}$ for the system $ \mathbf{0}\leq \gamma\perp -B^{\top}\alpha-\beta+\gamma+E_K^n\overline{w} \geq\mathbf{0}$.
\end{itemize}

The following relationships immediately follow: $I^{1}=I^{2}=Q_{u}$, $I^{3}=I^{4}=Q_{l}$ and $I^{5}=I^{6}=Q_b$.

To facilitate a comprehensive characterization of the set of gradient vectors in~(\ref{MPEC}), we introduce the following result.
\begin{proposition}\label{prop1}
The index sets have the following relationship:
\begin{itemize}
\item [{\rm (a)}] In Part $1,\ I_{H_{1}}=\Psi_{2},\ I_{G_{1}}=\Psi_{3},\ I_{GH_{1}}=\emptyset$;
\item [{\rm (b)}] In Part $2,\ I_{H_{2}}= \Psi_{2},\ I_{G_{2}}=\Psi_{3},\ I_{GH_{2}}=\emptyset$;
\item [{\rm (c)}] In Part $3,\ I_{H_{3}}=\Lambda_{2},\ I_{G_{3}}=\Lambda_{3} \cup \Lambda_{u},\ I_{GH_{3}}=\Lambda_{1}$;
\item [{\rm (d)}] In Part $4,\ I_{H_{4}}=\Lambda_{1} \cup \Lambda_{2} \cup \Lambda^{+}_{3},\ I_{G_{4}}=\Lambda_{u},\ I_{GH_{4}}=\Lambda^{c}_{3}$;
 \item [{\rm (e)}] In Part $5,\ I_{H_{5}}=\Pi_2\cup\Pi_3,\ I_{G_{5}}=\Pi_1^+,\ I_{GH_{5}}=\Pi_1^0$;
        \item [{\rm (f)}] In Part $6,\ I_{H_{6}}=\Pi_1\cup\Pi_2,\ I_{G_{6}}=\Pi_3^+,\ I_{GH_{6}}=\Pi_3^0.$
\end{itemize}
\end{proposition}
\begin{proof}    
Throughout this paper, we have established the assumption that for any validation data point $x_i$, the condition $x_i^{\top} w_i^t \neq 0$ holds for each $t = 1, \cdots, K$. Consequently, this assumption implies that the set $\left\{i \in Q_{u} \mid \left(A B^{\top}\alpha+A\beta-A\gamma+z\right)_{i} = 0\right\}$ must be empty. From this, it follows that both $\Psi_1^0$ and $\Psi_1^+$ are empty sets. Combining these observations with the findings in~\cite{li2022bilevel} completes the proof. 
\end{proof}
Now with Proposition \ref{prop1}, we are ready to characterize the set of gradient vectors in (\ref{grad1}) and (\ref{grad2}).

\begin{proposition}\label{pro5}
 The set of gradient vectors in \eqref{grad1} and \eqref{grad2} at a feasible point $v$ for the MPEC in \eqref{MPEC} can be written as the row vectors in the following matrix \rm{(denoted as $\Gamma$)}:
 
\begin{equation*}
\Gamma:=\left[
\small
\begin{array}{cccccccc}
\centering{\Gamma_{a}^{1}}&\mathbf{0}_{(I_{g_{1}},\ L_{2})}&\mathbf{0}_{(I_{g_{c}},\ L_{3})}&\mathbf{0}_{(I_{g_{c}},\ L_{4})}&\mathbf{0}_{(I_{g_{c}},\ L_{5})}&\mathbf{0}_{(I_{g_{c}},\ L_{6})}&\mathbf{0}_{(I_{g_{c}},\ L_{7})}&\mathbf{0}_{(I_{g_{c}},\ L_{8})}\\
\mathbf{0}_{(I_{g_{2}},\ L_{1})}& \centering{\Gamma_{b}^{2}}&\mathbf{0}_{(I_{g_{w}},\ L_{3})}&\mathbf{0}_{(I_{g_{w}},\ L_{4})}&\mathbf{0}_{(I_{g_{w}},\ L_{5})}&\mathbf{0}_{(I_{g_{w}},\ L_{6})}&\mathbf{0}_{(I_{g_{w}},\ L_{7})}&\mathbf{0}_{(I_{g_{w}},\ L_{8})}\\
\mathbf{0}_{(I_{G_{1}},\ L_{1})}& \mathbf{0}_{(I_{G_{1}},\ L_{2})} &\mathbf{0}_{(I_{G_{1}},\ L_{3})}&\centering{\Gamma_{c}^{4}} & (A B^{\top})_{(I_{G_{1}},\ \cdot \ )} & \mathbf{0}_{(I_{G_{1}},\ L_{6})}& A_{(I_{G_{1}},\ \cdot \ )} &-A_{(I_{G_{1}},\ \cdot \ )}\\
\mathbf{0}_{(I_{H_{1}},\ L_{1})} &\mathbf{0}_{(I_{H_{1}},\ L_{2})}&\centering{\Gamma_{f}^{3}} & \mathbf{0}_{(I_{H_{1}},\ L_{4}) } & \mathbf{0}_{(I_{H_{1}},\ L_{5})} & \mathbf{0}_{(I_{H_{1}},\ L_{6})}& \mathbf{0}_{(I_{H_{1}},\ L_{7})}& \mathbf{0}_{(I_{H_{1}},\ L_{8})}\\
\mathbf{0}_{(I_{G_{2}},\ L_{1})} &\mathbf{0}_{(I_{G_{2}},\ L_{2})}&
\centering{\Gamma_{g}^{3}}& \mathbf{0}_{(I_{G_{2}},\ L_{4}) } & \mathbf{0}_{(I_{G_{2}},\ L_{5})} & \mathbf{0}_{(I_{G_{2}},\ L_{6})}& \mathbf{0}_{(I_{G_{2}},\ L_{7})}& \mathbf{0}_{(I_{G_{2}},\ L_{8})}\\
\mathbf{0}_{(I_{H_{2}},\ L_{1})}& \mathbf{0}_{(I_{H_{2}},\ L_{2})}& \mathbf{0}_{(I_{H_{2}},\ L_{3})}&\centering{\Gamma_{h}^{4}} & \mathbf{0}_{(I_{H_{2}},\ L_{5}) } & \mathbf{0}_{(I_{H_{2}},\ L_{6})} & \mathbf{0}_{(I_{H_{2}},\ L_{7})}& \mathbf{0}_{(I_{H_{2}},\ L_{8})}\\
\mathbf{0}_{(I_{G_{3}},\ L_{1})}&\mathbf{0}_{(I_{G_{3}},\ L_{2})}& \mathbf{0}_{(I_{G_{3}},\ L_{3})} &\mathbf{0}_{(I_{G_{3}},\ L_{4})} & (B B^{\top})_{(I_{G_{3}},\ \cdot \ )} &\centering{\Gamma_{i}^{6}} & B_{(I_{G_{3}},\ \cdot \ )} & -B_{(I_{G_{3}},\ \cdot \ )}\\
\mathbf{0}_{(I_{GH_{3}},\ L_{1})}&\mathbf{0}_{(I_{GH_{3}},\ L_{2})} & \mathbf{0}_{(I_{GH_{3}},\ L_{3})} &\mathbf{0}_{(I_{GH_{3}},\ L_{4})} & (B B^{\top})_{(I_{GH_{3}},\ \cdot \ )} & \centering{\Gamma_{j}^{6}} & B_{(I_{GH_{3}},\ \cdot \ )} &-B_{(I_{GH_{3}},\ \cdot \ )} \\
\mathbf{0}_{(I_{GH_{3}},\ L_{1})} &\mathbf{0}_{(I_{GH_{3}},\ L_{2})} & \mathbf{0}_{(I_{GH_{3}},\ L_{3})} &\mathbf{0}_{(I_{GH_{3}},\ L_{4})} & \centering{\Gamma_{k}^{5}} & \mathbf{0}_{(I_{GH_{3}},\ L_{6})}&\mathbf{0}_{(I_{GH_{3}},\ L_{7})}&\mathbf{0}_{(I_{GH_{3}},\ L_{8})}\\
\mathbf{0}_{(I_{H_{3}},\ L_{1})}  &\mathbf{0}_{(I_{H_{3}},\ L_{2})}& \mathbf{0}_{(I_{H_{3}},\ L_{3})} &\mathbf{0}_{(I_{H_{3}},\ L_{4})} & \centering{\Gamma_{l}^{5}} & \mathbf{0}_{(I_{H_{3}},\ L_{6})}&\mathbf{0}_{(I_{H_{3}},\ L_{7})}&\mathbf{0}_{(I_{H_{3}},\ L_{8})}\\
\mathbf{1}_{(I_{G_{4}},\ L_{1})}&\mathbf{0}_{(I_{G_{4}},\ L_{2})}& \mathbf{0}_{(I_{G_{4}},\ L_{3})} & \mathbf{0}_{(I_{G_{4}},\ L_{4})} & \centering{\Gamma_{m}^{5}} & \mathbf{0}_{(I_{G_{4}},\ L_{6}) }& \mathbf{0}_{(I_{G_{4}},\ L_{7}) }& \mathbf{0}_{(I_{G_{4}},\ L_{8}) } \\
\mathbf{1}_{(I_{GH_{4}},\ L_{1})}  & \mathbf{0}_{(I_{GH_{4}},\ L_{2})}&\mathbf{0}_{(I_{GH_{4}},\ L_{3})} & \mathbf{0}_{(I_{GH_{4}},\ L_{4})} & \centering{\Gamma_{n}^{5}} & \mathbf{0}_{(I_{GH_{4}},\ L_{6}) }& \mathbf{0}_{(I_{GH_{4}},\ L_{7}) }& \mathbf{0}_{(I_{GH_{4}},\ L_{8}) }\\
\mathbf{0}_{(I_{GH_{4}},\ L_{1})}  & \mathbf{0}_{(I_{GH_{4}},\ L_{2})}& \mathbf{0}_{(I_{GH_{4}},\ L_{3})}  & \mathbf{0}_{(I_{GH_{4}},\ L_{4})}& \mathbf{0}_{(I_{GH_{4}},\ L_{5})}& \centering{\Gamma_{o}^{6}}& \mathbf{0}_{(I_{GH_{4}},\ L_{7})}& \mathbf{0}_{(I_{GH_{4}},\ L_{8})}\\
\mathbf{0}_{(I_{H_{4}},\ L_{1})}  &\mathbf{0}_{(I_{H_{4}},\ L_{2})}& \mathbf{0}_{(I_{H_{4}},\ L_{3})}  & \mathbf{0}_{(I_{H_{4}},\ L_{4})}& \mathbf{0}_{(I_{H_{4}},\ L_{5})}& \centering{\Gamma_{p}^{6}}& \mathbf{0}_{(I_{H_{4}},\ L_{7})}& \mathbf{0}_{(I_{H_{4}},\ L_{8})}\\
\mathbf{0}_{(I_{G_{5}},\ L_{1})}&\centering{\Gamma_q^2}&\mathbf{0}_{(I_{G_{5}},\ L_{3})}&\mathbf{0}_{(I_{G_{5}},\ L_{4})}& B^{\top}_{(I_{G_{5}},\ \cdot \ )} &\mathbf{0}_{(I_{G_{5}},\ L_{6})}&\centering{\Gamma_q^7}&\Gamma_q^8 \\
\mathbf{0}_{(I_{GH_{5}},\ L_{1})}& \centering{\Gamma_r^2}&\mathbf{0}_{(I_{GH_{5}},\ L_{3})}&\mathbf{0}_{(I_{GH_{5}},\ L_{4})}& B^{\top}_{(I_{GH_{5}},\ \cdot \ )} &\mathbf{0}_{(I_{GH_{5}},\ L_{6})}&\centering{\Gamma_r^7}&\Gamma_r^8\\
\mathbf{0}_{(I_{GH_{5}},\ L_{1})}&\mathbf{0}_{(I_{GH_{5}},\ L_{2})}&\mathbf{0}_{(I_{GH_{5}},\ L_{3})}&\mathbf{0}_{(I_{GH_{5}},\ L_{4})}&\mathbf{0}_{(I_{GH_{5}},\ L_{5})}&\mathbf{0}_{(I_{GH_{5}},\ L_{6})}&\centering{\Gamma_s^7}&\mathbf{0}_{(I_{GH_{5}},\ L_{8})}\\
\mathbf{0}_{(I_{H_{5}},\ L_{1})}&\mathbf{0}_{(I_{H_{5}},\ L_{2})}&\mathbf{0}_{(I_{H_{5}},\ L_{3})}&\mathbf{0}_{(I_{H_{5}},\ L_{4})}&\mathbf{0}_{(I_{H_{5}},\ L_{5})}&\mathbf{0}_{(I_{H_{5}},\ L_{6})}&\centering{\Gamma_t^7}&\mathbf{0}_{(I_{H_{5}},\ L_{8})}\\
\mathbf{0}_{(I_{G_{6}},\ L_{1})}& \centering{\Gamma_u^2}&\mathbf{0}_{(I_{G_{6}},\ L_{3})}&\mathbf{0}_{(I_{G_{6}},\ L_{4})}& -B^{\top}_{(I_{G_{6}},\ \cdot \ )} &\mathbf{0}_{(I_{G_{6}},\ L_{6})}&\centering{\Gamma_u^7}&\Gamma_u^8\\
\mathbf{0}_{(I_{GH_{6}},\ L_{1})}& \centering{\Gamma_v^2}&\mathbf{0}_{(I_{GH_{6}},\ L_{3})}&\mathbf{0}_{(I_{GH_{6}},\ L_{4})}& -B^{\top}_{(I_{GH_{6}},\ \cdot \ )} &\mathbf{0}_{(I_{GH_{6}},\ L_{6})}&\centering{\Gamma_v^7}&\Gamma_v^8\\
\mathbf{0}_{(I_{GH_{6}},\ L_{1})}&\mathbf{0}_{(I_{GH_{6}},\ L_{2})}&\mathbf{0}_{(I_{GH_{6}},\ L_{3})}&\mathbf{0}_{(I_{GH_{6}},\ L_{4})}&\mathbf{0}_{(I_{GH_{6}},\ L_{5})}&\mathbf{0}_{(I_{GH_{6}},\ L_{6})}&\mathbf{0}_{(I_{GH_{6}},\ L_{7})}&\Gamma_w^8\\
\mathbf{0}_{(I_{H_{6}},\ L_{1})}&\mathbf{0}_{(I_{H_{6}},\ L_{2})}&\mathbf{0}_{(I_{H_{6}},\ L_{3})}&\mathbf{0}_{(I_{H_{6}},\ L_{4})}&\mathbf{0}_{(I_{H_{6}},\ L_{5})}&\mathbf{0}_{(I_{H_{6}},\ L_{6})}&\mathbf{0}_{(I_{H_{6}},\ L_{7})}&\Gamma_x^8
\end{array}
\right],
\end{equation*}
 with $L_{q},\ q=1,\ \cdots,\ 8$ being index sets of columns corresponding to the variables $C$, $\overline{w}$, $\zeta$, $z$, $\alpha$, $\xi$, $\beta$, and $\gamma$ respectively, while we have 
\begin{equation*}
\begin{array}{ll}
\Gamma_{a}^{1} := 
\left[
\begin{array}{c}
1_{(I_{g_{1}},\ \cdot)} \\ -1_{(I_{g_{2}} ,\ \cdot)}
\end{array}
\right],&
\Gamma_{b}^{2}  := 
\left[
\begin{array}{c}
I_{(I_{g_{3}},\ \cdot)} \\ -I_{(I_{g_{4}},\ \cdot)} 
\end{array}
\right],\\
\Gamma_{q}^{2}:=\left[\begin{array}{c}  {E^n_K}_{(I_{G_{5}},\ \cdot)} \end{array} \right],&
\Gamma_{r}^{2}:=\left[\begin{array}{c} {E^n_K}_{(I_{GH_{5}},\ \cdot)}  \end{array} \right],\\
\Gamma_{u}^{2}:=\left[\begin{array}{c}  {E^n_K}_{(I_{G_{6}},\cdot)} \end{array} \right],&
\Gamma_{v}^{2}:=\left[\begin{array}{c}  {E^n_K}_{(I_{GH_{6}},\cdot)} \end{array} \right],\\
 \Gamma_{f}^{3}:=\left[\begin{array}{cc} \mathbf{0}_{(I_{H_{1}},\ \Psi_{1} \cup \Psi_{3})}&  I_{(I_{H_{1}},\ \Psi_{2})} \end{array} \right],&
\Gamma_{g}^{3} := \left[\begin{array}{cc} \mathbf{0}_{(I_{G_{2}},\  \Psi_{1} \cup \Psi_{2})}& -I_{(I_{G_{2}},\ \Psi_{3})} \end{array} \right],\\
\Gamma_{c}^{4}:= 
\left[
\begin{array}{cc}
\mathbf{0}_{(I_{G_{1}},\  \Psi_{2})} & I_{(I_{G_{1}} ,\  \Psi_{3})}
\end{array}
\right],&
 \Gamma_{h}^{4} :=\left[\begin{array}{cc} I_{(I_{H_{2}},\  \Psi_{2})}& \mathbf{0}_{(I_{H_{2}},\ \Psi_{3})} \end{array} \right],\\
\Gamma_{k}^{5}:=\left[\begin{array}{cc} I_{(I_{GH_{3}},\ \Lambda_{1}) }& \mathbf{0}_{(I_{GH_{3}},\ \Lambda_{2} \cup \Lambda_{3} \cup \Lambda_{u})} \end{array} \right],&
 \Gamma_{l}^{5}:=\left[\begin{array}{cc} \mathbf{0}_{(I_{H_{3}},\  \Lambda_{1} \cup \Lambda_{3} \cup \Lambda_{u})}& I_{(I_{H_{3}},\ \Lambda_{2})} \end{array} \right],\\
 \Gamma_{m}^{5}:=\left[\begin{array}{cc} \mathbf{0}_{(I_{G_{4}},\ \Lambda_{1} \cup \Lambda_{2} \cup \Lambda_{3})} &  -I_{(I_{G_{4}},\ \Lambda_{u})} \end{array} \right],&
 \Gamma_{n}^{5}:=\left[\begin{array}{cc} \mathbf{0}_{(I_{GH_{4}},\ \Lambda_{1} \cup \Lambda_{2} \cup \Lambda^{+}_{3} \cup \Lambda_{u})} & -I_{(I_{GH_{4}},\  \Lambda^{c}_{3})} \end{array} \right],\\ 
\Gamma_{i}^{6}:=\left[\begin{array}{cc}\mathbf{0}_{(I_{G_{3}},\ \Lambda_{1} \cup \Lambda_{2})}& I_{(I_{G_{3}},\ \Lambda_{3} \cup \Lambda_{u})} \end{array} \right],&
\Gamma_{j}^{6}:=\left[\begin{array}{cc}I_{(I_{GH_{3}},\  \Lambda_{1})}&\mathbf{0}_{(I_{GH_{3}},\ \Lambda_{2}\cup \Lambda_{3} \cup \Lambda_{u})} \end{array} \right], \\
\Gamma_{o}^{6}:=\left[\begin{array}{cc} \mathbf{0}_{(I_{GH_{4}},\ \Lambda_{1} \cup \Lambda_{2} \cup \Lambda^{+}_{3} \cup \Lambda_{u})} & I_{(I_{GH_{4}},\ \Lambda^{c}_{3})} \end{array} \right],&
 \Gamma_{p}^{6}:=\left[\begin{array}{cc} I_{(I_{H_{4}},\ \Lambda_{1} \cup \Lambda_{2} \cup \Lambda^{+}_{3})} & \mathbf{0}_{(I_{H_{4}},\  \Lambda^{c}_{3} \cup \Lambda_{u}) }  \end{array} \right],\\ 
\Gamma_{q}^{7}:=\left[\begin{array}{cc} \mathbf{0}_{(I_{G_{5}},\ \Pi_{1}^0 \cup \Pi_{2} \cup \Pi_{3})} & I_{(I_{G_{5}},\ \Pi_{1}^+)} \end{array} \right],&
\Gamma_{r}^{7}:=\left[\begin{array}{cc} I_{(I_{GH_{5}},\ \Pi_{1}^0)} & \mathbf{0}_{(I_{GH_{5}},\ \Pi_1^+\cup \Pi_2\cup\Pi_3)} \end{array} \right],\\
\Gamma_{s}^{7}:=\left[\begin{array}{cc} I_{(I_{GH_{5}},\ \Pi_{1}^0)} & \mathbf{0}_{(I_{GH_{5}},\  \Pi^{+}_{1} \cup \Pi_{2}\cup \Pi_3) }  \end{array} \right],&
\Gamma_{t}^{7}:=\left[\begin{array}{cc} \mathbf{0}_{(I_{H_{5}},\ \Pi_{1})} & I_{(I_{H_{5}},\   \Pi_{2}\cup \Pi_3) }  \end{array} \right],\\
\Gamma_{u}^{7}:=\left[\begin{array}{cc} \mathbf{0}_{(I_{G_{6}},\ \Pi_{1} \cup \Pi_{2} \cup \Pi_{3}^0)} & -I_{(I_{G_{6}},\ \Pi_{3}^+)} \end{array} \right],&
\Gamma_{v}^{7}:=\left[\begin{array}{cc} \mathbf{0}_{(I_{GH_{6}},\ \Pi_{1} \cup \Pi_{2} \cup \Pi_{3}^+)} & -I_{(I_{GH_{6}},\ \Pi_{3}^0)} \end{array} \right],\\
\Gamma_{q}^{8}:=\left[\begin{array}{cc} \mathbf{0}_{(I_{G_{5}},\ \Pi_{1}^0 \cup \Pi_{2} \cup \Pi_{3})} & -I_{(I_{G_{5}},\ \Pi_{1}^+)} \end{array} \right],&
\Gamma_{r}^{8}:=\left[\begin{array}{cc} -I_{(I_{GH_{5}},\ \Pi_{1}^0)} & \mathbf{0}_{(I_{GH_{5}},\ \Pi_1^+\cup \Pi_2\cup\Pi_3)} \end{array} \right],\\
\Gamma_{u}^{8}:=\left[\begin{array}{cc} \mathbf{0}_{(I_{G_{6}},\ \Pi_{1} \cup \Pi_{2} \cup \Pi_{3}^0)} & I_{(I_{G_{6}},\ \Pi_{3}^+)} \end{array} \right],&
\Gamma_{v}^{8}:=\left[\begin{array}{cc} \mathbf{0}_{(I_{GH_{6}},\ \Pi_{1} \cup \Pi_{2} \cup \Pi_{3}^+)} & I_{(I_{GH_{6}},\ \Pi_{3}^0)} \end{array} \right],\\
\Gamma_{w}^{8}:=\left[\begin{array}{cc} \mathbf{0}_{(I_{GH_{6}},\ \Pi_{1} \cup \Pi_{2} \cup \Pi_{3}^+)} & I_{(I_{GH_{6}},\ \Pi_{3}^0)} \end{array} \right],&
\Gamma_{x}^{8}:=\left[\begin{array}{cc}  I_{(I_{H_{6}},\ \Pi_{1}\cup\Pi_2)}&\mathbf{0}_{(I_{H_{6}},\ \Pi_{3} )} \end{array} \right].
\end{array}    
\end{equation*}
\end{proposition}
\begin{proof}
    Based on the definition provided in equation \eqref{grad1} and \eqref{grad2}, we introduce the matrix 
\begin{equation*}
\begin{array}{cccccc}
        
        \Gamma =[\nabla g(v)_{I_g}^{\top},&\
        \nabla G(v)_{I_{G_1}}^{\top},&\ 
         \nabla H(v)_{I_{{H}_1}}^{\top},&\
         \nabla G(v)_{I_{G_2}}^{\top} &\
         \nabla H(v)_{I_{{H}_2}}^{\top},\\
         \nabla G(v)_{I_{G_3}}^{\top},&\
         \nabla G(v)_{I_{{GH}_3}}^{\top},&\
         \nabla H(v)_{I_{{GH}_3}}^{\top},&\
         \nabla H(v)_{I_{{H}_3}}^{\top},&\
         \nabla G(v)_{I_{G_4}}^{\top},\\ 
         \nabla G(v)_{I_{{GH}_4}}^{\top},&\
         \nabla H(v)_{I_{{GH}_4}}^{\top},&\
         \nabla H(v)_{I_{{H}_4}}^{\top},&\
         \nabla G(v)_{I_{G_5}}^{\top},&\
         \nabla G(v)_{I_{{GH}_5}}^{\top},\\
         \nabla H(v)_{I_{{GH}_5}}^{\top},&\
          \nabla H(v)_{I_{{H}_5}}^{\top},&\
         \nabla G(v)_{I_{G_6}}^{\top},&\ 
         \nabla G(v)_{I_{{GH}_6}}^{\top},&\
         \nabla H(v)_{I_{{GH}_6}}^{\top},&\
         \nabla H(v)_{I_{{H}_6}}^{\top}].
   \end{array}   
\end{equation*}
    In view of Proposition~\ref{prop1}, we have that the sets  $I_{{GH}_1}$ and $I_{{GH}_2}$ are empty; thus, the corresponding columns are excluded from the matrix $\Gamma$.  By incorporating the results established in Propositions~\ref{prop1}, the desired matrix formulation follows directly. 
\end{proof}


We are now ready to provide the proof of Theorem \ref{thr1}.

\begin{proof}
Under the assumption $-\overline{w} < w^t < \overline{w}$ for each $t = 1, \cdots, K$, it immediately follows that
\begin{equation}
    \beta = 0, \quad \gamma = 0.
\end{equation}
Consequently, the associated index sets become empty, specifically,
\begin{equation*}
    I_{G_5} = \emptyset, \quad I_{G_6} = \emptyset, \quad I_{{GH}_5} = \emptyset, \quad I_{{GH}_6} = \emptyset.
\end{equation*}

Thus, the matrix $\Gamma$ in Propositon 3 reduces to the following:
\begin{equation*}
\Gamma=\left[
\small
\begin{array}{cccccccc}
\centering{\Gamma_{a}^{1}}&\mathbf{0}_{(I_{g_{1}},\ L_{2})}&\mathbf{0}_{(I_{g_{c}},\ L_{3})}&\mathbf{0}_{(I_{g_{c}},\ L_{4})}&\mathbf{0}_{(I_{g_{c}},\ L_{5})}&\mathbf{0}_{(I_{g_{c}},\ L_{6})}&\mathbf{0}_{(I_{g_{c}},\ L_{7})}&\mathbf{0}_{(I_{g_{c}},\ L_{8})}\\
\mathbf{0}_{(I_{g_{2}},\ L_{1})}& \centering{\Gamma_{b}^{2}}&\mathbf{0}_{(I_{g_{w}},\ L_{3})}&\mathbf{0}_{(I_{g_{w}},\ L_{4})}&\mathbf{0}_{(I_{g_{w}},\ L_{5})}&\mathbf{0}_{(I_{g_{w}},\ L_{6})}&\mathbf{0}_{(I_{g_{w}},\ L_{7})}&\mathbf{0}_{(I_{g_{w}},\ L_{8})}\\
\mathbf{0}_{(I_{G_{1}},\ L_{1})}& \mathbf{0}_{(I_{G_{1}},\ L_{2})} &\mathbf{0}_{(I_{G_{1}},\ L_{3})}&\centering{\Gamma_{c}^{4}} & (A B^{\top})_{(I_{G_{1}},\ \cdot \ )} & \mathbf{0}_{(I_{G_{1}},\ L_{6})}& \mathbf{0}_{(I_{G_{1}},\ L_{7})} &\mathbf{0}_{(I_{G_{1}},\ L_{8})}\\
\mathbf{0}_{(I_{H_{1}},\ L_{1})} &\mathbf{0}_{(I_{H_{1}},\ L_{2})}&\centering{\Gamma_{f}^{3}} & \mathbf{0}_{(I_{H_{1}},\ L_{4}) } & \mathbf{0}_{(I_{H_{1}},\ L_{5})} & \mathbf{0}_{(I_{H_{1}},\ L_{6})}& \mathbf{0}_{(I_{H_{1}},\ L_{7})}& \mathbf{0}_{(I_{H_{1}},\ L_{8})}\\
\mathbf{0}_{(I_{G_{2}},\ L_{1})} &\mathbf{0}_{(I_{G_{2}},\ L_{2})}&
\centering{\Gamma_{g}^{3}}& \mathbf{0}_{(I_{G_{2}},\ L_{4}) } & \mathbf{0}_{(I_{G_{2}},\ L_{5})} & \mathbf{0}_{(I_{G_{2}},\ L_{6})}& \mathbf{0}_{(I_{G_{2}},\ L_{7})}& \mathbf{0}_{(I_{G_{2}},\ L_{8})}\\
\mathbf{0}_{(I_{H_{2}},\ L_{1})}& \mathbf{0}_{(I_{H_{2}},\ L_{2})}& \mathbf{0}_{(I_{H_{2}},\ L_{3})}&\centering{\Gamma_{h}^{4}} & \mathbf{0}_{(I_{H_{2}},\ L_{5}) } & \mathbf{0}_{(I_{H_{2}},\ L_{6})} & \mathbf{0}_{(I_{H_{2}},\ L_{7})}& \mathbf{0}_{(I_{H_{2}},\ L_{8})}\\
\mathbf{0}_{(I_{G_{3}},\ L_{1})}&\mathbf{0}_{(I_{G_{3}},\ L_{2})}& \mathbf{0}_{(I_{G_{3}},\ L_{3})} &\mathbf{0}_{(I_{G_{3}},\ L_{4})} & (B B^{\top})_{(I_{G_{3}},\ \cdot \ )} &\centering{\Gamma_{i}^{6}} & \mathbf{0}_{(I_{G_{3}},\ L_{7})} & \mathbf{0}_{(I_{G_{3}},\ L_{8})}\\
\mathbf{0}_{(I_{GH_{3}},\ L_{1})}&\mathbf{0}_{(I_{GH_{3}},\ L_{2})} & \mathbf{0}_{(I_{GH_{3}},\ L_{3})} &\mathbf{0}_{(I_{GH_{3}},\ L_{4})} & (B B^{\top})_{(I_{GH_{3}},\ \cdot \ )} & \centering{\Gamma_{j}^{6}} & \mathbf{0}_{(I_{GH_{3}},\ L_{7})} &\mathbf{0}_{(I_{GH_{3}},\ L_{8})} \\
\mathbf{0}_{(I_{GH_{3}},\ L_{1})} &\mathbf{0}_{(I_{GH_{3}},\ L_{2})} & \mathbf{0}_{(I_{GH_{3}},\ L_{3})} &\mathbf{0}_{(I_{GH_{3}},\ L_{4})} & \centering{\Gamma_{k}^{5}} & \mathbf{0}_{(I_{GH_{3}},\ L_{6})}&\mathbf{0}_{(I_{GH_{3}},\ L_{7})}&\mathbf{0}_{(I_{GH_{3}},\ L_{8})}\\
\mathbf{0}_{(I_{H_{3}},\ L_{1})}  &\mathbf{0}_{(I_{H_{3}},\ L_{2})}& \mathbf{0}_{(I_{H_{3}},\ L_{3})} &\mathbf{0}_{(I_{H_{3}},\ L_{4})} & \centering{\Gamma_{l}^{5}} & \mathbf{0}_{(I_{H_{3}},\ L_{6})}&\mathbf{0}_{(I_{H_{3}},\ L_{7})}&\mathbf{0}_{(I_{H_{3}},\ L_{8})}\\
\mathbf{1}_{(I_{G_{4}},\ L_{1})}&\mathbf{0}_{(I_{G_{4}},\ L_{2})}& \mathbf{0}_{(I_{G_{4}},\ L_{3})} & \mathbf{0}_{(I_{G_{4}},\ L_{4})} & \centering{\Gamma_{m}^{5}} & \mathbf{0}_{(I_{G_{4}},\ L_{6}) }& \mathbf{0}_{(I_{G_{4}},\ L_{7}) }& \mathbf{0}_{(I_{G_{4}},\ L_{8}) } \\
\mathbf{1}_{(I_{GH_{4}},\ L_{1})}  & \mathbf{0}_{(I_{GH_{4}},\ L_{2})}&\mathbf{0}_{(I_{GH_{4}},\ L_{3})} & \mathbf{0}_{(I_{GH_{4}},\ L_{4})} & \centering{\Gamma_{n}^{5}} & \mathbf{0}_{(I_{GH_{4}},\ L_{6}) }& \mathbf{0}_{(I_{GH_{4}},\ L_{7}) }& \mathbf{0}_{(I_{GH_{4}},\ L_{8}) }\\
\mathbf{0}_{(I_{GH_{4}},\ L_{1})}  & \mathbf{0}_{(I_{GH_{4}},\ L_{2})}& \mathbf{0}_{(I_{GH_{4}},\ L_{3})}  & \mathbf{0}_{(I_{GH_{4}},\ L_{4})}& \mathbf{0}_{(I_{GH_{4}},\ L_{5})}& \centering{\Gamma_{o}^{6}}& \mathbf{0}_{(I_{GH_{4}},\ L_{7})}& \mathbf{0}_{(I_{GH_{4}},\ L_{8})}\\
\mathbf{0}_{(I_{H_{4}},\ L_{1})}  &\mathbf{0}_{(I_{H_{4}},\ L_{2})}& \mathbf{0}_{(I_{H_{4}},\ L_{3})}  & \mathbf{0}_{(I_{H_{4}},\ L_{4})}& \mathbf{0}_{(I_{H_{4}},\ L_{5})}& \centering{\Gamma_{p}^{6}}& \mathbf{0}_{(I_{H_{4}},\ L_{7})}& \mathbf{0}_{(I_{H_{4}},\ L_{8})}\\
\mathbf{0}_{(I_{H_{5}},\ L_{1})}&\mathbf{0}_{(I_{H_{5}},\ L_{2})}&\mathbf{0}_{(I_{H_{5}},\ L_{3})}&\mathbf{0}_{(I_{H_{5}},\ L_{4})}&\mathbf{0}_{(I_{H_{5}},\ L_{5})}&\mathbf{0}_{(I_{H_{5}},\ L_{6})}&\centering{\Gamma_t^7}&\mathbf{0}_{(I_{H_{5}},\ L_{8})}\\
\mathbf{0}_{(I_{H_{6}},\ L_{1})}&\mathbf{0}_{(I_{H_{6}},\ L_{2})}&\mathbf{0}_{(I_{H_{6}},\ L_{3})}&\mathbf{0}_{(I_{H_{6}},\ L_{4})}&\mathbf{0}_{(I_{H_{6}},\ L_{5})}&\mathbf{0}_{(I_{H_{6}},\ L_{6})}&\mathbf{0}_{(I_{H_{6}},\ L_{7})}&\Gamma_x^8
\end{array}
\right],
\end{equation*}
According to Definition \ref{positive_linearly_independence}, we assume that there exists a set of vectors $ {\rho}=\left({\rho}^{1},{\rho}^{2},{\rho}^{3},...,{\rho}^{15},{\rho}^{16}\right)$, satisfying the nonnegativity condition: 
\begin{equation}\label{nonnegativity}
    \left({\rho}^1,{\rho}^2,{\rho}^8,{\rho}^9,{\rho}^{12},{\rho}^{13}\right)\geq \mathbf{0}
\end{equation}
such that the following equality holds:
\begin{equation}\label{MPECMFCQeq}
    \mathbf{0}=\rho^{\top}\Gamma=[ S_1,\ S_2,\ S_3,\ S_4,\ S_5,\ S_6,\ S_7,\ S_8],
\end{equation}
where ${\rho}_{i}$ corresponds to i-th row of the block matrix $\Gamma$, $i=1,\cdots,16$. Consequently, we have the following relationship:
\begin{align}    S_1&=\left({\rho^{1}}\right)^{\top}\Gamma_a^1+\left({\rho^{11}}\right)^{\top}1_{\left(I_{G_4},L_1\right)}+\left({\rho^{12}}\right)^{\top}1_{\left(I_{{GH}_4},L_1\right)}=\mathbf{0},\label{S1} \\ 
S_2&=\left({\rho^{2}}\right)^{\top}\Gamma_b^2=\mathbf{0},\label{S2}\\    
S_3&=\left({\rho^{4}}\right)^{\top}\Gamma_f^3+\left({\rho^{5}}\right)^{\top}\Gamma_g^3=\mathbf{0},\label{S3}\\    S_4&=\left({\rho^{3}}\right)^{\top}\Gamma_c^4+\left({\rho^{6}}\right)^{\top}\Gamma_h^4=\mathbf{0},\label{S4}\\    S_5&=\left({\rho^{3}}\right)^{\top}\left(AB^{\top}\right)_{\left(I_{G_1},\cdot\right)}+\left({\rho^{7}}\right)^{\top}\left(BB^{\top}\right)_{\left(I_{G_3},\cdot\right)}\nonumber      +\left({\rho^{8}}\right)^{\top}\left(BB^{\top}\right)_{\left(I_{{GH}_3},\cdot\right)}   \\ &\quad+\left({\rho^{9}}\right)^{\top}\Gamma_k^5+\left({\rho^{10}}\right)^{\top}\Gamma_l^5+\left({\rho^{11}}\right)^{\top}\Gamma_m^5+\left({\rho^{12}}\right)^{\top}\Gamma_n^5 =\mathbf{0} ,\label{S5}\\ 
S_6&=\left({\rho^{7}}\right)^{\top}\Gamma_i^6+\left({\rho^{8}}\right)^{\top}\Gamma_j^6+\left({\rho^{13}}\right)^{\top}\Gamma_o^6+\left({\rho^{14}}\right)^{\top}\Gamma_p^6=\mathbf{0},\label{S6}\\    S_7&=\left({\rho^{15}}\right)^{\top}\Gamma_t^7=\mathbf{0},\label{S7}\\
    S_8&=\left({\rho^{16}}\right)^{\top}\Gamma_x^8=\mathbf{0}\label{S8}.
\end{align}

To manage the complexity of multiple multipliers and their corresponding index sets efficiently, we employ a standardized representation. Each multiplier $\rho^k$ is decomposed based on designated index sets, defined as $\rho^k = \left( \rho^k_{A}, \rho^k_{B} \right)$, where $A$ and $B$ correspond to specific components of these sets. For instance, the multiplier $\rho^7$ is partitioned as $\rho^7 = \left( \left(\rho^7\right)_{\Lambda_3^+}, \left(\rho^7\right)_{\Lambda_3^c}, \left(\rho^7\right)_{\Lambda_u} \right)$. This partitioning scheme is consistently applied across all multipliers to ensure clarity and coherence in our analysis.

Based on the above setup,  in (\ref{S3}), we examine the matrices \(\Gamma_f^3\), and \(\Gamma_g^3\). Given that these matrices exhibit distinct nonzero columns, we derive the following relations:
\begin{equation*} \left({\rho^{4}}\right)^{\top}\Gamma_f^3=\mathbf{0},\quad\left({\rho^{5}}\right)^{\top}\Gamma_g^3=\mathbf{0},
\end{equation*}
Given that the nonzero blocks in these matrices take the form of identity matrices, we conclude that \({\rho}^{4}=\mathbf{0}\) and \( {\rho}^{5}=\mathbf{0}\).
In (\ref{S4}), the matrix \(\Gamma_c^4\) and  exhibits a nonzero block in column \((\cdot, \Psi_3)\), and similarly, \(\Gamma_h^4\) exclusively contains a nonzero block in column \((\cdot, \Psi_2)\). This setup leads to the following relations:
\begin{equation*}
    \left(\rho^3\right)^{\top} I_{(\Psi_3, \Psi_3)} = \mathbf{0}, \quad \left(\rho^6\right)^{\top} I_{(\Psi_2, \Psi_2)} = \mathbf{0},
\end{equation*}
Consequently, it holds that $\rho^{3}=\mathbf{0}$ and $\ \rho^{6}=\mathbf{0}$. Using a similar technique for equation (\ref{S6}), considering the columns $(\ \cdot\ , \Lambda_u)$ and $(\ \cdot\ , \Lambda_2)$, we derive the following equations:
\begin{equation*}
    {\rho}^{7}_{\Lambda_u}=\mathbf{0}, \quad {\rho}^{14}_{\Lambda_2}=\mathbf{0}.
\end{equation*}

Similarly, for (\ref{S2}), (\ref{S7}) and (\ref{S8}), we have ${\rho}^{15}= \mathbf{0}$ and ${\rho}^{16}= \mathbf{0}$.

In conclusion, we establish the following result:
\begin{equation*}
    {\rho}^{2}=\mathbf{0}, {\rho}^{3}=\mathbf{0},{\rho}^{4}=\mathbf{0},{\rho}^{5}=\mathbf{0},{\rho}^{6}=\mathbf{0}, {\rho}^{7}_{\Lambda_u}=\mathbf{0}, {\rho}^{14}_{\Lambda_2}=\mathbf{0}, {\rho}^{15}=\mathbf{0}, {\rho}^{16}=\mathbf{0}.
\end{equation*}

Thus, the proof that the matrix \(\Gamma\) exhibits positive linear independence is reduced to considering the positive linear independence of the following matrix:

\begin{equation*}
\Tilde{\Gamma}=\left[
\small
\begin{array}{ccc}
\centering{\Gamma_{a}^{1}}&\mathbf{0}_{(I_{g_{c}},\ L_{5})}&\mathbf{0}_{(I_{g_{c}},\ \Lambda_3\cup\Lambda_1)}\\
\mathbf{0}_{(\Lambda_{3},\ L_{1})} & (B B^{\top})_{(\Lambda_{3},\ \cdot \ )} &\Tilde{\Gamma}^{6}_{{i}_{(\cdot,\ \Lambda_3\cup\Lambda_1)}} \\
\mathbf{0}_{(I_{GH_{3}},\ L_{1})} & (B B^{\top})_{(\Lambda_{1},\ \cdot \ )} & \Gamma^{6}_{{j}_{(\cdot,\ \Lambda_3\cup\Lambda_1)}}  \\
\mathbf{0}_{(I_{GH_{3}},\ L_{1})}    & \Gamma_{k}^{5}&  \mathbf{0}_{(I_{GH_{3}},\ \Lambda_3\cup\Lambda_1)} \\
\mathbf{0}_{(I_{H_{3}},\ L_{1})}   & \centering{\Gamma_{l}^{5}} & \mathbf{0}_{(I_{H_{3}},\ \Lambda_3\cup\Lambda_1)}\\
\mathbf{1}_{(I_{G_{4}},\ L_{1})} & \centering{\Gamma_{m}^{5}} & \mathbf{0}_{(I_{G_{4}},\ \Lambda_3\cup\Lambda_1) } \\
\mathbf{1}_{(I_{GH_{4}},\ L_{1})}   & \centering{\Gamma_{n}^{5}} & \mathbf{0}_{(I_{GH_{4}},\ \Lambda_3\cup\Lambda_1) }\\
\mathbf{0}_{(I_{GH_{4}},\ L_{1})}  & \mathbf{0}_{(I_{GH_{4}},\ L_{5})}& \Gamma^{6}_{{o}_{(\cdot,\ \Lambda_3\cup\Lambda_1)}}\\
\mathbf{0}_{(\Lambda_{1} \cup\Lambda^{+}_{3},\ L_{1})}  & \mathbf{0}_{(\Lambda_{1} \cup\Lambda^{+}_{3},\ L_{5})}& \Tilde{\Gamma}^{6}_{{p}_{(\cdot,\ \Lambda_3\cup\Lambda_1)}}\\

\end{array}
\right],
\end{equation*}
with 
\begin{equation*}
\begin{array}{ll}
{\Tilde{\Gamma}}_{i}^{6}:=\left[\begin{array}{cc}\mathbf{0}_{(\Lambda_{3},\ \Lambda_{1} )}& I_{(\Lambda_{3},\ \Lambda_{3} )} \end{array} \right],&
{\Tilde{\Gamma}}_{p}^{6}:=\left[\begin{array}{cc} I_{(\Lambda_{1} \cup\Lambda^{+}_{3},\ \Lambda_{1} \cup  \Lambda^{+}_{3})} & \mathbf{0}_{(\Lambda_{1} \cup\Lambda^{+}_{3},\  \Lambda^{c}_{3} ) }  \end{array} \right].
\end{array}
\end{equation*}


We have shown the multipliers of $S_3$, $S_4$, $S_7$ and $S_8$ to be zeros. The remaining equations are $S_1$, $S_5$ and $S_6$, where $S_5$ and $S_6$ are reduced as $\Tilde{S}_5$ and $\Tilde{S}_6$. We rewrite them as follows.

\begin{align}    S_1&=\left({{\rho}^{1}}\right)^{\top}\Gamma_a^1+\left({{\rho}^{11}}\right)^{\top}\mathbf{1}_{\left(I_{G_4},L_1\right)}+\left({{\rho}^{12}}\right)^{\top}\mathbf{1}_{\left(I_{{GH}_4},L_1\right)}=\mathbf{0},\label{S1_1} \\ 
\Tilde{S}_5&=\left({{\rho}^{7}}\right)_{\Lambda_3}^{\top}\left(BB^{\top}\right)_{\left(\Lambda_3,\cdot\right)}\nonumber      +\left({{\rho}^{8}}\right)^{\top}\left(BB^{\top}\right)_{\left(\Lambda_1,\cdot\right)}   \\ &+\left({{\rho}^{9}}\right)^{\top}\Gamma_k^5+\left({{\rho}^{10}}\right)^{\top}\Gamma_l^5+\left({{\rho}^{11}}\right)^{\top}\Gamma_m^5+\left({{\rho}^{12}}\right)^{\top}\Gamma_n^5 =\mathbf{0} ,\label{S2_1}\\ 
\Tilde{S}_6&=\left({{\rho}^{7}}\right)_{\Lambda_3}^{\top}{\Tilde{\Gamma}}_i^6+\left({{\rho}^{8}}\right)^{\top}\Gamma_j^6+\left({{\rho}^{13}}\right)^{\top}\Gamma_o^6+\left({{\rho}^{14}}\right)_{\Lambda_1\cup\Lambda_3^+}^{\top}{\Tilde{\Gamma}}_p^6=\mathbf{0}.\label{S3_1}
\end{align}

In the context of the columnar structure of equation (\ref{S3_1}), the matrix $\Tilde{S}_6$ is partitioned as follows: $\Tilde{S}_6 = \left[\left(\Tilde{S}_6\right)_{(\cdot,\Lambda_3^+)},\ \left(\Tilde{S}_6\right)_{(\cdot,\Lambda_3^c)},\ \left(\Tilde{S}_6\right)_{(\cdot,\Lambda_1)}\right]$. Consequently, we establish:

\begin{align}
&\left(\Tilde{S}_6\right)_{(\cdot,\Lambda_3^+)}=  \left( {\rho}^{7}\right)_{\Lambda_3^+}^{\top} + \left( {\rho}^{14}\right)_{\Lambda_3^+}^{\top} = \mathbf{0},\label{S3_1_1}\\
    &\left(\Tilde{S}_6\right)_{(\cdot,\Lambda_3^c)}= \left( {\rho}^{7}\right)_{\Lambda_3^c}^{\top} + \left({\rho}^{13}\right)^{\top} = \mathbf{0},\label{S3_1_2}\\
    &\left(\Tilde{S}_6\right)_{(\cdot,\Lambda_1)}=  \left({\rho}^{8}\right)^{\top} + \left( {\rho}^{14}\right)_{\Lambda_1}^{\top} = \mathbf{0}.\label{S3_1_3}
\end{align}

Due to the nonnegativity of ${\rho}^{13}$, it follows that
\begin{equation}
    \left( {\rho}^{7}\right)_{\Lambda_3^c}\leq\mathbf{0}.\label{nonnegativity2}
\end{equation}

Similarly, for the matrix $\Tilde{S}_5$, we establish the following partitioning: $$\Tilde{S}_5 = \left[\left(\Tilde{S}_5\right)_{(\cdot,\Lambda_3^+)},\ \left(\Tilde{S}_5\right)_{(\cdot,\Lambda_3^c)},\ \left(\Tilde{S}_5\right)_{(\cdot,\Lambda_1)},\ \left(\Tilde{S}_5\right)_{(\cdot,\Lambda_2)},\ \left(\Tilde{S}_5\right)_{(\cdot,\Lambda_u)}\right].$$ This configuration leads us to the following derivations:

\begin{align}
    &\left(\Tilde{S}_5\right)_{(\cdot,\Lambda_3^+)}= \left( {\rho}^{7}\right)_{\Lambda_3}^{\top}\left(B B^{\top}\right)_{(\Lambda_3, \Lambda_3^+)} + \left( {\rho}^{8}\right)^{\top}\left(B B^{\top}\right)_{(\Lambda_1, \Lambda_3^+)} = \mathbf{0},\label{S_2_1_1}\\
    &\left(\Tilde{S}_5\right)_{(\cdot,\Lambda_3^c)}= \left( {\rho}^{7}\right)_{\Lambda_3}^{\top}\left(B B^{\top}\right)_{(\Lambda_3, \Lambda_3^c)} + \left( {\rho}^{8}\right)^{\top}\left(B B^{\top}\right)_{(\Lambda_1, \Lambda_3^c)} - \left( {\rho}^{12} \right)^{\top} = \mathbf{0},\label{S_2_1_2}\\
    &\left(\Tilde{S}_5\right)_{(\cdot,\Lambda_1)}=  \left( {\rho}^{7}\right)_{\Lambda_3}^{\top}\left(B B^{\top}\right)_{(\Lambda_3, \Lambda_1)} + \left( {\rho}^{8}\right)^{\top}\left(B B^{\top}\right)_{(\Lambda_1, \Lambda_1)} + \left( {\rho}^{9} \right)^{\top} = \mathbf{0},\label{S_2_1_3}\\
    &\left(\Tilde{S}_5\right)_{(\cdot,\Lambda_2)}=  \left( {\rho}^{7}\right)_{\Lambda_3}^{\top}\left(B B^{\top}\right)_{(\Lambda_3, \Lambda_2)} + \left( {\rho}^{8}\right)^{\top}\left(B B^{\top}\right)_{(\Lambda_1, \Lambda_2)} + \left( {\rho}^{10} \right)^{\top} = \mathbf{0},\label{S_2_1_4}\\
    &\left(\Tilde{S}_5\right)_{(\cdot,\Lambda_u)}=  \left( {\rho}^{7}\right)_{\Lambda_3}^{\top}\left(B B^{\top}\right)_{(\Lambda_3, \Lambda_u)} + \left( {\rho}^{8}\right)^{\top}\left(B B^{\top}\right)_{(\Lambda_1, \Lambda_u)} - \left( {\rho}^{11} \right)^{\top} = \mathbf{0}\label{S_2_1_5}.
\end{align}

Now, let us consider the column corresponding to $\left[\left(\Tilde{S}_5\right)_{(\cdot,\Lambda_3^+)},\ \left(\Tilde{S}_5\right)_{(\cdot,\Lambda_3^c)},\ \left(\Tilde{S}_5\right)_{(\cdot,\Lambda_1)}\right]$. From (\ref{S_2_1_1})-(\ref{S_2_1_3}), we have the following results:
\begin{align}
    \left[\mathbf{0}_{\mid\Lambda_3^+\mid},\mathbf{0}_{\mid\Lambda_3^c\mid}, \mathbf{0}_{\mid\Lambda_1\mid}\right]&= \left[\left({\rho}^{7}\right)_{\Lambda_3}^{\top}, \left({\rho}^{8}\right)^{\top}\right] \left[\begin{array}{ccc}
         \left(B B^{\top}\right)_{(\Lambda_3, \Lambda_3^+)}& \left(B B^{\top}\right)_{(\Lambda_3, \Lambda_3^c)}& \left(B B^{\top}\right)_{(\Lambda_3, \Lambda_1)} \\
         \left(B B^{\top}\right)_{(\Lambda_1, \Lambda_3^+)}& 
         \left(B B^{\top}\right)_{(\Lambda_1, \Lambda_3^c)}&
         \left(B B^{\top}\right)_{(\Lambda_1, \Lambda_1)}
    \end{array}\right]\nonumber\\
    & \quad \quad+\left[ \mathbf{0}_{\mid\Lambda_3^+\mid}, - \left( {\rho}^{12} \right)^{\top}, \left( {\rho}^{9} \right)^{\top}\right],\\
    & = \left[\left({\rho}^{7}\right)_{\Lambda_3^+}^{\top}, \left({\rho}^{7}\right)_{\Lambda_3^c}^{\top}, \left({\rho}^{8}\right)^{\top}\right] \left[\begin{array}{ccc}
         \left(B B^{\top}\right)_{(\Lambda_3^+, \Lambda_3^+)}& \left(B B^{\top}\right)_{(\Lambda_3^+, \Lambda_3^c)}& \left(B B^{\top}\right)_{(\Lambda_3^+, \Lambda_1)} \\
         \left(B B^{\top}\right)_{(\Lambda_3^c, \Lambda_3^+)}& \left(B B^{\top}\right)_{(\Lambda_3^c, \Lambda_3^c)}& \left(B B^{\top}\right)_{(\Lambda_3^c, \Lambda_1)} \\
         \left(B B^{\top}\right)_{(\Lambda_1, \Lambda_3^+)}& 
         \left(B B^{\top}\right)_{(\Lambda_1, \Lambda_3^c)}&
         \left(B B^{\top}\right)_{(\Lambda_1, \Lambda_1)}
    \end{array}\right]\nonumber\\
    & \quad \quad+\left[ \mathbf{0}_{\mid\Lambda_3^+\mid}, - \left( {\rho}^{12} \right)^{\top}, \left( {\rho}^{9} \right)^{\top}\right].\label{Lambda1_3}
\end{align}
The equality in (\ref{Lambda1_3}) holds due to the decomposition $\Lambda_3=\left(\Lambda_3^+,\Lambda_3^c \right)$.
Let ${\overline{\rho}}^{\top}$ denote $\left[\left({\rho}^{7}\right)^{\top}_{\Lambda_3^+},\left({\rho}^{7}\right)^{\top}_{\Lambda_3^c},\left({\rho}^{8}\right)^{\top}\right]^{\top}$. 
Upon right-multiplying equation (\ref{Lambda1_3}) by  ${\overline{\rho}}^{\top}$ , we obtain
\begin{align}
    0 &= \left[\left({\rho}^{7}\right)^{\top}_{\Lambda_3^+},\left({\rho}^{7}\right)^{\top}_{\Lambda_3^c},\left({\rho}^{8}\right)^{\top}\right] \left( BB^{\top}\right)_{(\Lambda_1\cup\Lambda_3,\Lambda_1\cup\Lambda_3)} \left[\begin{array}{c}
         {\left({\rho}^{7}\right)}_{\Lambda_3^+}  \\
          {\left({\rho}^{7}\right)}_{\Lambda_3^c}\\
          {\rho}^{8}
    \end{array}\right]+\left[ \mathbf{0}_{\mid\Lambda_3^+\mid}, - \left( {\rho}^{12} \right)^{\top}, \left( {\rho}^{9} \right)^{\top}\right] \left[\begin{array}{c}
         {\left({\rho}^{7}\right)}_{\Lambda_3^+}  \\
          {\left({\rho}^{7}\right)}_{\Lambda_3^c}\\
          {\rho}^{8}
    \end{array}\right]\\
    &= \overline{\rho}^{\top} \left( BB^{\top}\right)_{(\Lambda_1\cup\Lambda_3,\Lambda_1\cup\Lambda_3)} \overline{\rho}- \left( {\rho}^{12} \right)^{\top} {\left({\rho}^{7}\right)}_{\Lambda_3^c}+ \left( {\rho}^{9} \right)^{\top}{\rho}^{8}\geq 0.\label{positiveterm}
\end{align}
The first term is non-negative due to positive-definiteness of $\left( BB^{\top}\right)_{(\Lambda_1\cup\Lambda_3,\Lambda_1\cup\Lambda_3)}$. The second and third term is non-negative because of non-negativity of $\rho^8$, $\rho^9$ and $\rho^{12}$ in (\ref{nonnegativity}) and non-positivity of  $\left({\rho}^{7}\right)_{\Lambda_3^c}$ in (\ref{nonnegativity2}). Altogether it follows that (\ref{positiveterm}) is non-negative  and the equality holds if and only if $\overline{\rho}=\mathbf{0}$ is valid, and thus we have
\begin{equation}\label{multiplier56}
    {\rho}^{7}=\mathbf{0}, {\rho}^{8}=\mathbf{0}.
\end{equation}
Upon inserting (\ref{multiplier56}) into (\ref{S3_1_1})-(\ref{S3_1_3}) and (\ref{S_2_1_2})-(\ref{S_2_1_5}), we get the remaining multipliers must be zeros. Consequently, $\Gamma$ is positive-linearly independent, thereby validating the MPEC-MFCQ condition.

\end{proof}

\section{The proof of Theorem \ref{thrm5}}\label{appendixb}
Before proceeding with the proof, we first introduce the required notations.
    Let $v$ be a feasible point for the MPEC in (\ref{MPEC}), we define the index sets as:
    \begin{align*}
        &J_{00}(v)=\{ i\in\mathbb{R}^{\overline{n}+\overline{m}}\mid G_i(v)=0,\ H_i(v)=0\},\\
        &J_{0+}(v)=\{ i\in\mathbb{R}^{\overline{n}+\overline{m}}\mid G_i(v)=0,\ H_i(v)>0\},\\
        &J_{+0}(v)=\{ i\in\mathbb{R}^{\overline{n}+\overline{m}}\mid G_i(v)>0,\ H_i(v)=0\}.
    \end{align*}
    We denote by $\mathrm{supp}(z):=\{\,i \mid z_i \neq 0\,\}$ the support of a vector $z$.

  With these notations established, we are now prepared to present the proof.
    \begin{proof}
Recall the detailed form of the stopping criteria (\ref{innerap}),
    \begin{align}
        & \Vert \nabla f(v^k)-\sum_{i=1}^{2\overline{n}} \lambda_i^{g,k} \nabla g_i(v^k)-\sum_{i=1}^{\overline{m}} \lambda_i^{G,k} \nabla G_i(v^k)-\sum_{i=1}^{\overline{m}} \lambda_i^{H,k} \nabla H_i(v^k)\\
        &+\sum_{i=1}^{\overline{m}} \lambda_i^{GH,k} H_i(v^k)\nabla G_i(v^k)+\sum_{i=1}^{\overline{m}} \lambda_i^{GH,k} G_i(v^k)\nabla H_i(v^k)\Vert_{\infty}\leq\epsilon_k,\label{30a}\\
        & \min\{\lambda^{g,k},g(v^k)\}\leq\epsilon_k,\label{30b}\\
        &\min\{\lambda^{G,k},G(v^k)\}\leq\epsilon_k,\label{30c}\\
        &\min\{\lambda^{H,k},H(v^k)\}\leq\epsilon_k,\label{30d}\\
        &\min\{\lambda^{GH,k},u\}\leq\epsilon_k,\label{30e}\\
        &\Vert u^{k}_i-\lbrack {\tau}_k- G_i(v^k)H_i(v^k)\rbrack\Vert_{\infty}\leq\epsilon_k,i=1,\cdots,\overline{m},\label{30f}\\
        & \lambda^k\geq0,u^{k}\geq0,g(v^k)\geq0,G(v^k)\geq0,H(v^k)\geq0. \label{30g}
    \end{align}
  Taking the limit as $k \rightarrow \infty$, we assume that the sequence $\{ v^k, \lambda^k, u^k \}$ converges to $(v^*, \lambda^*, u^*)$. Consequently, we have $u_i^{*} = -G_i(v^*)H_i(v^*) \geq 0$ for all $i = 1,\ldots,\overline{m}$, as well as $G(v^*) \geq 0$ and $H(v^*) \geq 0$. It follows directly that
\[
G_i(v^*)H_i(v^*) = 0,\quad i = 1,\cdots,\overline{m},
\]
indicating that the limit $v^*$ is feasible for the MPEC. Next, we define the multipliers as follows.

$$
           {\delta_i}^{G,k} :=
            \begin{cases}
                {\lambda_i}^{GH,k} H_i(v^k), &i\in J_{00}(v^*)\cup J_{0+}(v^*),\\
                 0,&\text{otherwise},
            \end{cases}
            \quad\quad
            {\delta_i}^{H,k}:=
            \begin{cases}
                 {\lambda_i}^{GH,k} G_i(v^k),&i\in J_{00}(v^*)\cup J_{+0}(v^*),\\
                 0,&\text{otherwise}.
            \end{cases}
$$
Then it holds that
\begin{align*}
    &\Vert \nabla f(v^k)-\sum_{i=1}^{2\overline{n}} \lambda_i^{g,k} \nabla g_i(v^k)-\sum_{i=1}^{\overline{m}} \lambda_i^{G,k} \nabla G_i(v^k)-\sum_{i=1}^{\overline{m}} \lambda_i^{H,k} \nabla H_i(v^k)+\sum_{i=1}^{\overline{m}} \delta_i^{G,k}\nabla G_i(v^k)\\
    &+\sum_{i=1}^{\overline{m}} \delta_i^{H,k}\nabla H_i(v^k) +\sum_{i\in I_{+0}(v^*)} \lambda_i^{GH,k} H_i(v^k)\nabla G_i(v^k) \\
    &+\sum_{i\in I_{0+}(v^*)} \lambda_i^{GH,k} G_i(v^k)\nabla H_i(v^k)
    \Vert_{\infty}\leq\epsilon_k.
\end{align*}
We claim that the multipliers $\left(\lambda^{g,k},\lambda^{G,k},\lambda^{H,k},\delta^{G,k},\delta^{H,k},\lambda_{I_{+0}\cup I_{0+}}^{GH,k}\right)$ are bounded. Otherwise, we could assume without loss of generality convergence of the sequence 
\begin{equation*}
    \frac{\left(\lambda^{g,k},\lambda^{G,k},\lambda^{H,k},\delta^{G,k},\delta^{H,k},\lambda_{I_{+0}\cup I_{0+}}^{GH,k}\right)}{\lVert \left(\lambda^{g,k},\lambda^{G,k},\lambda^{H,k},\delta^{G,k},\delta^{H,k},\lambda_{I_{+0}\cup I_{0+}}^{GH,k}\right)\rVert}\rightarrow \left(\bar{\lambda}^{g},\bar{\lambda}^{G},\bar{\lambda}^{H},\bar{\delta}^{G},\bar{\delta}^{H},\bar{\lambda}_{I_{+0}\cup I_{0+}}^{GH}\right)\neq 0.
\end{equation*}
Then the $\epsilon_k$-stationary of the $v^k$ yields
\begin{equation*}
    -\sum_{i=1}^{2\overline{n}}\bar{\lambda} _i^{g} \nabla g_i(v^*)-\sum_{i=1}^{\overline{m}}\bar{\lambda} _i^{G} \nabla G_i(v^*)-\sum_{i=1}^{\overline{m}}\bar{\lambda} _i^{H} \nabla H_i(v^*)+\sum_{i=1}^{\overline{m}}\bar{\delta} _i^{G}\nabla G_i(v^*)+\sum_{i=1}^{\overline{m}}\bar{\delta} _i^{H}\nabla H_i(v^*)=0,
\end{equation*}
where we take into account that $H_i(v^k)\rightarrow 0,\ i\in J_{+0}(v^*)$ and $G_i(v^k)\rightarrow 0,\ i\in J_{0+}(v^*)$. We know that
$\bar{\lambda}_i^g\geq0$ for $i=1,\cdots,2\overline{n}$. If $\bar{\lambda}_i^g>0$, we have $\lambda_i^{g,k}>c$ for some constant $c>0$ for all k sufficiently large. Without loss of generality, we assume $\epsilon_k<c$ for all k sufficiently large. This yields
$min\left\{\lambda_i^{g,k},g_i(v^k)\right\}=g_i(v^k)\rightarrow0\ (k\rightarrow\infty)$, and thus $i\in I_g(v^*)$. Similarly, we have $\bar{\lambda}_i^G\geq0$  for $i=1,\cdots,\overline{m}$ and $\bar{\lambda}_i^G>0$ implies $\Vert G(v^k)\Vert_{\infty}\rightarrow0\ (k\rightarrow\infty)$ and thus $i\in J_{00}(v^*)\cup J_{0+}(v^*)$; also we have $\bar{\lambda}_i^H\geq0$  for $i=1,\cdots,\overline{m}$ and $\bar{\lambda}_i^H>0$ implies $\Vert H(v^k)\Vert_{\infty}\rightarrow0 \ (k\rightarrow\infty)$ and thus $i\in J_{00}(v^*)\cup J_{+0}(v^*)$.\\
If $\left(\bar{\lambda}^{g},\bar{\lambda}^{G},\bar{\lambda}^{H},\bar{\delta}^{G},\bar{\delta}^{H}\right)\neq0$, it yields a contradiction to the definition of MPEC-MFCQ. If on the other hand, $\left(\bar{\lambda}^{g},\bar{\lambda}^{G},\bar{\lambda}^{H},\bar{\delta}^{G},\bar{\delta}^{H}\right)=0$, there has to be an index $i\in J_{+0}(v^*)\cup J_{0+}(v^*)$ with $\bar{\lambda}_i^{GH}\neq0$. First consider the case $i\in J_{0+}(v^*)$, as we assumed the multipliers to be unbounded. Then by definition,
\begin{equation*}
    \bar{\lambda}_i^G=\lim_{k\rightarrow\infty}=\frac{\lambda_i^{GH,k}H_i(v^k)}{\Vert \left(\lambda^{g,k},\lambda^{G,k},\lambda^{H,k},\delta^{G,k},\delta^{H,k},\lambda_{J_{+0}\cup J_{0+}}^{GH,k}\right)\Vert}=\bar{\lambda}_i^{GH} H_i(v^*)\neq 0,
\end{equation*}
a contradiction to the assumption $\bar{\lambda}_i^G=0$. The case $i\in J_{+0}(v^*)$ could be treated in a similar way. Therefore, the sequence $\left(\lambda^{g,k},\lambda^{G,k},\lambda^{H,k},\delta^{G,k},\delta^{H,k},\lambda_{J_{+0}\cup J_{0+}}^{GH,k}\right)$ is bounded and converges to some limit $\left(\lambda^{g,*},\lambda^{G,*},\lambda^{H,*},\delta^{G,*},\delta^{H,*},\lambda_{J_{+0}\cup J_{0+}}^{GH,*}\right)$. It is easy to see that
\begin{align*}
    & supp\left(\lambda^{g,*}\right)\subseteq I_g\left(v^*\right),\\
    & supp\left(\lambda^{G,*}\right)\cup supp\left(\delta^{G,*}\right)\subseteq J_{00}\left(v^*\right)\cup J_{0+}\left(v^*\right),\\
    & supp\left(\lambda^{H,*}\right)\cup supp\left(\delta^{H,*}\right)\subseteq J_{00}\left(v^*\right)\cup J_{+0}\left(v^*\right).
\end{align*}
Let us define the multipliers $\gamma^*:=\lambda^{G,*}-\delta^{G,*}$ and $\mu^*:=\lambda^{H,*}-\delta^{H,*}$. Then $v^*$ together with the multipliers $\left(\lambda^{g,*},\gamma^*,\mu^*\right)$ is a weakly stationary point of the MPEC.



In order to establish the C-stationarity of $v^*$, suppose for contradiction that there exists an index $i \in J_{00}(v^*)$ such that $\gamma^*\mu^* < 0$. Without loss of generality, we consider only the case $\gamma^* < 0,\ \mu^* > 0$, as the alternative scenario can be treated analogously. Given $\gamma^* < 0$, it follows directly that $\delta_i^{G,*} > 0$. Since $i \in J_{00}(v^*)$ and $\delta_i^{G,*} = \lim_{k\rightarrow\infty}\delta_i^{G,k} = \lim_{k\rightarrow\infty}\lambda_i^{GH,k}H_i(v^k)$, we must have $\lambda_i^{GH,k} \rightarrow +\infty$ as $k \rightarrow \infty$.

Considering further the facts $u_i^k < \epsilon_k$ and (\ref{30f}), we derive the inequality
\begin{equation}\label{39}
    0<\frac{\mid G_i(v^k)H_i(v^k)-{\tau}_k\mid}{\sqrt{\epsilon_k}}\leq\frac{2\epsilon_k}{\sqrt{\epsilon_k}}\leq 2\sqrt{\epsilon_k}\rightarrow0\ (k\rightarrow\infty).
\end{equation}

Now, if $\lambda_i^{H,*} > 0$, then, using the inequality $\|H_i(v^k)\|_{\infty}\leq \epsilon_k$, we obtain
\begin{equation*}
    0\leq\frac{\mid H_i(v^k)\mid}{\sqrt{\epsilon_k}}\leq \sqrt{\epsilon_k}\ \rightarrow0 \ (k\rightarrow\infty).
\end{equation*}

Combining this result with (\ref{39}) implies that
\[
\frac{{\tau}_k}{\sqrt{\epsilon_k}}\rightarrow 0,\quad \text{as}\ k\rightarrow\infty,
\]
contradicting the condition $\epsilon_k = O({\tau}_k^2)$. Consequently, we must have $\lambda_i^{H,*}=0$, which implies $\delta_i^{H,*} < 0$, contradicting the definition of $\delta_i^{H,*}$. Thus, we conclude that $v^*$, together with the multipliers $(\lambda^{g,*}, \gamma^*, \mu^*)$, is a C-stationary point of the MPEC.

\end{proof}

\end{appendices}
\end{document}